\documentclass[preprint,compress,3p,12pt]{elsarticle}
\usepackage{amsfonts}
\usepackage{amsthm}
\usepackage{mathrsfs}
\usepackage{amsmath}
\usepackage{amscd}
\usepackage{amssymb}
\usepackage{graphicx}
\usepackage{epstopdf}
\usepackage{subfigure}
\usepackage{array}
\usepackage{bm}
\usepackage{subfigure}
\usepackage{color}
\usepackage{booktabs}
\renewcommand{\vec}[1]{\mbox{\boldmath \small $#1$}}
 \usepackage{setspace}

\biboptions{numbers,sort&compress}

\allowdisplaybreaks \numberwithin{equation}{section}

\theoremstyle{plain}

\newtheorem{theorem}{Theorem}[section]
\newtheorem{lemma}[theorem]{Lemma}
\newtheorem{corollary}[theorem]{Corollary}
\newtheorem{remark}{Remark}[section]

\theoremstyle{definition}
\newtheorem{example}{Example}[section]
\newtheorem{definition}{Definition}[section]

\usepackage{geometry}
\usepackage[
                  bookmarksnumbered=true,
                  bookmarksopen=true,
                  colorlinks,
                  pdfborder=001,
                  linkcolor=green,
                  anchorcolor=green,
                  citecolor=green
                  ]{hyperref}
\journal{Elsevier}
\begin{document}

\begin{frontmatter}

\title{A  general class of linear unconditionally energy stable schemes for the gradient flows
}


\author{Zengqiang Tan}
\ead{tzengqiang@163.com}
\address{Center for Applied Physics and Technology, HEDPS and LMAM,
				School of Mathematical Sciences, Peking University, Beijing 100871, P.R. China}
\author{Huazhong Tang\corref{mycorrespondingauthor}}
\cortext[cor1]{Corresponding author. Fax:~+86-10-62751801.}\ead{hztang@pku.edu.cn}
			\address{Nanchang Hangkong University, Jiangxi Province, Nanchang 330000, P.R. China;  Center for Applied Physics and Technology, HEDPS and LMAM,
				School of Mathematical Sciences, Peking University, Beijing 100871, P.R. China}

\begin{abstract}
 This paper studies a    class of linear unconditionally energy stable schemes for the gradient flows.
Such schemes are built on the SAV technique and 
the  general linear  time discretization (GLTD) as well as the linearization based on the extrapolation for the nonlinear term, and
 may be  arbitrarily high-order accurate and very general,   containing many  existing   SAV schemes and  new SAV schemes.
It is shown that the  semi-discrete-in-time schemes  are  unconditionally energy stable
 when the GLTD is algebraically stable, and
 are convergent  with the order of $\min\{\hat{q},\nu\}$   under the diagonal stability and some suitable regularity and    accurate  starting values, where $\hat{q}$ is the generalized stage order of the GLTD and $\nu$ denotes the number of the extrapolation points in time.
The energy stability results can be easily extended to the  fully discrete schemes, for example, if the Fourier spectral method is employed in space when  the periodic boundary conditions are specified. Some numerical experiments on the Allen-Cahn, Cahn-Hilliard, and phase field crystal    models 
are conducted to validate those theories as well as the effectiveness, the energy stability and the accuracy of our schemes. 
\end{abstract}

\begin{keyword}
\texttt{}Gradient flows\sep scalar auxiliary variable \sep general linear time
discretization \sep  energy stability\sep convergence.
\end{keyword}

\end{frontmatter}

\section{Introduction}\label{intro}

Many physical problems arisen in science and engineering can be modeled by partial differential equations in the form of gradient flows, for example, the interface dynamics \cite{Anderson98,Gurtin96,Yue04}, the crystal growth \cite{Braun97,Kobayashi93,Tong01}, the tumor growth \cite{Oden10,Wise08}, the thin film \cite{Wang03,Karma98}, the polymers \cite{Fraaije03,Fraaije93,Maurits97}, the solidification \cite{Boettinger02,Karma96,Wang93} and so on.
The gradient flows are dynamics  determined by not only the driving free energy, but also the dissipation mechanism.
For a given free energy   $\mathcal{F}(u)$,  the gradient flows can be written as
\begin{align}  \label{1.1}
 \frac{\partial u}{\partial t} = \mathcal{G} \mu,
\end{align}
supplemented with suitable boundary conditions (e.g. with periodic or homogeneous Neumann boundary conditions) \cite{ShenJ19},
where   $\mu = \delta \mathcal{F}/\delta u$ denotes the variational derivative of $\mathcal{F}(u)$
and
 the operator $\mathcal{G}$ is     non-positive symmetric so that
 the free energy is monotonically decreasing
\begin{align}  \label{1.2}
\frac{d\mathcal{F}(u)}{dt}  = \left(\frac{\delta \mathcal{F}}{\delta u}, \frac{\partial u}{\partial t} \right) = (\mu, \mathcal{G} \mu) \le 0,
\end{align}
here $(\cdot,\cdot)$ denotes the inner product in $L^2(\Omega)$. 
If $\mathcal{G} = -I$ (resp. $\Delta$),  then one has the so-called $L^2$ (resp. $H^{-1}$) gradient flow.

Most of gradient flow equations are nonlinear so that it is hard to obtain their analytical solutions. Hence, studying them numerically is the primary approach.
Recently, designing efficient and energy stable numerical schemes for  the gradient flows has attracted much attention.
There exist  several efficient and popular  techniques to design energy stable schemes for the gradient flows.
The first is the convex splitting \cite{Elliott93,Eyre98}.
Based on it, one can design unconditionally energy stable and uniquely solvable schemes, but should solve a nonlinear system at each time step generally. Although the convex splitting technique has been developed case-by-case for some problems  \cite{Baskaran13,ShenJ12,WuK14}, it is not available to give an  unified formulation.
The second technique  is the stabilization  \cite{Tang06,ShenJ10b,WangL18}, which treats the nonlinear terms explicitly and adds a stabilization term to relax the time step constraint. It is  {simple} and efficient since the linear equations with constant coefficients are
solved at each time step. However, it is very challenging to design high-order unconditionally energy stable  {schemes}. Some  progresses can be found  in \cite{LiD17}.
The third technique is the invariant energy quadratization (IEQ) \cite{YangX16,YangX17a,YangX17b,YangX20}. It allows one to construct the linear and second-order unconditionally energy stable schemes for a large class of the gradient flows, but needs to solve the linear equations with  variable coefficients, and requires that the free energy density is bounded from below so that  its applications are limited for some physically interesting models, such as the molecular beam epitaxial (MBE) model without slope selection \cite{LiB03}.
The fourth technique is the  scalar auxiliary variable (SAV) \cite{ShenJ18a,ShenJ18b,ShenJ19}.
With the help of introducing some SAVs,
 the gradient flow model is reformulated into an equivalent form,
 then  some linear and unconditionally energy stable schemes can be developed by
approximating the reformulated system instead of the original gradient flow model.
It is convenient to construct second-order or higher-order unconditionally energy stable SAV schemes, which  need to solve several linear systems with constant coefficients at each time step.
In addition to the above, there are some other interesting techniques,
including but not limited to the exponential time differencing (ETD) \cite{WangX16,Du19},
the Lagrange multiplier \cite{Badia11,Guillen13}, the energy factorization \cite{KouJ20,WangX20},  and the averaged vector field \cite{Hairer-Lubich2014}  etc. 

Up to now, the SAV technique may result in more robust schemes with less restrictions on the energy functionals 
  and has been successfully applied to many existing gradient flow models  \cite{Akrivis19,Cheng18,Cheng19,Gong20,HouD19,HuangF20,ZhangC20,LiuZ20,LiuZ21,YangZ19,YangJ21}.
However, in those existing works, the time discretizations  are the backward Euler,  the Crank-Nicolson (CN), or the second-order backward differentiation formula (BDF2), except
 the high-order SAV-RK (Runge-Kutta) \cite{Akrivis19,Gong20}.
The aim of this paper  is to study a general class of  arbitrarily high-order linear unconditionally energy stable schemes  for the gradient flows.
Such schemes, abbreviated as the SAV-GL schemes  below for convenience, are based on the SAV technique and the general linear  time
discretization (GLTD) as well as the linearization based on the extrapolation for the nonlinear term.
It is worth noting that our studied SAV-GL schemes contain most of the time integration schemes (e.g. SAV-BDF1, SAV-CN, SAV-BDF2, SAV-RK etc.) for the gradient flows in literature and many  new schemes.
The GLTDs, as multistage multivalue  schemes  proposed in 
 \cite{Burrage80}, can be considered as a natural generalization of the  Runge-Kutta (RK) and linear multistep time discretizations,
 and are  flexible in developing numerical methods with  better stability and accuracy, except slightly complexity.
 Two special examples provided later are   the so-called one-leg \cite{Dahlquist75B,Dahlquist78} and   multistep Runge-Kutta (MRK) \cite{Burrage87,Li99} time integration  schemes.

%

The rest of this paper is organized as follows. Section \ref{sec:2} briefly reviews
the GLTDs for the ordinary differential equations (ODEs) and  extends them as a  general class of semi-discrete-in-time linear schemes  for  the gradient flow equations by using the SAV
and  the
linearization based on extrapolation.
Specifically, the gradient flow equation is first changed as an equivalent form by using
 the SAV, and then    the reformulated equation is approximated by the   GLTDs,
 in which the linear part is implicitly discretized
 while the nonlinear part is explicitly and linearly dealt with the   extrapolation.
The energy stability  and  the convergence of the SAV-GL schemes are addressed in Sections \ref{sec:3} and   \ref{sec:4}, respectively.
 A discrete energy dissipation law  is obtained    in the sense of the weighted inner product
 and the norm if the GLTDs are algebraically stable,
and the convergence order in time of  $\min\{\hat{q},\nu\}$ is  derived under some suitable regularity and    accurate  starting values, if the GLTDs are diagonal stable, where $\hat{q}$ is the generalized stage order of the {GLTDs} and $\nu$ denotes the number of the extrapolation points in time.
%
Section \ref{sec:5} introduces the Fourier spectral discretization  for the gradient flows with the periodic boundary conditions in order to conduct our numerical validation.  Section  \ref{sec:6} numerically {tests} the fully discrete SAV-GL schemes against  three widely concerned gradient flow equations
(which are the Allen-Cahn, Cahn-Hilliard, and phase field crystal   models)
in order to validate the effectiveness, energy stability and accuracy of our schemes. Some concluding remarks are given in Section \ref{sec:7}.

\section{SAV-GL schemes for gradient flows}  \label{sec:2}

This section briefly reviews the GLTDs \cite{Burrage80,ButcherAN2006,Jackiewicz09} for the ODEs and extends them to   the gradient flows by using the SAV technique \cite{ShenJ18a,ShenJ18b,ShenJ19} as a general class of linear numerical schemes, which will be abbreviated as ``SAV-GL'' below for convenience.

\subsection{A brief review to {GLTDs}}  \label{sec:2.1}
For the first-order  ODE
\begin{align}   \label{2.1.1}
\frac{d u}{dt}=u'(t)  = f(u(t)),~~ t\in(0,T],
\end{align}
subject to the initial data $u(0) = u_0$,
the GLTD is a large family of multistage multivalue  schemes for ODEs, which includes
the linear multistep,  predictor-corrector and Runge-Kutta
 schemes as special cases.
Many peoples have tried their best to  search for the useful GLTDs which do not exist within
the standard special cases, but possess as many of the advantages and as few of the disadvantages as possible. The readers are referred to the review paper \cite{ButcherAN2006} and the monograph \cite{Jackiewicz09} as well as  references therein.


Assume that the time interval $[0,T]$ is divided into $K$ equal parts
with   the time stepsize $\tau \!=\! \frac{T}{K}$,  $K\!\in\!\mathbb{Z}^+$.
The GLTDs
can be defined by 
\begin{align}   \label{2.1.2}
\begin{cases}
~~ U_{n,i}= \tau\sum \limits_{j=1}^s d_{ij}^{11} f(U_{n,i}) 
                  + \sum\limits_{j=1}^r d_{ij}^{12} u_j^{[n]},~~~ i =1,2,\cdots,s, \\[1 \jot]
u_i^{[n+1]} = \tau\sum\limits_{j=1}^s d_{ij}^{21} f(U_{n,i}) 
               + \sum\limits_{j=1}^r d_{ij}^{22} u_j^{[n]},~~~ i =1,2,\cdots,r,
\end{cases}
\end{align}
where $d_{ij}^{\imath\jmath}\in\mathbb R$, $\imath,\jmath = 1,2$, 
$U_{n,i}$ is an approximation of
stage order $q$ to $u(t_n +c_i \tau)$, $c_i\in\mathbb R$, $i=1,\cdots,s$, and
each of   $r$ import quantities $\{u_i^{[n]}\}$  is
an approximation of {order $p\ge q$} to the linear combination of the scaled derivatives   of the solution $u$ to
\eqref{2.1.1} at  $t_n$, i.e.
$$u_i^{[n]}=\sum\limits_{j=0}^p w_{ij}\tau^j u^{(j)}(t_n)+{\mathcal O}(h^{p+1}),\ i=1,\cdots,r,\
$$
 with   some scalars $w_{ij}$ and $u^{(j)}=\frac{d^j u}{dt^j}$.
Such time discretizations are characterized by   four integers $p,q, r,s$ (being respectively the method order,  the stage order,  the number of external approximations, and
the number of stages or internal approximations),
the abscissa vector $\vec{c} = (c_1,c_2,\cdots,c_s)^T \in \mathbb{R}^s$, the vectors $\vec{w}_j = (w_{1j},w_{2j},\ldots,w_{rj})^T \in \mathbb{R}^r$, $j = 0,1,\ldots,p$, and four coefficient matrices $\vec{D}_{\imath\jmath} = ( d_{ij}^{\imath\jmath})$ for $\imath,\jmath = 1,2$,
where  $\vec{D}_{11}\in\mathbb{R}^{s\times s}, \vec{D}_{12}\in\mathbb{R}^{s\times r}, \vec{D}_{21}\in\mathbb{R}^{r\times s}$ and $\vec{D}_{22}\in\mathbb{R}^{r\times r}$.
 Obviously, the GLTDs
  \eqref{2.1.2} are consistent with \eqref{2.1.1}  if
\begin{align*}    
\vec{D}_{21}\vec{e} + \vec{D}_{22}\vec{w}_1 = \vec{w}_0 + \vec{w}_1,
\ \
\vec{D}_{12}\vec{w}_0 = \vec{e},\ \vec{D}_{22}\vec{w}_0 = \vec{w}_0,
\end{align*}
where $\vec{e} = (1,1,\cdots,1)^T\in \mathbb{R}^s$.
Moreover,    for the stage consistency, one needs
\begin{align*}   
\vec{D}_{11}\vec{e} + \vec{D}_{12}\vec{w}_1 = \vec{c}.
\end{align*}

For the self completeness, we introduce  the definitions of  the algebraical and diagonal stabilities and the generalized stage order,  which will be used later.

\begin{definition}[{\cite[Def. 2.9.8]{Jackiewicz09}}]
\label{def2.1}
A GLTD is   algebraically stable (also called $G$-stable), if there exists a  symmetric and positive definite matrix $\vec{G} \in \mathbb{R}^{r\times r}$ and a   non-negative definite diagonal matrix $\vec{H} \in \mathbb{R}^{s\times s}$ such that the   matrix
\[  \vec{M} = \!\left[\! \begin{array}{*{4}{c}}\vec{G} - \vec{D}_{22}^T \vec{G} \vec{D}_{22} & \vec{D}_{12}^T\vec{H} - \vec{D}_{22}^T \vec{G} \vec{D}_{21} \vspace{0.1cm}\\ \vec{H}\vec{D}_{12} - \vec{D}_{21}^T \vec{G} \vec{D}_{22} & \vec{D}_{11}^T \vec{H} + \vec{H} \vec{D}_{11} - \vec{D}_{21}^T \vec{G} \vec{D}_{21} \end{array} \!\right]_{(r+s)\times(r+s)} \!,  \]
is non-negative definite.
\end{definition}
The algebraical stability is an important property, and the algebraically stable GLTDs can preserve the long time dynamics of dissipative ODEs \cite{Burrage80}. Next section will show that the algebraically stable GLTDs with SAV  may be unconditionally energy stable for the gradient flows \eqref{1.1}.

\begin{definition}[{\cite[Def. 1.2]{Li99}}]  \label{def2.2}
A GLTD is diagonally stable if there exists a   positive definite diagonal matrix $\tilde{\vec{H}} \in \mathbb{R}^{s\times s}$ such that the matrix $\tilde{\vec{H}}\vec{D}_{11} + \vec{D}_{11}^T\tilde{\vec{H}} \in \mathbb{R}^{s\times s} $ is positive definite.
\end{definition}

The condition in this definition implies the coefficient matrix $\vec{D}_{11}$  is nonsingular \cite{Li89} so that the  diagonally stable GLTD is uniquely solvable.
The diagonal stability will provide us convenience to derive the error estimates of the GLTDs.
%


\begin{definition}  [{\cite[Def. 1.5]{Li99}}]
\label{def2.4}
 The generalized stage order of the GLTD is $\hat{q}$, if $\hat{q}$ is the largest integer and there exist $\hat{u}_i(t_n)$, $i=1,2,\cdots,r$, such that
\[  \rho_{n,i} \!=\! \mathcal{O}(\tau^{\hat{q}+1}),~ i \!=\! 1,2,\cdots,s;~~ \rho_i^{[n]} \!=\! \mathcal{O}(\tau^{\hat{q}+1}),~u_i(t_n) \!-\! \hat{u}_i(t_n) \!=\! \mathcal{O}(\tau^{\hat{q}}),~~ i \!=\! 1,2,\cdots,r,  \]
where $u_i(t_n) := \sum\limits_{j=0}^p w_{ij}\tau^j u^{(j)}(t_n)$, $\rho_{n,i}$ and $\rho_i^{[n]}$ are the local truncation errors given by
\begin{align}  \label{2.1.4}
\begin{cases}
u(t_{n,i}) = \tau\sum \limits_{j=1}^s d_{ij}^{11} u'(t_{n,j}) + \sum\limits_{j=1}^r d_{ij}^{12} \hat{u}_j(t_n) + \rho_{n,i},~~~ i =1,2,\cdots,s, \\[1 \jot]
\hat{u}_i(t_{n+1})= \tau\sum\limits_{j=1}^s d_{ij}^{21} u'(t_{n,j}) + \sum\limits_{j=1}^r d_{ij}^{22} \hat{u}_j(t_n) + \rho_i^{[n]},~~~ i =1,2,\cdots,r,
\end{cases}
\end{align}
 here $t_{n,j}=t_n+c_j \tau$, and $u(t)$ is the smooth solution of \eqref{2.1.1}.
\end{definition}
\noindent
The generalized stage order of a GLTD is related to the stage order $q$ and the method order $p$ of the GLTD. Specifically, when a GLTD has the stage order $q$ and the method order $p=q$, taking $\hat{u}_i(t_n) = u_i(t_n)$ yields that the generalized stage order $\hat{q}$ is at least equal to $q$.

\begin{remark}
The generalized stage order of some GLTDs  is one higher than the stage order so that
it can be used to obtain a sharper error estimate \cite{Li99}. In fact, when a GLTD has the stage order $q$ and the method order $p=q+1$,   it means that
\begin{align*}
\rho_{n,i} = \mathcal{O}(\tau^{q+1}),~~ i = 1,2,\cdots,s;~~~~ \rho_i^{[n]} = \mathcal{O}(\tau^{q+2}),~~ i =1,2,\cdots,r,
\end{align*}
which are defined by \eqref{2.1.4} with $\hat{u}_i(t_n) = u_i(t_n)$. If there exists a constant $\kappa$ such that
\begin{align}   \label{2.1.6}
\rho_{n,i} - \kappa\tau^{q+1} u^{(q+1)}(t_n)  = \mathcal{O}(\tau^{q+2}),~~~~i= 1,2,\cdots,s,
\end{align}
and  one     chooses
\[ \hat{u}_i(t_n)  = u_i(t_n) + w_{i0} \kappa\tau^{q+1} u^{(q+1)}(t_n), \]
then using the consistency condition yields that the generalized stage order of the GLTD is $\hat{q} = q+1$.
\end{remark}

Before ending this subsection, we introduce two typical examples of the  GLTDs.

\begin{example}   \label{exp2.1}
The first   is the $r$-step one-leg time discretization   \cite{Dahlquist75B,Dahlquist78}, which has the following form for   \eqref{2.1.1}
\begin{align}  \label{2.1.7}
\sum_{j=0}^r \alpha_j u^{n+1-j} = \tau f\left( \sum_{j=0}^r \beta_j u^{n+1-j}\right),
\end{align}
where $u^{n+1-j}\approx u(t_{n+1-j})$, $\alpha_j, \beta_j\in \mathbb R$ satisfy
$\alpha_0\beta_0 \neq 0$  and  the   consistency conditions
\[ \sum_{j=0}^r \alpha_j = 0,~~  \sum_{j=0}^r (1\!-\!j)\alpha_j = \sum_{j=0}^r \beta_j = 1.   \]
If setting $u_i^{[n]} = u^{n+1-i}$ for $i=1,2,\ldots,r$, which can be viewed as an approximation to $u_i(t_n)$ with $\vec{w}_0 = (1,1,\ldots,1)^T \in \mathbb{R}^r$ and $\vec{w}_j = \frac{1}{j!}\big(0,(-1)^j,(-2)^j,\ldots,(1\!-\!r)^j\big)^T \in \mathbb{R}^r$ for $j=1,2,\ldots,p$, and using
\[ U_{n,1} = \sum_{j=0}^r \beta_j u^{n+1-j}  = \sum_{j=1}^r\left(\beta_j \!-\! \frac{\beta_0}{\alpha_0}\alpha_j \right) u^{n+1-j} +  \frac{\tau\beta_0}{\alpha_0} f\left(U_{n,1} \right), \]
to approximate $u(t_n+c_1\tau)$ with the consistency condition $c_1 = \sum\limits_{j=0}^r (1\!-\!j)\beta_j$, then the scheme \eqref{2.1.7}  has been reformulated as a GLTD form  \eqref{2.1.2} with the coefficient matrices $\vec{D}_{\imath\jmath}$, $\imath,\jmath = 1,2$, defined by
\[ \vec{D}_{11} \!=\! \left[\frac{\beta_0}{\alpha_0} \right]_{1\times 1},~~~ \vec{D}_{12} \!=\! \left[\beta_1\!-\!\frac{\beta_0}{\alpha_0}\alpha_1, \beta_2 \!-\! \frac{\beta_0}{\alpha_0}\alpha_2,\cdots,  \beta_{r} \!-\! \frac{\beta_0}{\alpha_0}\alpha_{r} \right]_{1\times r},  \]
\[ \vec{D}_{21} \!=\!\!  \left[\! \begin{array}{*{5}{c}}
 1\\ 0 \\  \vdots   \\ 0 \\ 0  \end{array} \!\right]_{r\times 1},~~~~ \vec{D}_{22} \!=\!\!  \left[ \! \begin{array}{*{25}{c}}
-\frac{\alpha_1}{\alpha_0} & -\frac{\alpha_2}{\alpha_0 }& \cdots & -\frac{\alpha_{r-1}}{\alpha_0} & -\frac{\alpha_{r}}{\alpha_0} \\
 1& 0  & \cdots & 0 & 0 \\ 0 & 1  &  \cdots & 0 & 0 \\ \vdots & \vdots & \ddots & \vdots & \vdots  \\ 0 & 0  & \cdots & 1 & 0 \end{array} \!\right]_{r\times r}. \]
The one-leg time discretization \eqref{2.1.7} is algebraically stable  if and only if it is $A$-stable,  see \cite[Theorem 3.3]{Dahlquist78}.
Besides, if \eqref{2.1.7} is $A$-stable, then   $\beta_0/\alpha_0 \!>\!0$ so that \eqref{2.1.7} is   diagonally stable.
In our numerical experiments, we will use
two   special   one-leg time discretizations.
The first  is 
\begin{align}   \label{2.1.8}
u^{n+1} = u^{n} + f\left( \theta u^{n+1} + (1-\theta) u^{n} \right),
\end{align}
where $\theta$ is a parameter. When $\frac{1}{2}\le \theta \le 1$,  \eqref{2.1.8} is $A$-stable, and thus is algebraically stable and diagonally stable \cite{Dahlquist75B}.
When \eqref{2.1.8} is written as a GLTD, the four parameters $\{p,q,r,s\} =\{1,1,1,1\}$ for $\frac{1}{2}<\theta \le 1$ and $\{p,q,r,s\} =\{2,1,1,1\}$ for \eqref{2.1.8} with $\theta =\frac{1}{2}$.
The second  is a class of two-step schemes
\begin{align}   \label{2.1.9}
\frac{1\!+\!\gamma}{2}u^{n+1} =  \gamma u^{n} - \frac{\gamma\!-\!1}{2}u^{n-1} + f\left( \frac{1\!+\!\gamma\!+\!\delta}{4}u^{n+1} + \frac{1\!-\!\delta}{2} u^{n} + \frac{1\!-\!\gamma\!+\!\delta}{4}u^{n-1} \right),
\end{align}
where $\gamma$ and $\delta$ are two parameters.
If $\gamma\ge 0$ and $\delta>0$,  then \eqref{2.1.9} is $A$-stable  \cite{Dahlquist75B},  and thus is algebraically stable and diagonally stable.
Four integers $\{p,q,r,s\} =\{2,1,2,1\}$ for \eqref{2.1.9} as a GLTD. \qed
\end{example}

\begin{example}    \label{exp2.2}
Another important subclass of the GLTDs \eqref{2.1.2} are the multistep Runge-Kutta (MRK) time
discretizations, see e.g. \cite{Burrage87,Li99}.
The $s$-stage and $r$-step MRK schemes   for  \eqref{2.1.1} can be given by
\begin{align}   \label{2.1.10}
\begin{cases}
U_{n,i} = \tau \sum\limits_{j=1}^s a_{ij} f(U_{n,j}) + \sum\limits_{j=1}^r \hat{a}_{ij} u^{n+1-j},~~~ i=1,2,\cdots, s, \\[1 \jot]
u^{n+1} = \tau \sum\limits_{j=1}^s b_j f(U_{n,j}) + \sum\limits_{j=1}^r \hat{b}_j u^{n+1-j},
\end{cases}
\end{align}
 where the coefficients should satisfy  the  consistency conditions
\[  \sum_{j=1}^r \hat{b}_j = 1,~~~\sum_{j=1}^r \hat{a}_{ij} = 1,~~\mbox{for}~~ i=1,2,\cdots,s,~~~\sum_{j=1}^s b_{j} +\sum_{j=1}^r \hat{b}_j(1\!-\!j) = 1, \]
and the stage consistence condition
\[ c_i =   \sum_{j=1}^s a_{ij} + \sum_{j=1}^r (1\!-\!j)\hat{a}_{ij},~~\mbox{for}~~ i=1,2,\cdots,s.
\]
If letting  $u_i^{[n]} = u^{n+1-i}~(i=1,2,\cdots,r)$,
which implies that the vectors $\vec{w}_j$ for  \eqref{2.1.10} have the same form as that of the one-leg time discretization \eqref{2.1.7}, then \eqref{2.1.10} can be written as the form of  \eqref{2.1.2} with the   coefficient matrices $\vec{D}_{\imath\jmath}$, $\imath,\jmath = 1,2$, given by
\[ \vec{D}_{11} \!=\! \vec{A} \!=\! \left(a_{ij}\right)_{s\times s},~~ \vec{D}_{12} \!=\! \hat{\vec{A}} \!=\! \left(\hat{a}_{ij}\right)_{s\times r},~~  \vec{D}_{21} \!=\!  \left[ \begin{array}{*{2}{c}}
\vec{b} \\  \vec{0} \end{array} \right]_{r\times s}, ~~ \vec{D}_{22} \!=\!  \left[ \begin{array}{*{3}{c}}
\hat{\vec{b}} \\
 \vec{I}_{r-1} ~~ \vec{0} \end{array} \right]_{r\times r},   \]
where $\vec{0}$ denotes the zero matrix or vector with appropriate dimensions, $\vec{I}_{r-1} \in \mathbb{R}^{r-1\times r-1}$ is the identity matrix, $\vec{b} = (b_1,b_2,\ldots,b_s) \in \mathbb{R}^s$ and $\hat{\vec{b}} = (\hat{b}_1,\hat{b}_2,\ldots,\hat{b}_r) \in \mathbb{R}^r$.
%
There are six classes of the MRK time discretizations, which are algebraically stable
and diagonally stable with $\vec{G} = \mbox{diag}\big(\hat{b}_1,\hat{b}_1+\hat{b}_2,\cdots,\sum\limits_{i=1}^r\hat{b}_i\big) \!\in \mathbb{R}^{r\times r}$
and $ \vec{H} = \mbox{diag}(\vec{b})\in\mathbb{R}^{s\times s}$, see \cite{Li99}. 
Particularly, when $r = 1$,  \eqref{2.1.10}  reduces to the standard   RK  time discretization (see e.g. \cite{Hairer}), and is algebraically stable if and only if $b_i \ge 0$, $i=1,2,\ldots,s$, and the matrix $\bar{\vec{M}}=(b_ia_{ij} + b_j a_{ji} - b_ib_j)\in\mathbb{R}^{s\times s}$
is nonnegative definite. 
Popular families of the  MRK schemes are the Gauss and Radau IIA time integrations
(four integers $\{p,q,r,s\}$  are $\{2s,s,1,s\}$ and $\{2s\!-\!1,s,1,s\}$, respectively), which are
algebraically stable and diagonally stable. 
Moreover, the MRK time discretization \eqref{2.1.10} has the stage $q$ and the method order $p$, when the following simplified conditions {\cite[pp. 363--364]{Hairer}} 
\begin{align}
\label{2.1.11}
& B(p):~~~ l\sum_{j=1}^s b_{j}c_j^{l-1}  +\sum_{j=1}^r \hat{b}_j(1\!-\!j)^l = 1,~~~ l =1,2,\cdots, p; \\
\label{2.1.12}
& C(q):~~~l\sum_{j=1}^s a_{ij} c_j^{l-1} + \sum_{j=1}^r \hat{a}_{ij} (1\!-\!j)^l = c_i^l ,~~~ l = 1,2,\cdots,q,~i=1,2,\cdots,s,
\end{align}
hold.
\qed\end{example}

\subsection{Semi-discrete SAV-GL schemes}   \label{sec:2.2}

Generally, the free energy $\mathcal{F}(u)$ can be split   as
\begin{align}   \label{2.2.3}
  \mathcal{F}(u) = \frac{1}{2} ( \mathcal{L} u, u) + \mathcal{F}_1(u),
  \ \mathcal{F}_1(u) = \int_{\Omega}F(u)d\mathbf{x},
\end{align}
where  $\Omega$ is a bounded open domain, $\mathcal{L}$ is a symmetric non-negative  linear self-adjoint elliptic operator,     $F(u)$ is
a nonlinear potential function, and $\mathcal{F}_1(u)$ is   bounded from below, i.e., $\mathcal{F}_1(u) \ge -C_0 >0$.
If introducing the SAV $z(t) := \sqrt{\mathcal{F}_1(u)+C_0}$, then one can rewrite the gradient flow equation \eqref{1.1} as follows
\begin{align}\label{2.2.4}
\begin{aligned}
\frac{\partial u}{\partial t} &= \mathcal{G} \mu,~~ \mu = \mathcal{L} u + z W(u), \\
\frac{d z}{dt} &= \frac{1}{2} \left( W(u), \frac{\partial u}{\partial t}\right),
\ W(u):= \frac1{z(t)}
\frac{\delta \mathcal{F}_1}{\delta u},
\end{aligned}
\end{align}
with periodic or homogeneous Neumann boundary conditions.
It is easy to check that \eqref{2.2.4} satisfies the   energy dissipation law
\begin{align*}
\frac{dE}{dt} = \left(\mathcal{L}u, \frac{\partial u}{\partial t}\right) + 2z \frac{dz}{dt} = \left(\mathcal{L}u + z W(u), \frac{\partial u}{\partial t}\right) = \left(\mu, \mathcal{G}\mu \right)  \le 0,
\end{align*}
where the reformulated free energy $E(u)$ is given by
\begin{align*}  
E(u) = \frac{1}{2}(\mathcal{L}u, u) + z^2 - C_0 \equiv \frac{1}{2} \left(\mathcal{L}u,u\right) + \mathcal{F}_1(u).
\end{align*}
The   energy splitting \eqref{2.2.3} is not unique, so is the SAV $z(t)$.
For example, another energy splitting with a linear and self-adjoint operator was considered in \cite{Gong20}.

For convenience, the gradient flows in this paper are assumed to satisfy suitable boundary conditions so that all boundary terms vanish when integration by parts is performed, such as the periodic boundary conditions or homogeneous Neumann boundary conditions. For any $u,v$ satisfying that kind of boundary conditions on the boundary $\partial\Omega$, the property $(\mathcal{L}u,v) = (u,\mathcal{L}v)$ holds. Specifically, for the case of $\mathcal{L} = -\Delta$, a simple calculation shows  $ (-\Delta u,v) = (u,-\Delta v)$ when $u,v$ are periodic or $\frac{\partial u}{\partial \vec{n}} = \frac{\partial v}{\partial \vec{n}} = 0$, where  $\vec{n}$ is the unit outward normal vector on  $\partial\Omega$.

The semi-discrete SAV-GL schemes are built on discretizing the
reformulated gradient flow equations \eqref{2.2.4} by using the GLTDs
\eqref{2.1.2} in time, and we will show that the positive semi-definiteness of  $\mathcal{L}$ plays an important role to derive their energy stability.
%
%
Assume that the quantities $u_i^{[n]}$ and $z_i^{[n]}$ are given, $i=1,2,\cdots,r$.
Extending the GLTDs 
\eqref{2.1.2}   to the system \eqref{2.2.4} yields
\begin{align}  \label{2.2.6}
\begin{cases}
U_{n,i} = \tau \sum\limits_{j=1}^s d_{ij}^{11} \dot{U}_{n,j} + \sum\limits_{j=1}^r d_{ij}^{12} u_j^{[n]},  \\[2 \jot]
Z_{n,i} = \tau \sum\limits_{j=1}^s d_{ij}^{11} \dot{Z}_{n,j} + \sum\limits_{j=1}^r d_{ij}^{12} z_j^{[n]},~~~  i =1,2,\cdots,s,
\end{cases}
\end{align}
\begin{align}   \label{2.2.8}
\begin{cases}
u_i^{[n+1]} = \tau \sum\limits_{j=1}^s d_{ij}^{21} \dot{U}_{n,j} + \sum\limits_{j=1}^r d_{ij}^{22} u_j^{[n]},  \\[2 \jot]
z_i^{[n+1]} = \tau \sum\limits_{j=1}^s d_{ij}^{21} \dot{Z}_{n,j} + \sum\limits_{j=1}^r d_{ij}^{22} z_j^{[n]},~~~ i =1,2,\cdots,r.
\end{cases}
\end{align}
where
\begin{align}  \label{2.2.7}
\dot{U}_{n,i} :=  \mathcal{G} \mu_{n,i},\
\mu_{n,i} = \mathcal{L} U_{n,i} \!+\! Z_{n,i} W(\bar{U}_{n,i}), \
\dot{Z}_{n,i} := \frac{1}{2} \left( W(\bar{U}_{n,i}),  \dot{U}_{n,i}\right),\ i=1,2,\cdots,s,
\end{align}
and  $\bar{U}_{n,i}$ denotes an explicit approximation to $u(\cdot,t_{n,i})$.
Since   $W(\bar{U}_{n,i})$ are   explicitly evaluated,
  \eqref{2.2.6} forms a system of linear equations for the unknown variables $U_{n,i}$ and $Z_{n,i}$, $i=1,2,\cdots,s$,
%
%
so that the SAV-GL schemes \eqref{2.2.6}-\eqref{2.2.8} can be  efficiently implemented.

 \begin{remark}
The schemes \eqref{2.2.6}-\eqref{2.2.8} are built on the original SAV technique (cf. \cite{ShenJ18a,ShenJ18b,ShenJ19}). Up to now, there exist some extensions of the SAV technique, such as the E-SAV \cite{LiuZ20,LiuZ21}, the G-SAV \cite{HuangF21}, the relaxed SAV \cite{JiangM22}, and the relaxed generalized SAV techniques \cite{ZhangY22} etc. One  can combine the GLTDs with those techniques for solving the gradients flows \eqref{1.1}.
\end{remark}

\begin{remark}
We derive   $\bar{U}_{n,i}$   by using a $\nu$-point extrapolation   with the possibly known values $U_{n-1,i}$, $i=1,2,\cdots,s$, and $u_i^{[n]}$, $i=1,2,\cdots,r$,
where  $\nu \le s+r$ and $n>1$,
and  get   $U_{0,i}$, $i=1,2,\cdots,s$, and $u_i^{[1]}, z_i^{[1]}$, $i=1,2,\cdots,r$,
  by using the nonlinear version of the SAV-GL schemes, that is, \eqref{2.2.6}-\eqref{2.2.8}
 with specifying $\bar{U}_{1,i}={U}_{1,i}$. 
Several specific $\nu$-point extrapolations will be given in Section \ref{sec:3.2}.
\end{remark}


\section{Unconditional energy stability}    \label{sec:3}
This section studies  the energy stability of the semi-discrete SAV-GL schemes \eqref{2.2.6}-\eqref{2.2.8}. 
To fix our discussion, similar to \cite{ShenJ19} etc.,
 we assume from here to the hereafter that the boundary conditions are either periodic or such that it allows for integration by parts without introducing
additional boundary terms.

 \subsection{Unconditional energy stability of SAV-GL  schemes}    \label{sec:3.1}
%

\begin{theorem}  \label{Thm3.1}
If the GLTDs  \eqref{2.1.2} are algebraically stable
with a  symmetric and positive definite matrix $\vec{G}\!=\! (g_{ij}) \!\in\! \mathbb{R}^{r\times r}$,
then
the   schemes \eqref{2.2.6}-\eqref{2.2.8}  satisfy the following energy decay property \begin{align}   \label{3.1}
\frac{1}{2} \left( \mathcal{L}\vec{u}^{[n+1]}, \vec{u}^{[n+1]}\right)_{\vec{G}} + \left\| \vec{z}^{[n+1]}\right\|_{\vec{G}}^2  \le
\frac{1}{2} \left( \mathcal{L}\vec{u}^{[n]}, \vec{u}^{[n]}\right)_{\vec{G}} + \left\| \vec{z}^{[n]}\right\|_{\vec{G}}^2,
\end{align}
where  $\vec{u}^{[n]} = \left(u_1^{[n]}, u_2^{[n]},\cdots, u_r^{[n]} \right)^T$,
$\vec{z}^{[n]} = \left(z_1^{[n]}, z_2^{[n]},\cdots, z_r^{[n]} \right)^T$, and
\[
(\mathcal{L}\vec{\Phi},\vec{\Psi})_{ {\vec{G}}} \!:=\! \!\sum_{i,j=1}^r \! {g}_{ij} (\mathcal{L}\phi_i, \psi_j),
\|\vec{\Phi}\|_{ {\vec{G}}}  \!:=\! (\vec{\Phi},\vec{\Psi})_{ {\vec{G}}}^{1/2} \!\!=\!  \!\left(\sum_{i,j=1}^r \! {g}_{ij} (\phi_i, \psi_j)\!\!\right)^{1/2}\!,
\|\vec{\sigma}\|_{{\vec{G}}} \!:=  \!\!\left(\sum_{i,j=1}^r \! {g}_{ij} \sigma_i \sigma_j\!\!\right)^{1/2}\!,
 \]
for any $\vec{\Phi} \!=\! (\phi_1,\phi_2,\cdots,\phi_r)^T$, $\vec{\Psi} \!=\! (\psi_1,\psi_2,\cdots,\psi_r)^T\in \big( L^2(\Omega)\big)^r$, and $\vec{\sigma} \!= \! (\sigma_1, \sigma_2,  \cdots, \sigma_r)^T \!\in\! \mathbb{R}^r$.
\end{theorem}

\begin{proof} 
  Because the GLTDs are algebraically stable in the sense of Definition \ref{def2.1},
  there exist a symmetric positive definite matrix $\vec{G}$ and a non-negative definite diagonal matrix $\vec{H} \!=\! \mbox{diag}([h_1,h_2,\cdots,h_s]) \!\in\! \mathbb{R}^{s\times s}$ such that the matrix $\vec{M} \!=\! (m_{ij}) \!\in\! \mathbb{R}^{(r+s)\times(r+s)}$ is non-negative definite.
Using both the first equations in \eqref{2.2.6} and \eqref{2.2.8} gives
\begin{align}     \label{3.2}
&  \left( \mathcal{L}\vec{u}^{[n+1]}, \vec{u}^{[n+1]}\right)_{\vec{G}} -  \left( \mathcal{L}\vec{u}^{[n]}, \vec{u}^{[n]}\right)_{\vec{G}} - 2\tau \sum_{i=1}^s h_i \left(\mathcal{L}U_{n,i}, \dot{U}_{n,i}\right)  \nonumber \\
=\;& \sum_{i,j}^r g_{ij} \left( \mathcal{L}u_i^{[n+1]}, u_j^{[n+1]} \right)
 - \sum_{i,j}^r g_{ij} \left( \mathcal{L}u_i^{[n]}, u_j^{[n]} \right)
 - 2\tau \sum_{i=1}^s h_i \left(\mathcal{L}U_{n,i}, \dot{U}_{n,i}\right) \nonumber  \\
=\;& - \sum_{i,j}^r g_{ij} \left( \mathcal{L}u_i^{[n]}, u_j^{[n]} \right) +  \sum_{i,j}^r g_{ij} \left( \sum_{k=1}^r d_{ik}^{22} \mathcal{L}u_k^{[n]}, \sum_{k=1}^r d_{jk}^{22} u_k^{[n]} \right)   \nonumber  \\
& - 2\tau \sum_{i=1}^s h_i \left( \sum_{k=1}^r d_{ik}^{12} \mathcal{L} u_k^{[n]}, \dot{U}_{n,i}\right) + 2\tau\sum_{i,j}^r g_{ij} \left( \sum_{k=1}^r d_{ik}^{22} \mathcal{L}u_k^{[n]}, \sum_{k=1}^s d_{jk}^{21}\dot{U}_{n,k} \right)   \nonumber  \\
& - 2\tau^2 \sum_{i=1}^s h_i \left(\sum_{k=1}^s d_{ik}^{11}\mathcal{L}\dot{U}_{n,k} , \dot{U}_{n,i}\right) + \tau^2\sum_{i,j}^r g_{ij} \left( \sum_{k=1}^s d_{ik}^{21}\mathcal{L}\dot{U}_{n,k} , \sum_{k=1}^s d_{jk}^{21}\dot{U}_{n,k}  \right)   \nonumber  \\
=\; & - \sum_{i,j}^r p_{ij} \left( \mathcal{L}u_i^{[n]}, u_j^{[n]} \right) - 2\tau \sum_{i=1}^r\sum_{j=1}^s s_{ij} \left( \mathcal{L}u_i^{[n]}, \dot{U}_{n,j} \right) - \tau^2 \sum_{i,j}^s q_{ij} \left( \mathcal{L}\dot{U}_{n,i}, \dot{U}_{n,j} \right),
\end{align}  
where the property $\left( \mathcal{L} u_i^{[n]}, \dot{U}_{n,j}\right) = \left( u_i^{[n]}, \mathcal{L} \dot{U}_{n,j}\right)$ has been used in the second equality,
$p_{ij}$, $s_{ij}$ and $q_{ij}$ form the matrices $\vec{P}, \vec{S}$ and $\vec{Q}$, respectively, and satisfy
\[ \vec{P} = \vec{G}-\vec{D}_{22}^T \vec{G} \vec{D}_{22},~~~ \vec{S} = \vec{D}_{12}^T \vec{H} - \vec{D}_{22}^T \vec{G} \vec{D}_{12},~~~ \vec{Q} = \vec{D}_{11}^T \vec{H} + \vec{H} \vec{D}_{11} - \vec{D}_{21}^T \vec{G} \vec{D}_{21}.  \]
If setting
\[  \tilde{\vec{U}}_n = \left(u_1^{[n]},u_2^{[n]},\cdots,u_r^{[n]},\tau \dot{U}_{n,1},\tau \dot{U}_{n,2}, \cdots,  \tau \dot{U}_{n,s} \right)^T, \]
then the identity \eqref{3.2} can be rewritten as
\begin{align}     \label{3.3}
\left( \mathcal{L}\vec{u}^{[n+1]}, \vec{u}^{[n+1]}\right)_{\vec{G}} -  \left( \mathcal{L}\vec{u}^{[n]}, \vec{u}^{[n]}\right)_{\vec{G}} = 2\tau \sum_{i=1}^s h_i \left(\mathcal{L}U_{n,i}, \dot{U}_{n,i}\right)  -  \sum_{i,j}^{r+s} m_{ij} \left( \mathcal{L}\tilde{U}_{n,i},\tilde{U}_{n,j} \right).
\end{align}
Since the matrix $\vec{M}$ is non-negative definite and the operator $\mathcal{L}$ is positive semi-definite, it holds from \eqref{3.3} that
\begin{align}     \label{3.4}
\left( \mathcal{L}\vec{u}^{[n+1]}, \vec{u}^{[n+1]}\right)_{\vec{G}} \le  \left( \mathcal{L}\vec{u}^{[n]}, \vec{u}^{[n]}\right)_{\vec{G}} + 2\tau \sum_{i=1}^s h_i \left(\mathcal{L}U_{n,i}, \dot{U}_{n,i}\right).
\end{align}
Similarly, from both the second equations in \eqref{2.2.6} and \eqref{2.2.8}, one can  obtain
\begin{align*}  
\left\| \vec{z}^{[n+1]}\right\|_{\vec{G}}^2 \le  \left\|\vec{z}^{[n]}\right\|_{\vec{G}}^2 + 2\tau \sum_{i=1}^s h_i Z_{n,i} \dot{Z}_{n,i}.
\end{align*}
Inserting the third equation of \eqref{2.2.7} into it gives
\begin{align}     \label{3.6}
\left\| \vec{z}^{[n+1]}\right\|_{\vec{G}}^2 \le \left\|\vec{z}^{[n]}\right\|_{\vec{G}}^2 + \tau \sum_{i=1}^s h_i Z_{n,i}\left( W(\bar{U}_{n,i}),  \dot{U}_{n,i}\right).
\end{align}
Moreover, taking the $L^2$ inner product of the second equation in \eqref{2.2.7} with $\dot{U}_{n,i}$ gives
\begin{align*}
\left(\mathcal{L} U_{n,i}, \dot{U}_{n,i} \right)  = \left(\mu_{n,i} - Z_{n,i} W(\bar{U}_{n,i}), \dot{U}_{n,i}\right).
\end{align*}
Substituting it into \eqref{3.4} yields
\begin{align}     \label{3.8}
\left( \mathcal{L}\vec{u}^{[n+1]}, \vec{u}^{[n+1]}\right)_{\vec{G}} \le  \left( \mathcal{L}\vec{u}^{[n]}, \vec{u}^{[n]}\right)_{\vec{G}} + 2\tau \sum_{i=1}^s h_i \left(\mu_{n,i} \!-\! Z_{n,i} W(\bar{U}_{n,i}), \dot{U}_{n,i}\right).
\end{align}
Combining \eqref{3.6} with \eqref{3.8} and using $\dot{U}_{n,i} = \mathcal{G}\mu_{n,i}$ yields
\begin{align*}
& \frac{1}{2} \left( \mathcal{L}\vec{u}^{[n+1]}, \vec{u}^{[n+1]}\right)_{\vec{G}} + \left\| \vec{z}^{[n+1]}\right\|_{\vec{G}}^2  \le
\frac{1}{2} \left( \mathcal{L}\vec{u}^{[n]}, \vec{u}^{[n]}\right)_{\vec{G}} + \left\| \vec{z}^{[n]}\right\|_{\vec{G}}^2 + \tau \sum_{i=1}^s h_i \left(\mu_{n,i}, \dot{U}_{n,i}\right)  \nonumber \\
\le\; & \frac{1}{2} \left( \mathcal{L}\vec{u}^{[n]}, \vec{u}^{[n]}\right)_{\vec{G}} + \left\| \vec{z}^{[n]}\right\|_{\vec{G}}^2 + \tau \sum_{i=1}^s h_i \left(\mu_{n,i}, \mathcal{G}\mu_{n,i}\right).
\end{align*}
Since $h_i>0 $ for $i=1,2,\cdots,s$ and the operator $\mathcal{G}$ is
non-positive, we can deduce that the   schemes \eqref{2.2.6}-\eqref{2.2.8} satisfy the desired energy decay property \eqref{3.1}. The proof is completed.
\end{proof}

\begin{remark}
The $\vec{G}$-weighted norm of  $\vec{\Phi} \!=\! (\phi_1,\phi_2,\cdots,\phi_r)^T\in\big(L^2(\Omega)\big)^r$ is equivalent to its $L^2$ norm, since $\lambda_{min} \sum\limits_{i=1}^r \|\phi_i\|^2 \le \|\vec{\Phi}\|_G ^2 \le \lambda_{max}\sum\limits_{i=1}^r \|\phi_i\|^2$, where $\lambda_{min}$ and $\lambda_{max}$ are the minimum and maximum eigenvalues of the matrix $\vec{G}$, respectively.
\end{remark}

\begin{remark}  \label{remark3.3}
 The discrete energy decay property does still hold  for the nonlinear version of  the SAV-GL schemes, i.e. \eqref{2.2.6}-\eqref{2.2.8} with   the unknown stage  values $U_{n,i}$  instead of the extrapolated values $\bar{U}_{n,i}$.
\end{remark}

\subsection{Applications to the one-leg and MRK methods}    \label{sec:3.2}

 This subsection discusses   the one-leg and MRK time discretizations
for  the gradient flows.

First of all, applying the one-step one-leg time discretization \eqref{2.1.8}
 with $\theta \in [\frac{1}{2}, 1]$ to the reformulated equation \eqref{2.2.4}  gives the following semi-discrete SAV-GL scheme
\begin{align} \label{3.2.1}
\begin{cases}
u^{n+1} - u^n = \tau\mathcal{G}\mu^{n+\theta},  \\[2 \jot]
\mu^{n+\theta} = \mathcal{L}\left[ \theta u^{n+1} \!+\! (1\!-\!\theta) u^n\right] + \left[ \theta z^{n+1} \!+\! (1\!-\!\theta) z^n\right] W(\bar{u}^{n+\theta}),  \\[2 \jot]
z^{n+1} \!-\! z^{n} = \frac{1}{2} \left( W(\bar{u}^{n+\theta}), u^{n+1} \!-\! u^{n} \right),
\end{cases}
\end{align}
where   $\bar{u}^{n+\theta} := (2\!-\!\theta)u^n \!-\! (1\!-\!\theta)u^{n-1}$    approximating explicitly $u(\cdot,t_{n+\theta})$. As mentioned in Section \ref{sec:2}, the one-leg time discretization \eqref{2.1.8} can be rewritten as a GLTD, and is algebraically stable with   $\vec{G} \!=\! 1$ and  $\vec{H} \!=\! 1$ for any $\theta \in [\frac{1}{2}, 1]$, see \cite[Th. 3.3]{Dahlquist75B},
so that  
  \eqref{3.2.1} is unconditionally energy stable   by Theorem \ref{Thm3.1}.
in the sense of that
\begin{align*}
\frac{1}{2} \left( \mathcal{L}u^{n+1}, u^{n+1}\right) + \left( z^{n+1}\right)^2  \le
\frac{1}{2} \left( \mathcal{L}u^{n}, u^{n}\right) + \left( z^{n}\right)^2,
\end{align*}
which can also be directly deduced 
by
taking the $L^2$ inner product of the first and second equations in \eqref{3.2.1} with $\mu^{n+\theta}$ and $u^{n+1}-u^n$, respectively, and multiplying the third equation in \eqref{3.2.1} with $2\left[ \theta z^{n+1} + (1\!-\!\theta) z^n\right]$, and then using the inequality $-ab\ge -\frac{1}{2}(a^2+b^2)$.

\begin{remark}
When $\theta = 1$ and $\frac{1}{2}$,    \eqref{3.2.1} becomes the SAV-BDF1 and SAV-CN scheme in \cite{ShenJ19}, respectively.
\end{remark}

Next, applying  the two-step one-leg time discretization \eqref{2.1.9} with
 $\gamma\ge 0$ and $\delta>0$ to   the reformulated equation \eqref{2.2.4} yields the following
 SAV-GL scheme
\begin{align}
\begin{cases}   \label{3.2.9}
\frac{1+\gamma}{2}u^{n+1} - \gamma u^{n} + \frac{\gamma-1}{2}u^{n-1} = \tau\mathcal{G}\mu^{n+\gamma-1},\\[2 \jot]
\mu^{n+\gamma-1} = \mathcal{L}\!\left[ \frac{1+\gamma+\delta}{4}u^{n+1} + \frac{1-\delta}{2} u^{n} + \frac{1-\gamma+\delta}{4}u^{n-1}\right] \\[2 \jot]
\hspace{1.3cm}~ + \left[ \frac{1+\gamma+\delta}{4}z^{n+1} + \frac{1-\delta}{2} z^{n} + \frac{1-\gamma+\delta}{4}z^{n-1}\right]  W(\bar{u}^{n+\gamma-1}), \\[2 \jot]
\frac{1+\gamma}{2}z^{n+1} -\gamma z^{n} + \frac{\gamma-1}{2} z^{n-1}= \frac{1}{2} \left( W(\bar{u}^{n+\gamma-1}),  \frac{1+\gamma}{2}u^{n+1} -\gamma u^{n} + \frac{\gamma-1}{2} u^{n-1} \right),
\end{cases}
\end{align}
where $\bar{u}^{n+\gamma-1} := \left(1+\frac{\gamma}{2}\right)u^{n} - \frac{\gamma}{2}u^{n-1}$   approximating  $u(\cdot, t_{n+\gamma-1})$.
As far as we know, there is no result on the energy stability of the general scheme \eqref{3.2.9} in the literature except for some special cases.
Here, it may be conveniently obtained by using  Theorem  \ref{Thm3.1} and the fact {\cite[Th. 3.3]{Dahlquist75B}} that  the time discretization  \eqref{2.1.9}
with $\gamma\ge 0$ and $\delta>0$ is algebraically stable  with $\vec{H} = 1$ and
\[ \vec{G} =   \frac{1}{4}\!\left[\! \begin{array}{*{4}{c}}
(1+\gamma)^2+\delta & 1-\delta-\gamma^2 \\ 1-\delta-\gamma^2 & (\gamma-1)^2 + \delta \end{array} \!\right]. \]

\begin{theorem}  \label{Thm3.3}
The semi-discrete SAV-GL scheme \eqref{3.2.9} is unconditionally energy stable in the sense that 
\begin{align}   \label{3.2.12}
\tilde{\mathcal{F}}\left(u^{n+1},u^{n},z^{n+1},z^{n} \right) \le \tilde{\mathcal{F}}\left(u^{n},u^{n-1},z^{n},z^{n-1} \right),
\end{align}
where
\begin{align*}
& \tilde{\mathcal{F}}\left(u^{n+1},u^{n},z^{n+1},z^{n} \right): = \frac{\delta}{2(\gamma\!-\!1)^2 \!+\! 2\delta} \left( \mathcal{L}u^{n+1}, u^{n+1}\right) + \frac{\delta}{(\gamma\!-\!1)^2 \!+\! \delta}\left( z^{n+1}\right)^2  \\
& \hspace{1.5cm} + \frac{(\gamma\!-\!1)^2\!+\!\delta}{8}\left( \mathcal{L}\left( \frac{\gamma^2\!+\!\delta\!-\!1}{(\gamma\!-\!1)^2\!+\!\delta} u^{n+1} \!-\! u^{n}\right), \left( \frac{\gamma^2\!+\!\delta\!-\!1}{(\gamma\!-\!1)^2\!+\!\delta} u^{n+1} \!-\! u^{n}\right)\right) \\
&\hspace{1.5cm}  + \frac{(\gamma\!-\!1)^2\!+\!\delta}{4} \left( \frac{\gamma^2\!+\!\delta\!-\!1}{(\gamma\!-\!1)^2\!+\!\delta} z^{n+1} \!-\! z^{n}\right)^2.
\end{align*}
\end{theorem}

\begin{remark}
When  $\gamma = 2$ and $\delta = 1$, the scheme \eqref{3.2.9}   reduces to the SAV-BDF2 in \cite{ShenJ19},
and the energy inequality \eqref{3.2.12} is the same as that in \cite{ShenJ19},   deduced with a different technique. 
When   $\gamma \!=\! 2\theta$ and $\delta \!=\! 1\!-\!4(1\!-\!\theta)^2$, the scheme \eqref{3.2.9}  reduces to that   in \cite{YangZ19}    for the Cahn-Hilliard equation, where
 the energy stability is analyzed by using some identities.
\end{remark}

\begin{remark}
Since the highest method order of the $A$-stable (algebraically stable) one-leg time discretizations is 2, this paper only considers the first- and second-order one-leg schemes.
It is interesting to explore higher-order one-leg schemes for the gradient flows with the aid of the novel SAV approach  \cite{HuangF20}. 
\end{remark}

Finally,   the high-order   algebraically stable MRK time discretizations   \eqref{2.1.10} with the stage order $q=s$ and the method order $p=s$ are   applied to  the reformulated equation \eqref{2.2.4}.
We first evaluate the stage quantities $U_{n,i}, Z_{n,i}$, $i=1,2,\cdots,s$, by the coupled linear system
\begin{align}  \label{3.2.13}
\begin{cases}
U_{n,i} = \tau \sum\limits_{j=1}^s a_{ij} \dot{U}_{n,j} + \sum\limits_{j=1}^r \hat{a}_{ij} u^{n+1-j},   \\[2 \jot]
Z_{n,i} = \tau \sum\limits_{j=1}^s a_{ij} \dot{Z}_{n,j} + \sum\limits_{j=1}^r \hat{a}_{ij} z^{n+1-j},  ~~~ i =1,2,\cdots,s,
\end{cases}
\end{align}
and then calculate the output quantities $ u^{n+1}, z^{n+1}$  by
\begin{align}   \label{3.2.15}
\begin{cases}
u^{n+1} = \tau \sum\limits_{j=1}^s b_{j} \dot{U}_{n,j} + \sum\limits_{j=1}^r \hat{b}_{j} u^{n+1-j}, \\[2 \jot]
z^{n+1} = \tau \sum\limits_{j=1}^s b_{j} \dot{Z}_{n,j} + \sum\limits_{j=1}^r \hat{b}_{j} z^{n+1-j}.
\end{cases}
\end{align}
where
\begin{align*}
\dot{U}_{n,i} = \mathcal{G} \mu_{n,i},~~~ \mu_{n,i} = \mathcal{L} U_{n,i} + Z_{n,i} W(\bar{U}_{n,i}),~~~ \dot{Z}_{n,i} = \frac{1}{2} \left( W(\bar{U}_{n,i}),  \dot{U}_{n,i}\right),
\end{align*}
and   $\bar{U}_{n,i}$   is evaluated by using the following  Lagrange interpolation
\begin{align*} 
\bar{U}_{n,i} = \sum_{j=1}^s L_j(1+c_i) U_{n-1,j},\quad
L_j(x) = \prod_{l=1,l\neq j}^s  \frac{x-c_l}{c_j-c_l}, ~~\mbox{for}~~ s\ge 2,
\end{align*}
and
\[ \bar{U}_{n,1}  = (1+c_1) u^n - c_1 u^{n-1},~~~ \mbox{for}~~ s = 1.  \]

As mentioned in Section \ref{sec:2}, the MRK time discretizations \eqref{2.1.10} belong to the GLTDs so that using  Theorem \ref{Thm3.1} can give the energy stability of the schemes \eqref{3.2.13}-\eqref{3.2.15}.

\begin{theorem}  \label{Thm4.2.1}
If the MRK time discretizations \eqref{2.1.10} are algebraically stable, then
the SAV-GL schemes \eqref{3.2.13}-\eqref{3.2.15} satisfy 
\begin{align}   \label{3.2.17}
\bar{\mathcal{F}}\!\left( u^{n+1},\cdots,u^{n+2-r},z^{n+1},\cdots,z^{n+2-r}\right)  \le \bar{\mathcal{F}}\!\left( u^{n},\cdots,u^{n+1-r},z^{n},\cdots,z^{n+1-r}\right)\!,
\end{align}
where
\begin{align*}
& \bar{\mathcal{F}}\left( u^{n},\cdots,u^{n+1-r},z^{n},\cdots,z^{n+1-r}\right)   \\
= \;& \frac{1}{2}\sum_{j=1}^r\hat{b}_j \left(\mathcal{L}u^{n},u^{n}\right) + \cdots +\frac{\hat{b}_1}{2} \left(\mathcal{L}u^{n+1-r},u^{n+1-r}\right) + \sum_{j=1}^r\hat{b}_j \left(z^{n}\right)^2 + \cdots +\hat{b}_1 \left(z^{n+1-r}\right)^2.
\end{align*}
In particular, the SAV-RK schemes \eqref{3.2.13}-\eqref{3.2.15}  satisfy
\begin{align}   \label{3.2.18}
\frac{1}{2} \left(\mathcal{L}u^{n+1},u^{n+1}\right) + \left(z^{n+1}\right)^2 \le \frac{1}{2} \left(\mathcal{L}u^{n},u^{n}\right) + \left(z^{n}\right)^2.
\end{align}
\end{theorem}

\begin{remark}
The schemes \eqref{3.2.13}-\eqref{3.2.15} contain the arbitrarily high-order (in time) schemes in \cite{Gong20}, which  were derived by
combining the structure-preserving Gaussian collocation time discretization with the SAV approach. Also, the extrapolated RK-SAV schemes derived in \cite{Akrivis19} for solving Allen-Cahn and Cahn-Hilliard equations were covered by the schemes \eqref{3.2.13}-\eqref{3.2.15}.
\end{remark}

\section{Error estimates of the SAV-GL methods}   \label{sec:4}

This section establishes the error estimates of the SAV-GL schemes \eqref{2.2.6}-\eqref{2.2.8} for   the $L^2$ gradient flow, i.e., $\mathcal{G} = -I$ with the  free energy density $F(u)$ in polynomial.  The analysis for the $H^{-1}$ gradient flow is quite similar
and omitted here  to avoid a repetitive discussion.
Our analyses will be based on the following hypothesises: \\
$\mathcal{H}_1$: The exact solutions $ u,z $ of the reformulated equation \eqref{2.2.4} is bounded and smooth enough, and $W(u)$ are locally Lipschitz continuous; \\
$\mathcal{H}_2$: The starting values $U_{0,i}, Z_{0,i}, u_i^{[1]}$ and $z_i^{[1]}$ are sufficiently accurate with  the generalized stage order $q$.

The readers are referred to    \cite{ShenJ18b}
for some discussions on the smoothness and bound of  $u,z$ of  \eqref{2.2.4} in  the hypothesis $\mathcal{H}_1$. 
The  polynomial   $F(u)$  is local Lipschitz continuous, so is $W(u)$.
The  hypothesis $\mathcal{H}_2$ is reasonable, since the nonlinear  arbitrarily high order  SAV-RK schemes
 can be used to compute the starting values, see also Remark \ref{remark3.3}.

\subsection{Local error analysis}

This subsection  estimates the local errors of the SAV-GL scheme \eqref{2.2.6}-\eqref{2.2.8}, where the local errors $\eta_{n,i}, \eta_i^{[n]}, \sigma_{n,i}$ and $\sigma_i^{[n]}$ are determined by
\begin{align}  \label{4.1}
\begin{cases}
u_{n,i} = \tau \sum\limits_{j=1}^s d_{ij}^{11} \dot{u}_{n,j} + \sum\limits_{j=1}^r d_{ij}^{12} \hat{u}_j(t_n) + \eta_{n,i},  \\[2 \jot]
z_{n,i} = \tau \sum\limits_{j=1}^s d_{ij}^{11} \dot{z}_{n,j} + \sum\limits_{j=1}^r d_{ij}^{12} \hat{z}_j(t_n) + \sigma_{n,i},~~~  i =1,2,\cdots,s,
\end{cases}
\end{align}
and
\begin{align}   \label{4.2}
\begin{cases}
\hat{u}_i(t_{n+1}) = \tau \sum\limits_{j=1}^s d_{ij}^{21} \dot{u}_{n,j} + \sum\limits_{j=1}^r d_{ij}^{22} \hat{u}_j(t_n) + \eta_i^{[n]}, \\[2 \jot]
\hat{z}_i(t_{n+1}) = \tau \sum\limits_{j=1}^s d_{ij}^{21} \dot{z}_{n,j} + \sum\limits_{j=1}^r d_{ij}^{22} \hat{z}_j(t_n) + \sigma_i^{[n]},~~~ i =1,2,\cdots,r,
\end{cases}
\end{align}
 here
\begin{align}  \label{4.3}
& \dot{u}_{n,i} = - \mathcal{L} u_{n,i} - z_{n,i} W(\bar{u}_{n,i}), ~~~~ \dot{z}_{n,i} = \frac{1}{2} \left( W(\bar{u}_{n,i}),  \dot{u}_{n,i}\right),
\end{align}
and $u_{n,i} = u(\cdot,t_{n,i}), z_{n,i} = z(t_{n,i})$,  $\hat{u}_i(t_{n})$ and $\hat{z}_i(t_{n})$ are abstract functions and may be equal to $u_i(\cdot,t_n)$ and $z_i(t_n)$, respectively, which denote the linear combination of  the scaled derivatives of  $u,z$ of    \eqref{2.2.4}, and $\bar{u}_{n,i}$ is the $\nu$-point extrapolation with the quantity $u_{n-1,i}$ and $\hat{u}_i(t_n)$.
%

\begin{lemma}  \label{Lam4.1}
Under the hypothesis $\mathcal{H}_1$, if the GLTDs \eqref{2.1.2} have the generalized stage order $\hat{q}$, then the local errors $\eta_{n,i}, \eta_i^{[n]}, \sigma_{n,i}$ and $\sigma_i^{[n]}$ satisfy
\begin{align}   \label{4.4}
\sum_{i=1}^s \left( \left\|\eta_{n,i}\right\| + \left|\sigma_{n,i} \right|\right) \le C \tau^{\min\{\hat{q}+1,\nu+1\}},~~ \sum_{i=1}^r \left( \left\|\eta_i^{[n]} \right\| + \left|\sigma_i^{[n]} \right|\right) \le C \tau^{\min\{\hat{q}+1,\nu+1\}},
\end{align}
where $C>0$ used above and hereafter is a constant independent on the time stepsize $\tau$.
\end{lemma}
\begin{proof} Using both the first relations in \eqref{4.1} and  \eqref{2.2.4} gives
\begin{align}    \label{4.5}
u_{n,i} \!-\!  \sum\limits_{j=1}^r d_{ij}^{12} \hat{u}_j(t_n) \!-\! \tau \sum\limits_{j=1}^s d_{ij}^{11} u_t(t_{n,j})  \!=\! \tau \sum\limits_{j=1}^s d_{ij}^{11} z_{n,j}\left[ W(\bar{u}_{n,j}) \!-\! W(u_{n,j})\right] \!+\! \eta_{n,i}.
\end{align}
Let us denote by $\tilde{\eta}_{n,i}$   the left hand side of \eqref{4.5}.
 Since the generalized stage order of the GLTDs is  $\hat{q}$, we have
\begin{align}  \label{4.6}
\left\|\tilde{\eta}_{n,i} \right\|  \le C \tau^{\hat{q}+1},~~~~ i=1,2,\cdots, s.
\end{align}
Note that due to the $\nu$-point extrapolation, the error $u_{n,i} - \bar{u}_{n,i}$ is at least $\mathcal{O}(\tau^{\nu})$, i.e.,
\begin{align*}
\left\|u_{n,i} - \bar{u}_{n,i} \right\|  \le C \tau^{\nu},~~~~ i=1,2,\cdots, s,
\end{align*}
which implies that
\begin{align}   \label{4.7}
\left\|W(\bar{u}_{n,i}) - W(u_{n,i}) \right\|  \le C \tau^{\nu},~~~~ i=1,2,\cdots, s.
\end{align}
Combining \eqref{4.5} with \eqref{4.6} and \eqref{4.7} yields
\begin{align*}
\left\|\eta_{n,i} \right\|  \le C \tau^{\min\{\hat{q}+1,\nu+1\}},~~~~ i=1,2,\cdots, s.
\end{align*}
On the other hands, it follows from both the first relations in \eqref{4.2} and   \eqref{2.2.4} that
\begin{align}    \label{4.9}
\hat{u}_i(t_{n+1}) \!-\!  \sum\limits_{j=1}^r d_{ij}^{22} \hat{u}_j(t_n) \!-\! \tau \sum\limits_{j=1}^s d_{ij}^{21} u_t(t_{n,j})  \!=\! \tau \sum\limits_{j=1}^s d_{ij}^{21} z_{n,j} \left[ W(\bar{u}_{n,j}) \!-\! W(u_{n,j})\right] \!+\! \eta_i^{[n]}.
\end{align}
If setting the quantity at the left hand side of \eqref{4.9} as $\tilde{\eta}_i^{[n]}$, then one can derive from the definition of the generalized stage order $\hat{q}$ that
\begin{align}  \label{4.10}
\left\|\tilde{\eta}_i^{[n]} \right\|  \le C \tau^{\hat{q}+1},~~~~ i=1,2,\cdots, r.
\end{align}
Combining \eqref{4.9} with \eqref{4.10} and \eqref{4.7}  deduces
\begin{align}  \label{4.11}
\left\|\eta_i^{[n]} \right\|  \le C \tau^{\min\{\hat{q}+1,\nu+1\}},~~~~ i=1,2,\cdots, r.
\end{align}
Similarly, the local errors $ \sigma_{n,i}$ and $\sigma_i^{[n]}$ can be derived.
Therefore, the estimates \eqref{4.4} hold and   the proof is completed.
\end{proof}

\subsection{Global error analysis}

This subsection   focuses on the global error analysis of the SAV-GL schemes \eqref{2.2.6}-\eqref{2.2.8}. To this end,   define the intermediate values $\mathcal{U}_{n,i}$, $\mathcal{Z}_{n,i}$, $\mathcal{U}_i^{[n+1]}$ and $\mathcal{Z}_i^{[n+1]}$ by
\begin{align}  \label{4.12}
\begin{cases}
\mathcal{U}_{n,i} = \tau \sum\limits_{j=1}^s d_{ij}^{11} \dot{\mathcal{U}}_{n,j} + \sum\limits_{j=1}^r d_{ij}^{12} \hat{u}_j(t_n),  \\[2 \jot]
\mathcal{Z}_{n,i} = \tau \sum\limits_{j=1}^s d_{ij}^{11} \dot{\mathcal{Z}}_{n,j} + \sum\limits_{j=1}^r d_{ij}^{12} \hat{z}_j(t_n),~~~  i =1,2,\cdots,s,
\end{cases}
\end{align}
and
\begin{align}   \label{4.13}
\begin{cases}
\mathcal{U}_i^{[n+1]} = \tau \sum\limits_{j=1}^s d_{ij}^{21} \dot{\mathcal{U}}_{n,j} + \sum\limits_{j=1}^r d_{ij}^{22} \hat{u}_j(t_n),  \\[2 \jot]
\mathcal{Z}_i^{[n+1]} = \tau \sum\limits_{j=1}^s d_{ij}^{21} \dot{\mathcal{Z}}_{n,j} + \sum\limits_{j=1}^r d_{ij}^{22} \hat{z}_j(t_n),~~~ i =1,2,\cdots,r,
\end{cases}
\end{align}
where
\begin{align}  \label{4.14}
& \dot{\mathcal{U}}_{n,i} = - \left[ \mathcal{L} \mathcal{U}_{n,i} + \mathcal{Z}_{n,i} W(\bar{u}_{n,i})\right],~~~ \dot{\mathcal{Z}}_{n,i} = \frac{1}{2} \left( W(\bar{u}_{n,i}),  \dot{\mathcal{U}}_{n,i}\right).
\end{align}
Those intermediate values will play an important role to derive the global error estimates of the SAV-GL schemes  \eqref{2.2.6}-\eqref{2.2.8}.
Such technique has been used to study the convergence of the GLTDs for the ODEs, see e.g. \cite{Li89,HuangC01}.
We   first give  the error estimates between  the intermediate values $\mathcal{U}_{n,i}$, $\mathcal{Z}_{n,i}$, $\mathcal{U}_i^{[n+1]}$, $\mathcal{Z}_i^{[n+1]}$
and the values $u_{n,i}$, $z_{n,i}$, $\hat{u}_i(t_{n+1})$, $\hat{z}_i(t_{n+1})$.

\begin{theorem}   \label{Thm4.1}
Under the hypothesis $\mathcal{H}_1$, if the GLTDs \eqref{2.1.2}  are diagonally stable and have the generalized stage order $\hat{q}$, then the following estimates can be obtained
\begin{align}  \label{4.15}
& \sum_{i=1}^s \left(\left\|u_{n,i} - \mathcal{U}_{n,i}\right\| + \left|z_{n,i} - \mathcal{Z}_{n,i}\right|\right)    \le C \tau^{\min\{\hat{q}+1,\nu+1\}},  \\
\label{4.16}
& \sum_{i=1}^r \left(\left\|\hat{u}_i(t_{n+1}) - \mathcal{U}_i^{[n+1]}\right\| + \left|\hat{z}_i(t_{n+1}) - \mathcal{Z}_i^{[n+1]}\right|\right)  \le C \tau^{\min\{\hat{q}+1,\nu+1\}},
\end{align}
when the time stepsize $\tau$ is sufficiently small.
\end{theorem}
\begin{proof}
Subtracting the first equation in \eqref{4.12} from that in \eqref{4.1} yields
\begin{align}  \label{4.17}
u_{n,i} \!- \mathcal{U}_{n,i} = \tau \sum_{j=1}^s d_{ij}^{11} \left[ \dot{u}_{n,j} \!- \dot{\mathcal{U}}_{n,j} \right] + \eta_{n,i},
\end{align}
where
\begin{align*}
\dot{u}_{n,i} \!- \dot{\mathcal{U}}_{n,i} =  -\mathcal{L}\left(u_{n,i} \!- \mathcal{U}_{n,i} \right) - W(\bar{u}_{n,i}) \left( z_{n,i} \!- \mathcal{Z}_{n,i} \right).
\end{align*}
Since the GLTDs are diagonally stable, there exists a positive definite diagonal matrix $\tilde{\vec{H}} \!=\! \mbox{diag}([\tilde{h}_1,\tilde{h}_2,\cdots,\tilde{h}_s])$ such that the matrix $\tilde{\vec{M}} \!=\! (\tilde{m}_{ij}) \!=\! \tilde{\vec{H}}\vec{D}_{11} \!+ \!\vec{D}_{11}^T\tilde{\vec{H}}$ is positive definite. Hence, the matrix $\vec{D}_{11}$ is nonsingular and there exists a positive constant $l$ dependent only on the method such that the matrix
\begin{align}  \label{4.18}
\tilde{\vec{M}}_l =(\tilde{m}_{ij}^{(l)})= \vec{D}_{11}^{-T}\tilde{\vec{M}} \vec{D}_{11}^{-1}- 2l\tilde{\vec{H}} = \vec{D}_{11}^{-T}\tilde{\vec{H}} + \tilde{\vec{H}} \vec{D}_{11}^{-1} -2l\tilde{\vec{H}}
\end{align}
is   positive definite. Use $\tilde{m}_{ij}^{(d)}$ to denote
the entries of the matrix $\tilde{\vec{M}}_d = \tilde{\vec{H}} \vec{D}_{11}^{-1}$.

 It holds
\begin{align}    \label{4.19*}
0\le\;& 2l\sum_{i=1}^s \tilde{h}_i \left\| u_{n,i} \!- \mathcal{U}_{n,i} \right\|^2 - 2\tau\sum_{i=1}^s \tilde{h}_i \left( -\mathcal{L}\left( u_{n,i} \!- \mathcal{U}_{n,i} \right), u_{n,i} \!- \mathcal{U}_{n,i}\right) \nonumber \\
\overset{\eqref{4.18}}{=} \;& -\sum_{i,j=1}^s \tilde{m}_{ij}^{(l)} \left( u_{n,i} \!- \mathcal{U}_{n,i}, u_{n,j} \!- \mathcal{U}_{n,j} \right) +2\sum_{i,j=1}^s \tilde{m}_{ij}^{(d)}\left( u_{n,i} \!- \mathcal{U}_{n,i}, u_{n,j} \!- \mathcal{U}_{n,j} \right) \nonumber \\
& - 2\tau\sum_{i=1}^s \tilde{h}_i \left( -\mathcal{L}\left( u_{n,i} \!- \mathcal{U}_{n,i} \right), u_{n,i} \!- \mathcal{U}_{n,i} \right)  \nonumber \\
\overset{\eqref{4.17}}{=} \; & -\sum_{i,j=1}^s \tilde{m}_{ij}^{(l)} \left( u_{n,i}\! - \mathcal{U}_{n,i}, u_{n,j} \!- \mathcal{U}_{n,j} \right) + 2\sum_{i,j=1}^s  \tilde{m}_{ij}^{(d)} \left( u_{n,i} \!- \mathcal{U}_{n,i}, \eta_{n,j} \right) \nonumber \\
& - 2\tau\sum_{i=1}^s \tilde{h}_i \left( W(\bar{u}_{n,i} )\left( z_{n,i} \!- \mathcal{Z}_{n,i} \right), u_{n,i} \!- \mathcal{U}_{n,i}  \right)  \nonumber \\
\le\; & -\lambda_l \sum_{i=1}^s\left\| u_{n,i} - \mathcal{U}_{n,i} \right\|^2 + C \sum_{i=1}^s\left\| u_{n,i} - \mathcal{U}_{n,i} \right\|\sum_{i=1}^s\left\| \eta_{n,i}   \right\| \nonumber \\
& + \tau C\sum_{i=1}^s\left\| u_{n,i} - \mathcal{U}_{n,i} \right\| \sum_{i=1}^s\left| z_{n,i} - \mathcal{Z}_{n,i} \right|,
\end{align}
where  the hypothesis $\mathcal{H}_1$ has been used in the last inequality, and $\lambda_l$ is the minimum eigenvalue of $\tilde{\vec{M}}_{l}$.
Therefore, one can obtain
\begin{align}  \label{4.19}
\sum_{i=1}^s\left\| u_{n,i} - \mathcal{U}_{n,i} \right\| \le \tau C \sum_{i=1}^s\left| z_{n,i} - \mathcal{Z}_{n,i} \right| + C \sum_{i=1}^s\left\| \eta_{n,i} \right\|.
\end{align}
{Combining} it with \eqref{4.17} gives
\begin{align}  \label{4.20}
\sum_{i=1}^s\left\| \dot{u}_{n,i} - \dot{\mathcal{U}}_{n,i} \right\| \le C \sum_{i=1}^s\left| z_{n,i} - \mathcal{Z}_{n,i} \right| + \tau^{-1}C  \sum_{i=1}^s \left\| \eta_{n,i} \right\|.
\end{align}
Using both the second equations in \eqref{4.1} and \eqref{4.12} gives
\begin{align}    \label{4.21}
z_{n,i} \!- \mathcal{Z}_{n,i}  \!=\! \tau \!\sum_{i=1}^s d_{ij}^{11} \!\left(\dot{z}_{n,i} \!- \dot{\mathcal{Z}}_{n,i} \right) \!+ \sigma_{n,i} \!= \frac{\tau}{2} \!\sum_{i=1}^s d_{ij}^{11} \!\left( W(\bar{u}_{n,i}),  \dot{u}_{n,i} \!- \dot{\mathcal{U}}_{n,i} \right) \!+ \sigma_{n,i},
\end{align}
and then further using \eqref{4.20} gets
\begin{align*}
\sum_{i=1}^s\left|  z_{n,i} - \mathcal{Z}_{n,i}  \right| \le\;&  \tau C \sum_{i=1}^s\left\| \dot{u}_{n,i} - \dot{\mathcal{U}}_{n,i} \right\|  + C \sum_{i=1}^s \left| \sigma_{n,i} \right|  \nonumber \\
\overset{\eqref{4.20}}{\le}\; & C\sum_{i=1}^s\left( \left\| \eta_{n,i} \right\| + \tau  \left| z_{n,i} - \mathcal{Z}_{n,i} \right| \right) + C \sum_{i=1}^s \left| \sigma_{n,i} \right|.
\end{align*}
When $\tau$ is sufficiently small, the above inequality infers
\begin{align}  \label{4.22}
\sum_{i=1}^s\left|  z_{n,i} - \mathcal{Z}_{n,i}  \right| \le C\sum_{i=1}^s\left( \left\| \eta_{n,i} \right\|  +   \left| \sigma_{n,i} \right|  \right).
\end{align}
Combining \eqref{4.22} with the first equality in \eqref{4.21} yields
\begin{align}  \label{4.22*}
\sum_{i=1}^s\left|  \dot{z}_{n,i} - \dot{\mathcal{Z}}_{n,i}  \right| \le C\tau^{-1}\sum_{i=1}^s\left( \left\| \eta_{n,i} \right\|  +   \left| \sigma_{n,i} \right|  \right).
\end{align}
On the other hands,  it follows from both the first equations in \eqref{4.2} and \eqref{4.13} and the inequality \eqref{4.20} that
\begin{align}   \label{4.23}
& \sum_{i=1}^r \left\|\hat{u}_i(t_{n+1}) - \mathcal{U}_i^{[n+1]}\right\| \le
\tau C \sum_{i=1}^s\left\| \dot{u}_{n,i} - \dot{\mathcal{U}}_{n,i} \right\| +  C \sum_{i=1}^r \left\| \eta_i^{[n]} \right\|    \nonumber \\
\overset{\eqref{4.20}}{\le}\;& C\sum_{i=1}^s\left( \left\| \eta_{n,i} \right\| + \tau  \left| z_{n,i} - \mathcal{Z}_{n,i} \right| \right)+ C \sum_{i=1}^r \left\| \eta_i^{[n]} \right\| \nonumber \\
\overset{\eqref{4.22}}{\le} \; &C\sum_{i=1}^s\left( \left\| \eta_{n,i} \right\| +  \left| \sigma_{n,i} \right| \right)+ C \sum_{i=1}^r \left\| \eta_i^{[n]} \right\|.
\end{align}
Also, by both the second equations in \eqref{4.2} and \eqref{4.13} and the inequality \eqref{4.22*}, one can conclude
\begin{align}  \label{4.24}
\sum_{i=1}^r \left|\hat{z}_i(t_{n+1}) \!- \mathcal{Z}_i^{[n+1]}\right| \le\;&
\tau C \sum_{i=1}^s\left|  \dot{z}_{n,i} \!- \dot{\mathcal{Z}}_{n,i}  \right|+  C \sum_{i=1}^r \left| \sigma_i^{[n]} \right| \nonumber \\
\overset{\eqref{4.22*}}{\le}\;& C\sum_{i=1}^s\left( \left\| \eta_{n,i} \right\| + \left| \sigma_{n,i} \right| \right) +  C \sum_{i=1}^r \left| \sigma_i^{[n]} \right|.
\end{align}
Finally, in terms of \eqref{4.19}, \eqref{4.22}, \eqref{4.23} and \eqref{4.24}, and using Lemma \ref{Lam4.1}, we can obtain the estimates \eqref{4.15} and \eqref{4.16}  so that the proof is completed.
\end{proof}

Denote the   ``errors'' by
\[ \mathcal{E}_{n,i} =  \mathcal{U}_{n,i} - U_{n,i},~~ \mathcal{E}_i^{[n]}=  \mathcal{U}_i^{[n]} - u_i^{[n]},~~ \mathcal{D}_{n,i} = \mathcal{Z}_{n,i} - Z_{n,i},~~ \mathcal{D}_i^{[n]} =  \mathcal{Z}_i^{[n]} - z_i^{[n]}, \]
\[ E_{n,i} =  u_{n,i} - U_{n,i},~~ E_i^{[n]}=  u_i(t_n) - u_i^{[n]},~~ D_{n,i} = z_{n,i} - Z_{n,i},~~ D_i^{[n]} =  z_i(t_n) - z_i^{[n]}. \]
The  errors for the SAV-GL schemes \eqref{2.2.6}-\eqref{2.2.8} can be estimated as follows.
\begin{theorem}  \label{Thm4.3}
Under the hypothesises $\mathcal{H}_1$ and $\mathcal{H}_2$, if the GLTDs \eqref{2.1.2} are algebraically stable and diagonally stable and their generalized stage order is $\hat{q}$,  then  the SAV-GL schemes \eqref{2.2.6}-\eqref{2.2.8} have the following error estimates
\begin{align}   \label{4.25}
\sum_{i=1}^r \left(\left\|E_i^{[n+1]}\right\|^2 + \left|D_i^{[n+1]}\right|^2\right)  + \tau\sum_{i=1}^s \left(\left\|E_{n,i}\right\|^2 + \left|D_{n,i}\right|^2\right)   \le C \tau^{\min\{2\hat{q},2\nu\}},
\end{align}
when the time stepsize $\tau$ is sufficiently small.
\end{theorem}
\begin{proof}
Subtracting   \eqref{2.2.6}-\eqref{2.2.8} from \eqref{4.12}-\eqref{4.14} yields
\begin{align}  \label{4.26}
\begin{cases}
\mathcal{E}_{n,i} = \tau \sum\limits_{j=1}^s d_{ij}^{11} \dot{\mathcal{E}}_{n,i} + \sum\limits_{j=1}^r d_{ij}^{12} \hat{E}_j^{[n]},   \\[2 \jot]
\mathcal{D}_{n,i} = \tau \sum\limits_{j=1}^s d_{ij}^{11} \dot{\mathcal{D}}_{n,i} + \sum\limits_{j=1}^r d_{ij}^{12} \hat{D}_j^{[n]},~~~  i =1,2,\cdots,s,
\end{cases}
\end{align}
and
\begin{align}   \label{4.27}
\begin{cases}
\mathcal{E}_i^{[n+1]} = \tau \sum\limits_{j=1}^s d_{ij}^{21} \dot{\mathcal{E}}_{n,i} + \sum\limits_{j=1}^r d_{ij}^{22} \hat{E}_j^{[n]},  \\[2 \jot]
\mathcal{D}_i^{[n+1]} = \tau \sum\limits_{j=1}^s d_{ij}^{21} \dot{\mathcal{D}}_{n,i} + \sum\limits_{j=1}^r d_{ij}^{22} \hat{D}_j^{[n]},~~~ i =1,2,\cdots,k,
\end{cases}
\end{align}
where  $\hat{E}_i^{[n]}\!=  \hat{u}_i(t_n) \!- u_i^{[n]}, \hat{D}_i^{[n]} \!=  \hat{z}_i(t_n) \!- z_i^{[n]}$ and
\begin{align}  \label{4.28}
& \dot{\mathcal{E}}_{n,i} = -  \mathcal{L} \mathcal{E}_{n,i} - \mathcal{D}_{n,i} W(\bar{U}_i^{(n)}) - \mathcal{Z}_{n,i} \left[ W(\bar{u}_i^{(n)}) - W(\bar{U}_i^{(n)})\right],   \\
\label{4.29}
& \dot{\mathcal{D}}_{n,i} = \frac{1}{2} \left( W(\bar{u}_i^{(n)}) - W(\bar{U}_i^{(n)}),  \dot{\mathcal{U}}_{n,i}\right) + \frac{1}{2}\left( W(\bar{U}_i^{(n)}),  \dot{\mathcal{E}}_{n,i}\right).
\end{align}
Since the GLTDs \eqref{2.1.2} are algebraically stable,  there exist a symmetric positive definite matrix $\vec{G} \in \mathbb{R}^{r\times r}$ and a non-negative definite diagonal matrix $\vec{H} \in \mathbb{R}^{s\times s}$ such that the matrix $\vec{M} = (m_{ij}) \in\mathbb{R}^{(r+s)\times(r+s)}$ is non-negative definite. Hence, from both the first equations in \eqref{4.26} and  \eqref{4.27}, one can deduce
\begin{align*} 
&  \left(\vec{\mathcal{E}}^{[n+1]}, \vec{\mathcal{E}}^{[n+1]}\right)_{\vec{G}} -  \left( \hat{\vec{E}}^{[n]}, \hat{\vec{E}}^{[n]} \right)_{\vec{G}} - 2\tau \sum_{i=1}^s h_i \left(\mathcal{E}_{n,i}, \dot{\mathcal{E}}_{n,i}\right)  \nonumber \\
=\;& \sum_{i,j}^r g_{ij} \left( \mathcal{E}_i^{[n+1]}, \mathcal{E}_j^{[n+1]} \right)
 - \sum_{i,j}^r g_{ij} \left( \hat{E}_i^{[n]}, \hat{E}_j^{[n]} \right)
 - 2\tau \sum_{i=1}^s h_i \left(\mathcal{E}_{n,i}, \dot{\mathcal{E}}_{n,i}\right)  \nonumber  \\
=\;& - \sum_{i,j}^r g_{ij} \left( \hat{E}_i^{[n]}, \hat{E}_j^{[n]} \right) +  \sum_{i,j}^r g_{ij} \left( \sum_{k=1}^r d_{ik}^{22} \hat{E}_k^{[n]}, \sum_{k=1}^r d_{jl}^{22} \hat{E}_l^{[n]} \right)   \nonumber  \\
& - 2\tau \sum_{i=1}^s h_i \left( \sum_{k=1}^r d_{ik}^{12} \hat{E}_i^{[n]}, \dot{\mathcal{E}}_{n,i}\right) + 2\tau\sum_{i,j}^r g_{ij} \left( \sum_{k=1}^r d_{ik}^{22} \hat{E}_k^{[n]}, \sum_{k=1}^s d_{jl}^{21}\dot{\mathcal{E}}_{n,l} \right)   \nonumber  \\
& - 2\tau^2 \sum_{i=1}^s h_i \left(\sum_{k=1}^s d_{ik}^{11}\dot{\mathcal{E}}_{n,k} , \dot{\mathcal{E}}_{n,i}\right) + \tau^2\sum_{i,j}^r g_{ij} \left( \sum_{k=1}^s d_{ik}^{21}\dot{\mathcal{E}}_{n,k} , \sum_{k=1}^s d_{jl}^{21}\dot{\mathcal{E}}_{n,l}  \right)  \nonumber  \\
=\; & - \sum_{i,j}^r p_{ij} \left( \hat{E}_i^{[n]}, \hat{E}_j^{[n]} \right) - 2\tau \sum_{i=1}^r\sum_{j=1}^s s_{ij} \left(\hat{E}_i^{[n]}, \dot{\mathcal{E}}_{n,i}\right) - \tau^2 \sum_{i,j}^s q_{ij} \left(\dot{\mathcal{E}}_{n,i}, \dot{\mathcal{E}}_{n,i}\right),
\end{align*} 
which can be rewritten as
\begin{align}     \label{4.31}
\left(\vec{\mathcal{E}}^{[n+1]}, \vec{\mathcal{E}}^{[n+1]}\right)_{\vec{G}} = & \left( \hat{\vec{E}}^{[n]}, \hat{\vec{E}}^{[n]} \right)_{\vec{G}} + 2\tau \sum_{i=1}^s h_i \left(\mathcal{E}_{n,i}, \dot{\mathcal{E}}_{n,i}\right)- \sum_{i,j}^{r+s} m_{ij} \left( \bar{E}_{n,i},\bar{E}_{n,j} \right),
\end{align}
where 
\[  \bar{\vec{E}}_n := \left(\hat{E}_1^{[n]},\hat{E}_2^{[n]},\cdots,\hat{E}_r^{[n]},\tau \dot{\mathcal{E}}_{n,1},\tau \dot{\mathcal{E}}_{n,2}, \cdots,  \tau \dot{\mathcal{E}}_{n,s} \right)^T.
\]
Since the matrix $\vec{M}$ is non-negative definite, one has from \eqref{4.31} that
\begin{align}     \label{4.32}
\left(\vec{\mathcal{E}}^{[n+1]}, \vec{\mathcal{E}}^{[n+1]}\right)_{\vec{G}} \le  \left( \hat{\vec{E}}^{[n]}, \hat{\vec{E}}^{[n]} \right)_{\vec{G}} + 2\tau \sum_{i=1}^s h_i \left(\mathcal{E}_{n,i}, \dot{\mathcal{E}}_{n,i}\right).
\end{align}

In the following,  the mathematical induction is used to prove the  inequality \eqref{4.25} for all $1\le n\le K-1$.
Assume $1\le m\le K-1$ and  \eqref{4.25} is true for all  $n\le m$.
Let us prove \eqref{4.25} for $n = m+1$.

For evaluating the starting values by the extrapolation,  it holds for $n\le m$ that
\[ \left\| \bar{u}_{n,i} - \bar{U}_{n,i}\right\| \le C \left( \sum_{i=1}^s \left\|E_{n-1,i} \right\| + \sum_{i=1}^r \left\|\hat{E}_i^{[n]} \right\|\right), \]
which further implies  that
\begin{align}   \label{4.33}
\left\|W(\bar{u}_{n,i}) - W(\bar{U}_{n,i}) \right\| \le C\left( \sum_{i=1}^s \left\|E_{n-1,i} \right\| + \sum_{i=1}^r \left\|\hat{E}_i^{[n]} \right\|\right),
\end{align}
since the function $W$ is locally Lipschitz continuous. In terms of \eqref{4.33} and the boundedness of the quantity $\bar{U}_{n,i}$ for $n\le m$ due to the induction assumption,  we can obtain
\begin{align}    \label{4.34}
\left\|  \mathcal{D}_{n,i} W(\bar{U}_{n,i}) \!+\! \mathcal{Z}_{n,i}\!\left[ W(\bar{u}_{n,i}) \!-\! W(\bar{U}_{n,i})\right] \!\right\|  \!\le\!  C \! \left(\!\left|\mathcal{D}_{n,i} \right|  \!+\! \sum_{i=1}^s \!\left\|E_{n-1,i} \right\| \!+\! \sum_{i=1}^r \!\left\|\hat{E}_i^{[n]} \right\|\!\right)\!.
\end{align}
Using \eqref{4.28} and \eqref{4.34} gives
\begin{align}   \label{4.35}
& \left(\mathcal{E}_{n,i}, \dot{\mathcal{E}}_{n,i}\right) \overset{\eqref{4.28}}{=} - \!\left(\mathcal{E}_{n,i},  \mathcal{L} \mathcal{E}_{n,i}\right)  \!-\! \left(\mathcal{E}_{n,i}, \mathcal{D}_{n,i} W(\bar{U}_{n,i}) \!+\! \mathcal{Z}_{n,i} \left[ W(\bar{u}_{n,i}) \!- W(\bar{U}_{n,i})\right]\right),  \nonumber \\
\overset{\eqref{4.34}}{\le} & - \left(\mathcal{E}_{n,i},  \mathcal{L} \mathcal{E}_{n,i}\right) \!+ \frac{1}{2}\left\| \mathcal{E}_{n,i}\right\|^2 \!+ C\left( \left|\mathcal{D}_{n,i} \right|^2 \!+ \sum_{i=1}^s \left\|E_{n-1,i} \right\|^2 \!+ \sum_{i=1}^r \left\|\hat{E}_i^{[n]} \right\|^2\right).
\end{align}
Substituting it into \eqref{4.32} yields for $n\le m$ that
\begin{align}     \label{4.36}
& \left\|\vec{\mathcal{E}}^{[n+1]}\right\|_{\vec{G}}^2 + 2\tau\sum_{i=1}^s h_i \left(\mathcal{E}_{n,i},  \mathcal{L} \mathcal{E}_{n,i}\right) \nonumber \\
\le \;&  \left\| \hat{\vec{E}}^{[n]} \right\|_{\vec{G}}^2 +    C \tau \sum_{i=1}^s \left( \left|\mathcal{D}_{n,i} \right|^2 + \left\| \mathcal{E}_{n,i}\right\|^2 + \left\|E_{n-1,i} \right\|^2\right)  + C\tau\sum_{i=1}^r \left\|\hat{E}_i^{[n]} \right\|^2.
\end{align}
Similarly,   using both the second equations in \eqref{4.26} and \eqref{4.27}  can derive
\begin{align}     \label{4.37}
\left\|\vec{\mathcal{D}}^{[n+1]}\right\|_{\vec{G}}^2 \le  \left\| \hat{\vec{D}}^{[n]} \right\|_{\vec{G}}^2 \!+    C \tau \sum_{i=1}^s \!\left( \left|\mathcal{D}_{n,i} \right|^2 \!+ \left\| \mathcal{E}_{n,i}\right\|^2 \!+ \left\|E_{n-1,i} \right\|^2\right) \! + C\tau\sum_{i=1}^r \left\|\hat{E}_i^{[n]} \right\|^2.
\end{align}
Combining \eqref{4.36} with \eqref{4.37} gives for $n\le m$ that
\begin{align}     \label{4.38}
& \left\|\vec{\mathcal{E}}^{[n+1]}\right\|_{\vec{G}}^2 +  \left\|\vec{\mathcal{D}}^{[n+1]}\right\|_{\vec{G}}^2  + 2\tau\sum_{i=1}^s h_i \left(\mathcal{E}_{n,i},  \mathcal{L} \mathcal{E}_{n,i}\right) \nonumber \\
\le \;& \! \left\| \hat{\vec{E}}^{[n]} \right\|_{\vec{G}}^2  + \left\| \hat{\vec{D}}^{[n]} \right\|_{\vec{G}}^2 + C\tau\left[ \sum_{i=1}^r \left\|\hat{E}_i^{[n]} \right\|^2  \!+ \sum_{i=1}^s \!\left( \left|\mathcal{D}_{n,i} \right|^2 \!+ \left\| \mathcal{E}_{n,i}\right\|^2 \!+ \left\|E_{n-1,i} \right\|^2\right)\right].
\end{align}
On the other hand,   testing the first relation in \eqref{4.26} with $\mathcal{E}_{n,i}$ yields
\begin{align*}
\sum_{i=1}^s \left\|  \mathcal{E}_{n,i}\right\|^2 \le \;& C\sum_{i=1}^r \left\| \hat{E}_i^{[n]}\right\| ^2 + C\tau \sum_{i,j=1}^s d_{ij}^{11}  \left(\mathcal{E}_{n,i}, \dot{\mathcal{E}}_{n,j}\right).
\end{align*}
Using \eqref{4.28} and \eqref{4.34} gives
\begin{align*}
& \sum_{i,j=1}^s \!d_{ij}^{11} \! \left(\mathcal{E}_{n,i}, \dot{\mathcal{E}}_{n,i}\right) \!
\overset{\eqref{4.28}}{=}\! -\!\sum_{i,j=1}^s \!d_{ij}^{11} \! \left(\mathcal{E}_{n,i}, \mathcal{L}\mathcal{E}_{n,j}\!\! +\! \mathcal{D}_{n,j} W(\bar{U}_{n,j}) \!+\! \mathcal{Z}_{n,j}\! \!\left[ W(\bar{u}_{n,j})\! - W(\bar{U}_{n,j})\right]\right)    \nonumber \\
\overset{{\eqref{4.34}}}{\le} & C\sum_{i=1}^s \left[ \left(\mathcal{E}_{n,i},  \mathcal{L} \mathcal{E}_{n,i}\right) \!+ \left|\mathcal{D}_{n,i} \right|^2  \!+ \left\| \mathcal{E}_{n,i}\right\|^2 \!+  \left\|E_{n-1,i} \right\|^2 \right] + C \sum_{i=1}^r \left\|\hat{E}_i^{[n]} \right\|^2.
\end{align*}
Thus, combining the last two inequalities derives
\begin{align}   \label{4.39}
\sum_{i=1}^s \left\|  \mathcal{E}_{n,i}\right\|^2
\le \;&  C\tau\sum_{i=1}^s \left[ \left(\mathcal{E}_{n,i},  \mathcal{L} \mathcal{E}_{n,i}\right) \!+ \left|\mathcal{D}_{n,i} \right|^2  \!+ \left\| \mathcal{E}_{n,i}\right\|^2 \!+  \left\|E_{n-1,i} \right\|^2 \right] \nonumber \\
& + C(1+\tau)\sum_{i=1}^r \left\|\hat{E}_i^{[n]} \right\|^2.
\end{align}
Also, in a similar way, we   obtain
\begin{align}   \label{4.40}
\sum_{i=1}^s \left|  \mathcal{D}_{n,i}\right|^2
\le\;&   C \tau \sum_{i=1}^s \left[ \left(\mathcal{E}_{n,i},  \mathcal{L} \mathcal{E}_{n,i}\right) \!+\left|\mathcal{D}_{n,i} \right|^2 + \left\| \mathcal{E}_{n,i}\right\|^2 + \left\|E_{n-1,i} \right\|^2\right] \nonumber \\
& + C\sum_{i=1}^r \left|\hat{D}_i^{[n]} \right|^2 + C\tau\sum_{i=1}^r \left\|\hat{E}_i^{[n]} \right\|^2.
\end{align}
Summing up  \eqref{4.39} and \eqref{4.40}    yields
\begin{align}    \label{4.41}
& \sum_{i=1}^s \!\left(\left\|  \mathcal{E}_{n,i}\right\|^2  \!+\! \left|  \mathcal{D}_{n,i}\right|^2\right)  \!\le\! C\sum_{i=1}^r \!\left( \left\|\hat{E}_i^{[n]} \right\|^2 \!+\! \left|\hat{D}_i^{[n]} \right|^2 \right) \!+\!  C \tau \sum_{i=1}^s \!\left[\left(\mathcal{E}_{n,i},  \mathcal{L} \mathcal{E}_{n,i}\right) \!+\! \left\|E_{n-1,i} \right\|^2 \right]  \nonumber \\
\le\; & C\left( \left\|\vec{\hat{E}}^{[n]} \right\|_{\vec{G}}^2 + \left\|\vec{\hat{D}}^{[n]} \right\|_{\vec{G}}^2\right) +  C \tau \sum_{i=1}^s \left[ h_i\left(\mathcal{E}_{n,i},  \mathcal{L} \mathcal{E}_{n,i}\right) \!+ \left\|E_{n-1,i} \right\|^2\right],
\end{align}
for sufficiently small $\tau$, where the equivalence between the weighted norm and the $L^2$ norm and the positivity of the weights $h_1,h_2,\cdots,h_s$ are used in the second inequality.
Inserting \eqref{4.41} to \eqref{4.38} gives
\begin{align}     \label{4.42}
& \left\|\vec{\mathcal{E}}^{[n+1]}\right\|_{\vec{G}}^2 +  \left\|\vec{\mathcal{D}}^{[n+1]}\right\|_{\vec{G}}^2  + 2\tau\sum_{i=1}^s h_i \left(\mathcal{E}_{n,i},  \mathcal{L} \mathcal{E}_{n,i}\right) \nonumber \\
\le \;& \!(1\!+\!C_1\tau)\left(\left\| \hat{\vec{E}}^{[n]} \right\|_{\vec{G}}^2  \!+\! \left\| \hat{\vec{D}}^{[n]} \right\|_{\vec{G}}^2\right) \!+  C_1\tau^2 \sum_{i=1}^s h_i\left(\mathcal{E}_{n,i},  \mathcal{L} \mathcal{E}_{n,i}\right) \!+  C_1\tau \sum_{i=1}^s \left\|E_{n-1,i}\right\|^2,
\end{align}
with a constant $C_1>0$. Multiplying \eqref{4.41} by $2C_1\tau$ and adding to \eqref{4.42} gives
\begin{align*}
& \left\|\vec{\mathcal{E}}^{[n+1]}\right\|_{\vec{G}}^2 +  \left\|\vec{\mathcal{D}}^{[n+1]}\right\|_{\vec{G}}^2  + 2\tau\sum_{i=1}^s h_i \left(\mathcal{E}_{n,i},  \mathcal{L} \mathcal{E}_{n,i}\right) + 2C_1\tau \sum_{i=1}^s \left(\left\|  \mathcal{E}_{n,i}\right\|^2  + \left|  \mathcal{D}_{n,i}\right|^2\right) \nonumber \\
\le \;& \!(1\!+\!C_2\tau)\left(\left\| \hat{\vec{E}}^{[n]} \right\|_{\vec{G}}^2  + \left\| \hat{\vec{D}}^{[n]} \right\|_{\vec{G}}^2\right) \!+  C_2\tau^2 \sum_{i=1}^s h_i\left(\mathcal{E}_{n,i},  \mathcal{L} \mathcal{E}_{n,i}\right)\! +  (C_1\!+\!C_2\tau)\tau \sum_{i=1}^s \left\|E_{n-1,i}\right\|^2,
\end{align*}
with a constant $C_2>0$. For sufficiently small $\tau$, the term $C_2\tau^2 \sum\limits_{i=1}^s h_i\left(\mathcal{E}_{n,i},  \mathcal{L} \mathcal{E}_{n,i}\right)$ can be absorbed by the left-hand side, and $C_1\!+\!C_2\tau \le 2C_1$. Hence, the above inequality is reduced to
\begin{align}   \label{4.44}
& \left\|\vec{\mathcal{E}}^{[n+1]}\right\|_{\vec{G}}^2 +  \left\|\vec{\mathcal{D}}^{[n+1]}\right\|_{\vec{G}}^2  + 2\tau\sum_{i=1}^s h_i \left(\mathcal{E}_{n,i},  \mathcal{L} \mathcal{E}_{n,i}\right) + 2C_1\tau \sum_{i=1}^s \left(\left\|  \mathcal{E}_{n,i}\right\|^2  + \left|  \mathcal{D}_{n,i}\right|^2\right) \nonumber \\
\le \;& \!(1\!+\!C_2\tau)\left[\left\| \hat{\vec{E}}^{[n]} \right\|_{\vec{G}}^2  + \left\| \hat{\vec{D}}^{[n]} \right\|_{\vec{G}}^2 \!+  2C_1\tau \sum_{i=1}^s \left(\left\|E_{n-1,i}\right\|^2 \!+\left|D_{n-1,i}\right|^2  \right) \right].
\end{align}
On the other hand, it follows from the Cauchy inequality that
\begin{align}   \label{4.45}
\left\|\hat{\vec{E}}^{[n+1]}\right\|_{\vec{G}}^2 + \left\|\hat{\vec{D}}^{[n+1]}\right\|_{\vec{G}}^2 \le \;& \left( 1\!+\!\frac{1}{\tau}\right)\left[  \left\| \hat{\vec{u}}(t_{n+1}) -\vec{\mathcal{U}}^{[n+1]}\right\|_{\vec{G}}^2  + \left\| \hat{\vec{z}}(t_{n+1}) -\vec{\mathcal{Z}}^{[n+1]}\right\|_{\vec{G}}^2\right],  \nonumber \\
& + (1\!+\!\tau)\left[ \left\|\vec{\mathcal{E}}^{[n+1]}\right\|_{\vec{G}}^2 + \left\|\vec{\mathcal{D}}^{[n+1]}\right\|_{\vec{G}}^2\right],
\end{align}
which implies by Theorem  \ref{Thm4.1} that
\begin{align}   \label{4.46}
\left\|\hat{\vec{E}}^{[n+1]}\right\|_{\vec{G}}^2 + \left\|\hat{\vec{D}}^{[n+1]}\right\|_{\vec{G}}^2 \le (1\!+\!\tau)\left( \left\|\vec{\mathcal{E}}^{[n+1]}\right\|_{\vec{G}}^2 + \left\|\vec{\mathcal{D}}^{[n+1]}\right\|_{\vec{G}}^2\right) + C \tau^{\min\{2\hat{q}+1,2\nu+1\}}.
\end{align}
Also,  we can obtain
\begin{align}   \label{4.47}
\sum_{i=1}^s \!\left(\left\|E_{n,i}\right\|^2 \!+ \left|D_{n,i}\right|^2 \right) \le (1\!+\!\tau)\sum_{i=1}^s\!\left( \left\|\mathcal{E}_{n,i}\right\|^2 \!+ \left|\mathcal{D}_{n,i}\right|^2\right) \!+ C\tau^{\min\{2\hat{q}+1,2\nu+1\}}.
\end{align}
Multiplying \eqref{4.47} by $2C_1\tau$ and adding to \eqref{4.46} yields
\begin{align}    \label{4.48}
& \left\|\hat{\vec{E}}^{[n+1]}\right\|_{\vec{G}}^2 +  \left\|\hat{\vec{D}}^{[n+1]}\right\|_{\vec{G}}^2 + 2C_1\tau  \sum_{i=1}^s \left(\left\|  E_{n,i}\right\|^2  + \left|  D_{n,i}\right|^2\right)  \nonumber  \\
\le\;& (1\!+\!\tau) \!\left[ \left\|\vec{\mathcal{E}}^{[n+1]}\right\|_{\vec{G}}^2 + \left\|\vec{\mathcal{D}}^{[n+1]}\right\|_{\vec{G}}^2 + 2C_1\tau \sum_{i=1}^s\!\left( \left\|\mathcal{E}_{n,i}\right\|^2 \!+ \left|\mathcal{D}_{n,i}\right|^2\right) \!\right] \! + C_3 \tau^{\min\{2\hat{q}+1,2\nu+1\}}   \nonumber \\
\overset{\eqref{4.44}}{\le} \;&  (1\!+\!C_3\tau)  \left[ \left\|\hat{\vec{E}}^{[n]} \right\|_{\vec{G}}^2 + \left\|\hat{\vec{D}}^{[n]} \right\|_{\vec{G}}^2 + 2C_1\tau \sum_{i=1}^s \left( \left\|E_{n-1,i} \right\|^2  + \left|D_{n-1,i} \right|^2\right)\right]   \nonumber  \\
& + C_3 \tau^{\min\{2\hat{q}+1,2\nu+1\}},
\end{align}
with some positive constant $C_3$. According to the sum formula of the geometric sequence and
the common inequality $(1+a)^n\le \exp(na)$, $\forall a\ge 0$, an induction to \eqref{4.48}  concludes
\begin{align}   \label{4.49}
& \left\|\hat{\vec{E}}^{[n+1]}\right\|_{\vec{G}}^2 +  \left\|\hat{\vec{D}}^{[n+1]}\right\|_{\vec{G}}^2 + 2C_1\tau  \sum_{i=1}^s \left(\left\|  E_{n,i}\right\|^2  + \left|  D_{n,i}\right|^2\right)   \nonumber \\
\le \;& C_4 \left[ \left\|\hat{\vec{E}}^{[1]}\right\|_{\vec{G}}^2 +  \left\|\hat{\vec{D}}^{[1]}\right\|_{\vec{G}}^2 + 2C_1\tau  \sum_{i=1}^s \left(\left\|  E_{0,i}\right\|^2  + \left|  D_{0,i}\right|^2\right) \right] + C_4\tau^{\min\{2\hat{q},2\nu\}}.
\end{align}
Considering the definition of the generalized stage order gives
\[  u_i(t_n) - \hat{U}_i(t_n) = \mathcal{O}(h^{\hat{q}}),~~~ z_i(t_n) - \hat{z}_i(t_n) = \mathcal{O}(h^{\hat{q}}),~~~ i =1,2,\cdots,r. \]
Finally, by \eqref{4.49} and the commonly used triangle inequality, it can be deduced
\begin{align}   \label{4.50}
& \left\|\vec{E}^{[n+1]}\right\|_{\vec{G}}^2 +  \left\|\vec{D}^{[n+1]}\right\|_{\vec{G}}^2 + 2C_1\tau  \sum_{i=1}^s \left(\left\|  E_{n,i}\right\|^2  + \left|  D_{n,i}\right|^2\right)   \nonumber \\
\le \;& C_5 \left[ \left\|\vec{E}^{[1]}\right\|_{\vec{G}}^2 +  \left\|\vec{D}^{[1]}\right\|_{\vec{G}}^2 + 2C_1\tau  \sum_{i=1}^s \left(\left\|  E_{0,i}\right\|^2  + \left|  D_{0,i}\right|^2\right) \right] + C_5\tau^{\min\{2\hat{q},2\nu\}}.
\end{align}
Thanks to the equivalence between the weighted norm and the $L^2$ norm,
  \eqref{4.50} implies  \eqref{4.25} for $n=m+1$.
Therefore, by the mathematical induction, the estimate \eqref{4.25} holds for all $1\le n\le K-1$. The proof is completed.
\end{proof}

When both the stage order  and method order of the GLTDs  \eqref{2.1.2} are $q$, their generalized stage orders are at least $\hat{q} = q$. Hence, Theorem \ref{Thm4.3} implies the following result directly.

\begin{corollary}   \label{Cor4.4}
Under the hypothesises $\mathcal{H}_1$ and $\mathcal{H}_2$, if the GLTDs \eqref{2.1.2} are algebraically stable and diagonally stable, their both  stage order and   method order are  $q$, then the discrete solutions derived by the schemes \eqref{2.2.6}-\eqref{2.2.8} satisfy
\begin{align}   \label{4.57}
\sum_{i=1}^r \left(\left\|E_i^{[n+1]}\right\|^2 + \left|D_i^{[n+1]}\right|^2\right)  + \tau\sum_{i=1}^s \left(\left\|E_{n,i}\right\|^2 + \left|D_{n,i}\right|^2\right)   \le C \tau^{\min\{2q,2\nu\}},
\end{align}
when the time stepsize $\tau$ is sufficiently small.
\end{corollary}

In the following, we further investigate when the convergence orders of the SAV-GL schemes \eqref{2.2.6}-\eqref{2.2.8} are one higher than the stage order of the GLTDs.  Suppose that the GLTDs \eqref{2.1.2} have the stage order $q$ and method order $p=q+1$, then it follows that
\begin{align}   \label{4.53}
\sum_{i=1}^s \!\left( \left\|\eta_{n,i}\right\| + \left|\sigma_{n,i} \right|\right) \le C \tau^{\min\{q+1,\nu+1\}},~~ \sum_{i=1}^r \!\left( \left\|\eta_i^{[n]} \right\| + \left|\sigma_i^{[n]} \right|\right) \le C \tau^{\min\{q+2,\nu+1\}},
\end{align}
where the local errors $\eta_{n,i}, \eta_i^{[n]}, \sigma_{n,i}$ and $\sigma_i^{[n]}$ are defined by \eqref{4.1} and \eqref{4.2} with $\hat{u}_i(t_n) = u_i(t_n)$ and $\hat{z}_i(t_n)= z_i(t_n)$ for $i=1,2,\ldots, r$. Moreover, if the condition \eqref{2.1.6} holds, we have
\begin{align}
\label{4.54}
\eta_{n,i} \!- \kappa\tau^{q+1} u^{(q+1)}(t_n)  \!=\! \mathcal{O}(\tau^{q+2}),~\sigma_{n,i} \!- \kappa\tau^{q+1} z^{(q+1)}(t_n)  \!=\! \mathcal{O}(\tau^{q+2}),~~i\!=\!1,2,\cdots,s.
\end{align}
Hence, in \eqref{4.1}-\eqref{4.2}, we can take
\[ \hat{u}_i(t_n)  = u_i(t_n) + w_{i0} \kappa\tau^{q+1} u^{(q+1)}(t_n),~~~ \ \hat{z}_i(t_n)  = z_i(t_n) + w_{i0} \kappa\tau^{q+1} z^{(q+1)}(t_n), \]
such that
\begin{align*}
\sum_{i=1}^s \left( \left\|\eta_{n,i}\right\| + \left|\sigma_{n,i} \right|\right) \le C \tau^{\min\{q+2,\nu+1\}}.
\end{align*}
Therefore, by Theorem \ref{Thm4.3}, the following result is derived.

\begin{theorem}  \label{Thm4.4}
Under the hypothesises $\mathcal{H}_1$ and $\mathcal{H}_2$, if the GLTDs \eqref{2.1.2} are algebraically stable and diagonally stable, their stage order and method order are  $q$ and $q+1$, respectively, and the condition \eqref{2.1.6} holds, then the discrete solutions derived by the schemes \eqref{2.2.6}-\eqref{2.2.8} satisfy
\begin{align}   \label{4.56}
\sum_{i=1}^r \left(\left\|E_i^{[n+1]}\right\|^2 + \left|D_i^{[n+1]}\right|^2\right)  + \tau\sum_{i=1}^s \left(\left\|E_{n,i}\right\|^2 + \left|D_{n,i}\right|^2\right)   \le C \tau^{\min\{2q+2,2\nu\}},
\end{align}
when the time stepsize $\tau$ is sufficiently small.
\end{theorem}

\begin{remark}
The above result shows the advantage of the generalized stage order that
 the convergence orders of the SAV-GL schemes \eqref{2.2.6}-\eqref{2.2.8} may be one higher than the stage order of the GLTDs \eqref{2.1.2}.
\end{remark}

\subsection{Applications to the one-leg and MRK time discretization}

This subsection presents the convergence results for the special SAV-GL schemes  \eqref{3.2.1}, \eqref{3.2.9} and \eqref{3.2.13}-\eqref{3.2.15}
  as practical applications of Theorems \ref{Thm4.3} and \ref{Thm4.4}.

For the SAV-GL  scheme  \eqref{3.2.1}, where $\nu=2$, it can be checked that  the one-step one-leg time
discretization \eqref{2.1.8} has the generalized stage order  $\hat{q} = 2$ for $\theta = \frac{1}{2}$ and $\hat{q}=1$ for $\frac{1}{2}< \theta \le 1$ when it is written as a GLTD, so that using Theorem \ref{Thm4.3} can give the following results.

\begin{theorem}
Under the hypothesises $\mathcal{H}_1$ and $\mathcal{H}_2$,  if the time stepsize $\tau$ is sufficiently small, then the scheme \eqref{3.2.1}   has the following error estimates
\begin{align}   \label{4.2.2}
\left\|u(\cdot,t_{n+1}) - u^{n+1}\right\|^2 + \left|z(t_{n+1}) - z^{n+1}\right|^2   \le C \tau^2,
\end{align}
for $\frac{1}{2}<\theta \le 1$, and
\begin{align}   \label{4.2.3}
\left\|u(\cdot,t_{n+1}) - u^{n+1}\right\|^2 + \left|z(t_{n+1}) - z^{n+1}\right|^2   \le C \tau^4,
\end{align}
for  $\theta = \frac{1}{2}$.
\end{theorem}

\begin{remark}
The inequality \eqref{4.2.2} can be derived by Corollary \ref{Cor4.4}, since  both the stage order and method order of the  one-step one-leg time discretization \eqref{2.1.8} with $\frac{1}{2}<\theta \le 1$  are 1 when it is written as a GLTD.
However, for $\theta = \frac{1}{2}$, the stage order and method order of  \eqref{2.1.8}  are respectively 1 and 2 and the condition \eqref{2.1.6} holds so that  the inequality  \eqref{4.2.3} can be derived by Theorem \ref{Thm4.4}.
\end{remark}

\begin{remark}
An   error estimate was also derived in  \cite{ShenJ18b} for the scheme \eqref{3.2.1} with $\theta = 1$ and $\frac{1}{2}$, respectively. Although that result  can be extended to the scheme \eqref{3.2.1} for any $\theta\in [\frac{1}{2},1]$, it seems quite difficult to obtain the rigorous error estimate for the general SAV-GL schemes \eqref{2.2.6}-\eqref{2.2.8}.
\end{remark}

 For the SAV-GL  scheme \eqref{3.2.9}, where the number of extrapolation points is two,
 a simple calculation can show that the two-step one-leg time discretization \eqref{2.1.9} has the generalized stage order $\hat{q} = 2$ so that in terms of Theorem \ref{Thm4.3}, the following optimal error estimate can be obtained.

\begin{theorem}
Under the hypothesises $\mathcal{H}_1$ and $\mathcal{H}_2$, if the time stepsize $\tau$ is sufficiently small,  the scheme  \eqref{3.2.9} with $\gamma\ge 0$ and $\delta>0$ satisfies
\begin{align}   \label{4.2.4}
\left\|u(\cdot,t_{n+1}) - u^{n+1}\right\|^2 + \left|z(t_{n+1}) - z^{n+1}\right|^2   \le C \tau^4.
\end{align}
\end{theorem}
\begin{remark}
The estimate \eqref{4.2.4}  can also be derived by Theorem \ref{Thm4.4}, since  the stage order and the method order of the time discretization \eqref{2.1.9} are 1 and 2, respectively when it is written as a GLTD, and the condition \eqref{2.1.6} holds.
\end{remark}


For the MRK time discretization \eqref{2.1.10}, due to the simplified conditions $B(s)$ and $C(s)$,
its generalized stage order $\hat{q}$ is at least 2,
when the integer $s = 1$ and the number of extrapolation points $\nu = 2$,
while it is at least $ s$  when the integer $s \ge 2$ and the number of extrapolation points $\nu = s$.
Therefore,  using Theorem \ref{Thm4.3} can
 conclude the following error estimate for the SAV-GL scheme \eqref{3.2.13}-\eqref{3.2.15}.

\begin{theorem}   \label{Thm4.8}
Under the hypothesises $\mathcal{H}_1$ and $\mathcal{H}_2$,  if the MRK time discretization is algebraically stable and diagonally stable, and the conditions $B(s)$ and $C(s)$ hold, then   the SAV-GL  scheme \eqref{3.2.13}-\eqref{3.2.15} satisfies the following error estimates
\begin{align}   \label{4.2.8}
\left\|u(\cdot,t_{n+1}) - u^{n+1}\right\|^2 + \left|z(t_{n+1}) - z^{n+1}\right|^2 \le C \tau^{4}, ~~\mbox{for}~~ s = 1,
\end{align}
and
\begin{align}   \label{4.2.9}
\left\|u(\cdot,t_{n+1}) - u^{n+1}\right\|^2 + \left|z(t_{n+1}) - z^{n+1}\right|^2 \le C \tau^{2s},~~\mbox{for}~~ s \ge 2,
\end{align}
when the time stepsize $\tau$ is sufficiently small.
\end{theorem}

\begin{remark}
The error estimates \eqref{4.2.8} and \eqref{4.2.9} of the SAV-GL schemes \eqref{3.2.13}-\eqref{3.2.15} can be reduced to those  for the Allen-Cahn equation in \cite{Akrivis19}, which was deduced with a different technique.
\end{remark}

 \begin{remark}  \label{remark4.6}
The inequality \eqref{4.2.9} shows that the SAV-GL schemes \eqref{3.2.13}-\eqref{3.2.15} with the integer $s\ge 2$  and the number of extrapolation points $\nu = s$ are convergent with the order of $s$ for the gradient flows. However, numerical experiments in Section \ref{sec:6} will show that the SAV-GL schemes \eqref{3.2.13}-\eqref{3.2.15} can be convergent with the order of $s+1$
 when adding an extrapolation point, i.e., $\nu = s+1$.
\end{remark}

\section{Spatial discretization}  \label{sec:5}

This section introduces the Fourier spectral spatial discretization of {the SAV-GL schemes \eqref{2.2.6}-\eqref{2.2.8}}  to derive
the fully discrete SAV-GL schemes for the  gradient flows with the periodic boundary conditions in order to conduct our numerical validation in next section.
Because the proof of the energy stability in Section \ref{sec:3}  is variational and the energy stability is available for the boundary conditions which make all boundary terms  disappear when the integration by parts is performed,
the above results on the energy stability can be straightforwardly extended to the fully discrete SAV-GL schemes with the Galerkin finite element  or  the spectral methods  or  the finite
difference methods, satisfying the summation by parts for the spatial discretization.

Assume that the domain $\Omega = (0,L)\times(0,L)$  is uniformly partitioned into
\[  \Omega_h = \{(x_i,y_j)|x_i = ih,y_j = jh,~0\le i,j\le N-1\},  \]
with $h = \frac{L}{N}$ and $N\in \mathbb{Z}^+$  (assumed even).
Temporarily ignore the time dependence of function $u$ etc in {\eqref{2.2.6}-\eqref{2.2.8}}.
The Fourier spectral spatial discretization is a function-space method that approximates an arbitrary discrete periodic function $u(x_i,y_j)$ defined on $\Omega_h$ by a
finite sum of $N^2$ complex exponentials

\begin{equation*}
 u(x_i,y_j) \approx u_N(x_i,y_j)  =  \sum_{m,l=-{N}/{2}}^{{N}/{2}-1} \widehat{u}_{m,l} e^{\imath \xi_m x_i}e^{\imath \eta_l y_j},\quad 0\le i,j\le N-1,
 \end{equation*}
where   $\imath = \sqrt{-1}$, $\xi_m = 2\pi m/L$, $\eta_l = 2\pi l/L$,  and
$\widehat{u}_{m,l}$ are  the (discrete) Fourier coefficients calculated by the discrete Fourier transform
$$ \widehat{u}_{m,l} = \frac{1}{N^2} \sum_{i,j=0}^{N-1} u(x_i,y_j) e^{-\imath(\xi_m x_i + \eta_l y_j)}.
$$

Let $ \mathcal{V}_h := \left\{\vec{v}=(v_{ij}), v_{ij}\in\mathbb R,0\le i,j\le N-1\right\}$ be the grid function space defined on $\Omega_h$, and assume that $\vec{u}_i^{[n]}\in\mathcal{V}_h$ and $z_i^{[n]}\in \mathbb{R}$ are given for $i=1,2,\cdots,r$.
Applying the discrete Fourier transform  to the semi-discrete SAV-GL schemes \eqref{2.2.6}-\eqref{2.2.8} yields
\begin{align}  \label{6.1}
\begin{cases}
\widehat{\vec{U}}_{n,i} = \tau \sum\limits_{j=1}^s d_{ij}^{11} \dot{\widehat{\vec{U}}}_{n,j} + \sum\limits_{j=1}^r d_{ij}^{12} \widehat{\vec{u}}_j^{[n]},  \\[2 \jot]
Z_{n,i} = \tau \sum\limits_{j=1}^s d_{ij}^{11} \dot{Z}_{n,j} + \sum\limits_{j=1}^r d_{ij}^{12} z_j^{[n]},~~~  i =1,2,\cdots,s,
\end{cases}
\end{align}
and
\begin{align}   \label{6.2}
\begin{cases}
\widehat{\vec{u}}_i^{[n+1]} = \tau \sum\limits_{j=1}^s d_{ij}^{21} \dot{\widehat{\vec{U}}}_{n,j} + \sum\limits_{j=1}^r d_{ij}^{22} \widehat{\vec{u}}_j^{[n]},  \\[2 \jot]
z_i^{[n+1]} = \tau \sum\limits_{j=1}^s d_{ij}^{21} \dot{Z}_{n,j} + \sum\limits_{j=1}^r d_{ij}^{22} z_j^{[n]},~~~ i =1,2,\cdots,r,
\end{cases}
\end{align}
where $\widehat{\vec{U}}_{n,i}$  and $\widehat{\vec{u}}_i^{[n]}$ are the discrete  Fourier coefficients of $\vec{U}_{n,i}$ and $\vec{u}_i^{[n]} \in \mathcal{V}_h$, respectively,  $Z_{n,i}$ and $z_i^{[n+1]}$ are used as the same as the symbols in the semi-discrete scheme \eqref{2.2.6}-\eqref{2.2.8}, since those quantities are not changed when applying the discrete Fourier transform,  and
\begin{align}  \label{6.3}
\dot{\widehat{\vec{U}}}_{n,i} \!=\! \mathcal{G}_h\!\circ\!\widehat{\vec{\mu}}_{n,i}, ~\widehat{\vec{\mu}}_{n,i} \!=\! \mathcal{L}_h\!\circ\!\widehat{\vec{U}}_{n,i} + Z_{n,i} \widehat{W}(\bar{\vec{U}}_{n,i}),~\dot{Z}_{n,i} \!=\! \frac{1}{2} \left\langle \widehat{W}(\bar{\vec{U}}_{n,i}),  \dot{\widehat{\vec{U}}}_{n,i}\right\rangle, ~ i \!=\! 1,2,\cdots,s,
\end{align}
where ``$\circ$'' denotes the Shur product symbol, $\left\langle\cdot,\cdot\right\rangle$ is the discrete $L^2$  inner product defined by $\left\langle\vec{\phi},\vec{\psi}\right\rangle := h^2\sum\limits_{i,j=0}^{N-1} \vec{\phi}_{ij}\vec{\psi}_{ij} $ for any $\vec{\phi},\vec{\psi} \in \mathcal{V}_h$, and
$\mathcal{G}_h$ and $\mathcal{L}_h$ are the analytical formulas of the operators $\mathcal{G}$ and $\mathcal{L}$ in the discrete Fourier space.
Specifically, $\mathcal{G}_h$ is a $N\times N$ matrix with the elements $\mathcal{G}_h(m,l) = -(\xi_m^2+\eta_l^2)$ (resp. $-1$) for the $H^{-1}$ (resp. $L^2$) gradient flow, and $\mathcal{L}_h$ also is a $N\times N$ matrix with the elements $\mathcal{L}_h(m,l) = -\alpha(\xi_m^2+\eta_l^2) + \beta$ for the case of $\mathcal{L} = \alpha \Delta + \beta$ with $\alpha>0,\beta\ge 0$.
Once $\{\widehat{\vec{U}}_{n,i}$, $\widehat{\vec{u}}_i^{[n+1]}\}$ are known by solving \eqref{6.1}-\eqref{6.3}, one can compute the numerical solutions $\{\vec{U}_{n,i}$, $\vec{u}_i^{[n+1]} \in \mathcal{V}_h\}$   by using the discrete inverse Fourier transform.
Similar to the semi-discrete SAV-GL schemes \eqref{2.2.6}-\eqref{2.2.8}, the fully discrete schemes \eqref{6.1}-\eqref{6.3} are also unconditionally energy stable.

\begin{theorem}  \label{Thm6.1}
If the GL time discretizations \eqref{2.1.2} are algebraically stable with a symmetric and positive definite matrix $\vec{G}=(g_{ij})\in\mathbb{R}^{r\times r}$, then
the fully discrete SAV-GL schemes \eqref{6.1}-\eqref{6.3} satisfy the following energy decay property
\begin{align}   \label{6.5}
\frac{1}{2} \sum_{i,j}^r g_{ij} \left\langle \mathcal{L}_h\circ \widehat{\vec{u}}_i^{[n+1]}, \widehat{\vec{u}}_j^{[n+1]}\right\rangle + \left\| \vec{z}^{[n+1]}\right\|_{\vec{G}}^2  \le
\frac{1}{2} \sum_{i,j}^r g_{ij}\left\langle \mathcal{L}_h \circ \widehat{\vec{u}}_i^{[n]}, \widehat{\vec{u}}_j^{[n]}\right\rangle + \left\| \vec{z}^{[n]}\right\|_{\vec{G}}^2.
\end{align}
\end{theorem}

\begin{remark}
The proof of Theorem \ref{Thm6.1} is similar to that of Theorem \ref{Thm3.1} so that it is skipped here to avoid repetition.
Due to the discrete Parseval equality, the inequality \eqref{6.5} can be extended for $\{\vec{U}_{n,i}$, $\vec{u}_i^{[n+1]}\}$.
The Fourier spectral discretization in \eqref{6.1}-\eqref{6.3} can be directly applied to the special semi-discrete schemes \eqref{3.2.1}, \eqref{3.2.9} and \eqref{3.2.13}-\eqref{3.2.15} with the energy decay deduced from Theorem \ref{Thm6.1}.
\end{remark}

\begin{remark}
This paper does not focus on the error estimates of the fully discrete schemes \eqref{6.1}-\eqref{6.3}.
The readers are referred to \cite{LiX20,ChengH20}, in which  the error estimates of the fully discrete SAV-BDF1 and SAV-CN schemes with the Fourier spectral or  finite-element discretization in space are addressed.
\end{remark}

Before ending this section, the implementation  of the fully discrete SAV-GL schemes \eqref{6.1}-\eqref{6.2} is outlined here for the case of that the GLTDs are diagonally stable. For any $sN\times sN$ matrix $\vec{\Phi}=(\vec{\phi}_{ij})$ with $\vec{\phi}_{ij}\in \mathbb{R}^{N\times N}, i,j=1,2,\ldots,s$, and $sN\times N$ matrix $\vec{\Psi}=(\vec{\psi}_s,\cdots,\vec{\psi}_s)^T$ with  $\vec{\psi}_i \in \mathbb{R}^{N\times N}, i =1,2,\ldots,s$, we define the $sN\times N$ matrix $\vec{V} =(\vec{v}_1,\cdots,\vec{v}_s)^T:= \vec{\Phi}\bullet \vec{\Psi} $ with  $\vec{v}_i = \sum\limits_{j=1}^s \vec{\phi}_{ij} \circ \vec{\psi}_j \in \mathbb{R}^{N\times N}, i =1,2,\ldots,s$.

It can be deduced from \eqref{6.1} and \eqref{6.3} that
\begin{align}   \label{6.6}
& \left[\tau^{-1} \vec{D}_{11}^{-1}\!\otimes\! \vec{\mathcal{I}}_N  - \left(\mathcal{G}_h\!\circ\!\mathcal{L}_h\right)\!\otimes\! \vec{I}_s\right] \bullet\widehat{\vec{U}}^n =  \tau^{-1}\left[ \left(\vec{D}_{11}^{-1}\vec{D}_{12}\right)\!\otimes\! \vec{\mathcal{I}}_N \right] \bullet \widehat{\vec{u}}^{[n]} +  \left(\vec{Z}^n\!\otimes\!\vec{\mathcal{I}}_N\right)\!\circ\!\vec{B}^n, \\
\label{6.7}
& \left( \tau^{-1} \vec{D}_{11}^{-1} - \vec{C}^n/2 \right) \vec{Z}^n = \tau^{-1} \vec{D}_{11}^{-1}\vec{D}_{12} \vec{z}^{[n]} +  \widetilde{\vec{C}}^n/2,
\end{align}
where
\begin{align*}
& \widehat{\vec{U}}^n :=\! \left(\widehat{\vec{U}}_{n,1},\widehat{\vec{U}}_{n,2},\ldots,\widehat{\vec{U}}_{n,s} \right)^T,~~
\widehat{\vec{u}}^{[n]} :=\! \left(\widehat{\vec{u}}_{1}^{[n]},\widehat{\vec{u}}_{2}^{[n]},\ldots,\widehat{\vec{u}}_{s}^{[n]} \right)^T,~~ \vec{Z}^n:=\! \left(Z_{n,1},Z_{n,2},\ldots,Z_{n,s} \right)^T,
\\
& \vec{z}^{[n]}:=\! \left(z_{1}^{[n]},z_{2}^{[n]},\ldots,z_{s}^{[n]} \right)^T, ~~\widehat{\vec{W}}_{n,i} :=\! \widehat{W}(\bar{\vec{U}}_{n,i}), ~~\widehat{\vec{W}}^n:=\! \left(\widehat{\vec{W}}_{n,1}, \widehat{\vec{W}}_{n,2}, \ldots, \widehat{\vec{W}}_{n,s}\right)^T,
\\
& \vec{C}^n:=\! \mbox{diag}\left[\left\langle\widehat{\vec{W}}_{n,1},\mathcal{G}_h\circ \widehat{\vec{W}}_{n,1}\right\rangle,  \left\langle\widehat{\vec{W}}_{n,2},\mathcal{G}_h \circ \widehat{\vec{W}}_{n,2}\right\rangle, \ldots, \left\langle\widehat{\vec{W}}_{n,s},\mathcal{G}_h\circ \widehat{\vec{W}}_{n,s}\right\rangle\right],
\\
& \vec{B}^n:=\! \left(\mathcal{G}_h\circ \widehat{\vec{W}}_{n,1}, \mathcal{G}_h \circ \widehat{\vec{W}}_{n,2},\ldots, \mathcal{G}_h \circ \widehat{\vec{W}}_{n,s}\right)^T,
\end{align*}
and $\otimes$ denotes the Kronecker product symbol,
$\vec{I}_s$ is the $s\times s$ identity matrix, $\vec{\mathcal{I}}_N$ is the $N\times N$ matrix with   all elements being one, while $\widetilde{\vec{C}}^n$ is an unknown vector given by
\[ \widetilde{\vec{C}}^n: = \left(\left\langle\widehat{\vec{W}}_{n,1},\mathcal{G}_h\!\circ\!\mathcal{L}_h\!\circ\!\widehat{\vec{U}}_{n,1} \right\rangle, \left\langle\widehat{\vec{W}}_{n,2},\mathcal{G}_h\!\circ\!\mathcal{L}_h\!\circ\! \widehat{\vec{U}}_{n,2} \right\rangle,\ldots, \left\langle\widehat{\vec{W}}_{n,s},\mathcal{G}_h\!\circ\!\mathcal{L}_h\!\circ\!\widehat{\vec{U}}_{n,s} \right\rangle \right)^T.
\]
The discrete operator $\mathcal{G}_h$ in Fourier space is non-positive, so is the diagonal matrix $\vec{C}^n$.
When the GLTDs  \eqref{2.1.2} are diagonally stable, the matrix  $  \tau^{-1} \vec{D}_{11}^{-1} - \frac{1}{2} \vec{C}^n$ is invertible, see Theorem 3.1 in \cite{Calvo99}.
Multiplying \eqref{6.7} with $\left( \tau^{-1} \vec{D}_{11}^{-1} - \frac{1}{2} \vec{C}^n \right)^{-1}$ and  substituting the derived equation into \eqref{6.6} yields
\begin{align}   \label{6.8}
\left[\tau^{-1} \vec{D}_{11}^{-1}\!\otimes\! \vec{\mathcal{I}}_N  - \left(\mathcal{G}_h\!\circ\!\mathcal{L}_h\right)\!\otimes\! \vec{I}_s\right] \bullet \widehat{\vec{U}}^n
=\;& \vec{R}^n  + \frac{1}{2}\left[ \left(\tau^{-1} \vec{D}_{11}^{-1} -  \vec{C}^n/2 \right)^{-1} \widetilde{\vec{C}}^n \!\otimes\! \vec{\mathcal{I}}_N\right] \!\circ\! \vec{B}^n,
\end{align}
where  $\tau^{-1} \vec{D}_{11}^{-1}\!\otimes\! \vec{\mathcal{I}}_N  - \left(\mathcal{G}_h\!\circ\!\mathcal{L}_h\right)\!\otimes\! \vec{I}_s$ is  invertible and
\[ \vec{R}^n := \tau^{-1}\left[ \left(\vec{D}_{11}^{-1}\vec{D}_{12}\right)\!\otimes\! \vec{\mathcal{I}}_N \right] \bullet \widehat{\vec{u}}^{[n]}
+ \left[ \left(\tau^{-1} \vec{D}_{11}^{-1} -  \vec{C}^n/2 \right)^{-1}\!\vec{D}_{11}^{-1}\vec{D}_{12} \vec{z}^{[n]} \!\otimes\! \vec{\mathcal{I}}_N\right] \!\circ\! \vec{B}^n . \]

To summarize, the SAV-GL schemes \eqref{6.1}-\eqref{6.2} are implemented as follows:
\begin{description}
\item[(1)] Compute $\widehat{\vec{U}}^{n}$  from \eqref{6.8};
\item[(2)] Compute $\vec{Z}^n$ from \eqref{6.7} or the second equation of \eqref{6.1};
\item[(3)] Compute $\widehat{\vec{u}}^{[n+1]}$ and $\vec{z}^{[n+1]}$  from \eqref{6.2}.
\end{description}

\begin{remark}
We solve the linear system \eqref{6.8} by using an incomplete iteration.
Instead of \eqref{6.8}, iteratively,
for each $k\geq 0$, one  solves the simplified linear system
\begin{align}  \label{6.9}
\left[\tau^{-1} \vec{D}_{11}^{-1}\!\otimes\! \vec{\mathcal{I}}_N  \!-\! \left(\mathcal{G}_h\!\circ\!\mathcal{L}_h\right)\!\otimes\! \vec{I}_s\right] \bullet \widehat{\vec{U}}^{n,(k+1)}
\!=\! \vec{R}^n  \!+ \! \frac{1}{2}\!\left[ \left(\tau^{-1} \vec{D}_{11}^{-1} \!-\!  \vec{C}^n/2 \right)^{-1} \widetilde{\vec{C}}^{n,(k)} \!\otimes \vec{\mathcal{I}}_N\right] \!\circ \vec{B}^n,
\end{align}
where
\[  \widetilde{\vec{C}}^{n,(k)}: = \left(\left\langle\widehat{\vec{W}}_{n,1},\mathcal{G}_h\!\circ\!\mathcal{L}_h\!\circ\! \widehat{\vec{U}}_{n,1}^{(k)} \right\rangle, \left\langle\widehat{\vec{W}}_{n,2},\mathcal{G}_h\!\circ\!\mathcal{L}_h\!\circ\! \widehat{\vec{U}}_{n,2}^{(k)} \right\rangle,\ldots, \left\langle\widehat{\vec{W}}_{n,s},\mathcal{G}_h\!\circ\!\mathcal{L}_h\!\circ\! \widehat{\vec{U}}_{n,s}^{(k)}\right\rangle \right)^T,    \]
and $\widehat{\vec{U}}^{n,(0)} := \big(\widehat{\bar{\vec{U}}}_{n,1},\widehat{\bar{\vec{U}}}_{n,2},\ldots, \widehat{\bar{\vec{U}}}_{n,s} \big)^T$.
If $\left\| \widehat{\vec{U}}^{n,(k+1)} - \widehat{\vec{U}}^{n,(k)} \right\| = \sum_{i=1}^s \left\| \widehat{\vec{U}}_{n,i}^{(k+1)} - \widehat{\vec{U}}_{n,i}^{(k)}\right\| \le 10^{-12}$, then stop the above iteration and do $\widehat{\vec{U}}^n := \widehat{\vec{U}}^{n,(k+1)}$.
\end{remark}

\begin{remark}
In our computations, the codes are written in MATLAB  and call  both  {\tt fft} and  {\tt ifft} functions directly for the discrete Fourier and  inverse  Fourier transforms, so that they are simple and efficient.
\end{remark}

\section{Numerical experiments}   \label{sec:6}

This section applies respectively the SAV-GL schemes \eqref{3.2.1}, \eqref{3.2.9} and \eqref{3.2.13}-\eqref{3.2.15} combined with the Fourier spectral method to three typical gradient flow models
(the Allen-Cahn, Cahn-Hilliard, and phase field crystal) with
the periodic boundary conditions in order to demonstrate their energy
stability and accuracy.
Specially,  \eqref{3.2.1} with $\theta = \frac{3}{4}$,   \eqref{3.2.9} with $\gamma = \delta = 1$, and  \eqref{3.2.9} with $\gamma = \delta = 2$ are chosen and corresponding fully-discrete SAV-GL schemes
are abbreviated as {\tt SAV-GL(1)}, {\tt SAV-GL(2)} and {\tt SAV-GL(3)}, respectively, for convenience.
For \eqref{3.2.13}-\eqref{3.2.15}, the coefficients are chosen as one-stage members of the two-step Runge-Kutta time discretizations  {\cite[pp. 1497, Example~1]{Li99}}  and two- and three-stage members of the Radau IIA time discretizations  {\cite[Section~IV, pp.~74]{Hairer}}, and corresponding fully-discrete SAV-GL schemes
are named as {\tt SAV-GL(4)}, {\tt SAV-GL(5)} and {\tt SAV-GL(6)}, respectively, for simplicity.
Unless otherwise specified, the domain $\Omega=[0,2\pi]\times [0,2\pi]$, the spatial stepsize $h=\frac{2\pi}{256}$, the discrete free energy is defined by
\[ \Upsilon(\vec{u}^{[n]},\vec{r}^{[n]}) =  \frac{1}{2} \sum_{i,j}^r g_{ij}\left\langle \mathcal{L}_h \circ \widehat{\vec{u}}_i^{[n]}, \widehat{\vec{u}}_j^{[n]}\right\rangle + \left\| \vec{z}^{[n]}\right\|_{\vec{G}}^2 - C_0,  \]
where $\vec{u}^{[n]} = \left( \vec{u}_1^{[n]}, \vec{u}_2^{[n]},\cdots, \vec{u}_r^{[n]}\right)$ with $\vec{u}_i^{[n]} \in \mathcal{V}_h$, $\vec{z}^{[n]} = \left( z_1^{[n]}, z_2^{[n]},\cdots, z_r^{[n]}\right)$ with $z_i^{[n]} \in \mathbb{R}$, and the matrix $\vec{G} = (g_{ij}) \in\mathbb{R}^{r\times r}$ is only dependent on the {GLTDs}.

\subsection{Allen-Cahn model}

The Allen-Cahn model is a second-order nonlinear partial differential equation (PDE)
\begin{align}   \label{5.2}
\frac{\partial u}{\partial t} = \epsilon^2 \Delta u +  (u - u^3),
\end{align}
 introduced to describe the motion of anti-phase boundaries in crystalline solids   \cite{Allen79}  and then widely used to study the phase transition  and the interfacial dynamics in material sciences, see e.g. \cite{Chen02,ShenJ10a,Golubovic11}.
It can be derived from the $L^2$ gradient flow of
the free energy
\begin{align}   \label{5.1}
\mathcal{F}(u) = \int_{\Omega} \frac{\epsilon^2}{2} |\nabla u|^2 + \frac{1}{4}(u^2-1)^2 dx.
\end{align}

In the following, we implement {\tt SAV-GL(1)}$\sim${\tt SAV-GL(6)} for \eqref{5.2} 
and choose
the operators $\mathcal{L}$, $\mathcal{G}$ and the energy $\mathcal{F}_1$ as follows
\[  \mathcal{L} = - \epsilon^2 \Delta + \beta ,~~~ \mathcal{G} = -1, ~~~ \mathcal{F}_1(u) = \int_{\Omega} \frac{1}{4}\left(u^2 - 1 \right)^2 - \frac{\beta}{2}u^2 dx,\]
where $\beta$ is a non-negative parameter, e.g. $\beta = 2$, and $\mathcal{F}_1(u)$ is bounded from below.

\begin{example} \label{Exp5.1.1}
This example is used to check the   accuracy of {\tt SAV-GL(1)}$\sim${\tt SAV-GL(6)}
 for the Allen-Cahn equation \eqref{5.2}.
 For this purpose,   the parameter $\epsilon$ is taken as $0.1$,
 the initial data are chosen as $u(x,y,0) = \sin(x)\sin(y)$, and
 {\tt SAV-GL(6)} with $\tau = 10^{-4}$ is used to get the reference solution for computing the $L^2$ errors.
Table \ref{tab:5.1} presents the $L^2$ errors of {\tt SAV-GL(1)}$\sim${\tt SAV-GL(6)} at $t=1.5$ and corresponding convergence rates  with different time stepsizes.
It can be found that the numerical accuracies of {\tt SAV-GL(1)}$\sim${\tt SAV-GL(4)} are consistent with the theoretical,
and {\tt SAV-GL(5)} and {\tt SAV-GL(6)} can arrive at the third-order and fourth-order accuracy for the Allen-Cahn equation \eqref{5.2}, since the number of extrapolation points is $\nu = 3$ and $4$, respectively. Those results well verify  the statements in Remark \ref{remark4.6}.
\end{example}

\begin{table}[!htbp]
\begin{center}
\caption{Example \ref{Exp5.1.1}. $L^2$ errors of {\tt SAV-GL(1)}$\sim${\tt SAV-GL(6)} at $t = 1.5$ and corresponding convergence rates.}
\begin{tabular*}{\textwidth}{@{\extracolsep{\fill}}lccccccccccc}
\toprule[1pt]
{}&\multicolumn{2}{c}{{\tt SAV-GL(1)}}&{} & \multicolumn{2}{c}{{\tt SAV-GL(2)}} & {} & \multicolumn{2}{c}{{\tt SAV-GL(3)}} & \\
\cmidrule[0.5pt]{2-3}\cmidrule[0.5pt]{5-6} \cmidrule[0.5pt]{8-9}
{$K$} & Errors & Orders &  {} & Errors & Orders  &  {} & Errors & Orders   \\
\midrule[1pt]
$80$        &4.7399e-02     &--          &{}      &1.2682e-03     &--       &{}      &2.0128e-03     &--       \\
$120$       &3.1691e-02     &0.9928      &{}      &5.6554e-04     &1.9918   &{}      &8.9734e-04     &1.9924     \\
$160$       &2.3803e-02     &0.9950      &{}      &3.1866e-04     &1.9941   &{}      &5.0554e-04     &1.9946     \\
$200$       &1.9058e-02     &0.9962      &{}      &2.0415e-04     &1.9954   &{}      &3.2385e-04     &1.9958     \\
$240$       &1.5891e-02     &0.9969      &{}      &1.4187e-04     &1.9962   &{}      &2.2504e-04     &1.9965     \\
\midrule[1pt]
{}&\multicolumn{2}{c}{{\tt SAV-GL(4)}}&{} & \multicolumn{2}{c}{{\tt SAV-GL(5)}} & {} & \multicolumn{2}{c}{{\tt SAV-GL(6)}} & \\
\cmidrule[0.5pt]{2-3}\cmidrule[0.5pt]{5-6} \cmidrule[0.5pt]{8-9}
{$K$} & Errors & Orders &  {} & Errors & Orders  &  {} & Errors & Orders   \\
\midrule[1pt]
$80$        &1.0273e-03     &--          &{}      &1.2299e-06     &--       &{}      &1.2539e-08     &--       \\
$120$       &4.5990e-04     &1.9822      &{}      &3.4481e-07     &3.1365   &{}      &2.2253e-09     &4.2641     \\
$160$       &2.5961e-04     &1.9877      &{}      &1.4145e-07     &3.0975   &{}      &6.5717e-10     &4.2398    \\
$200$       &1.6650e-04     &1.9906      &{}      &7.1206e-08     &3.0759   &{}      &2.5853e-10    &4.1809     \\
$240$       &1.1579e-04     &1.9923      &{}      &4.0742e-08     &3.0622   &{}      &1.2207e-10     &4.1160     \\
\bottomrule[1pt]
\end{tabular*}        \label{tab:5.1}
\end{center}
\end{table}

\begin{example} \label{Exp5.1.2}
This example uses  {\tt SAV-GL(1)}$\sim${\tt SAV-GL(6)} to simulate the  phase
separation and coarsening process.
The parameter $\epsilon$  in \eqref{5.2} is chosen as $0.05$,
the time stepsize $\tau$ is taken as $0.1$ or $ 0.01$,
and the initial data are  $u(x,y,0) = 0.1\times \mbox{\tt rand}(x,y) - 0.05$, where $\mbox{\tt rand}(x,y)$ generates random number between $-1$ and $1$.

Figure \ref{fig5.1.1} gives the cut lines and contour lines of the numerical solution at $t=200$ derived by {\tt SAV-GL(1)}, {\tt SAV-GL(3)}, {\tt SAV-GL(4)} and {\tt SAV-GL(5)} with $\tau=0.1$.
We see that the numerical solutions obtained by those schemes are  similar or have a little difference due to the low-accuracy of {\tt SAV-GL(1)}.
Figure \ref{fig5.1.2}  presents the snapshots of the numerical solutions at $t = 0, 2, 10, 50, 100$ and  $200$ obtained by {\tt SAV-GL(5)} with $\tau = 0.1$. One can clearly observe the phase separation and coarsening process.
In order to check numerically the discrete maximum principle of those schemes, Figure \ref{fig5.1.3} shows the maximal and minimal values of the numerical solutions obtained by {\tt SAV-GL(1)}, {\tt SAV-GL(3)}, {\tt SAV-GL(4)} and {\tt SAV-GL(5)} with $\tau = 0.1$.
Figure \ref{fig5.1.4} displays
the discrete energy curves of {\tt SAV-GL(1)}, {\tt SAV-GL(3)}, {\tt SAV-GL(4)} and {\tt SAV-GL(5)} with $\tau = 0.1$ and $0.01$.
It is shown that the discrete energy curves are monotonically decreasing so that those schemes are energy stable in solving the Allen-Cahn model \eqref{5.2}; there are  obvious differences
between those discrete energy curves with $\tau = 0.1$ but the differences are indistinguishable  for $\tau=0.01$;
and the third-order accurate {\tt SAV-GL(5)} can reach steady state faster than  SAV-GL(1), SAV-GL(3) and {\tt SAV-GL(4)}.

\end{example}

\begin{figure}
\centering
{\includegraphics[width=7.5cm]{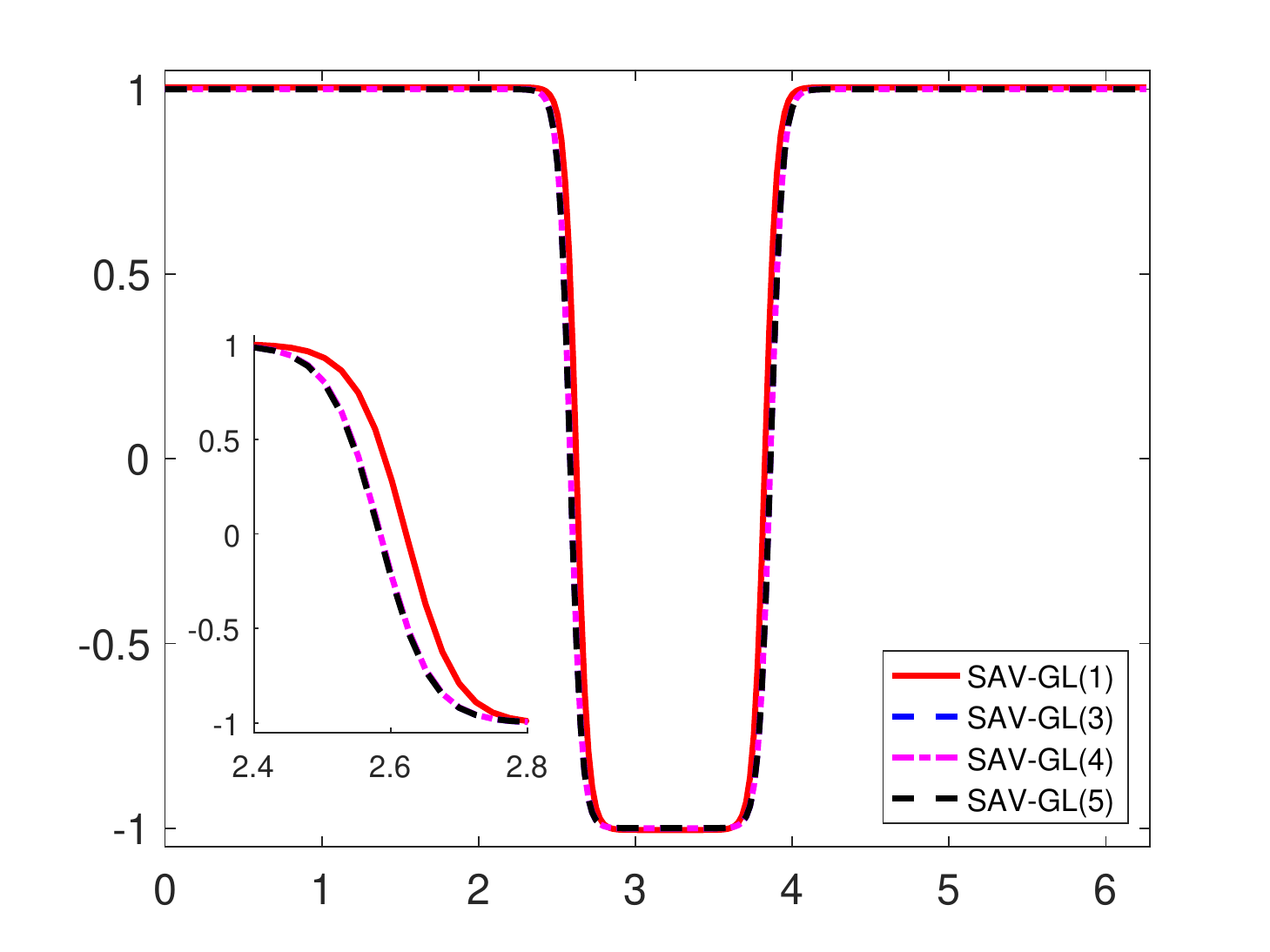}}
{\includegraphics[width=7.5cm]{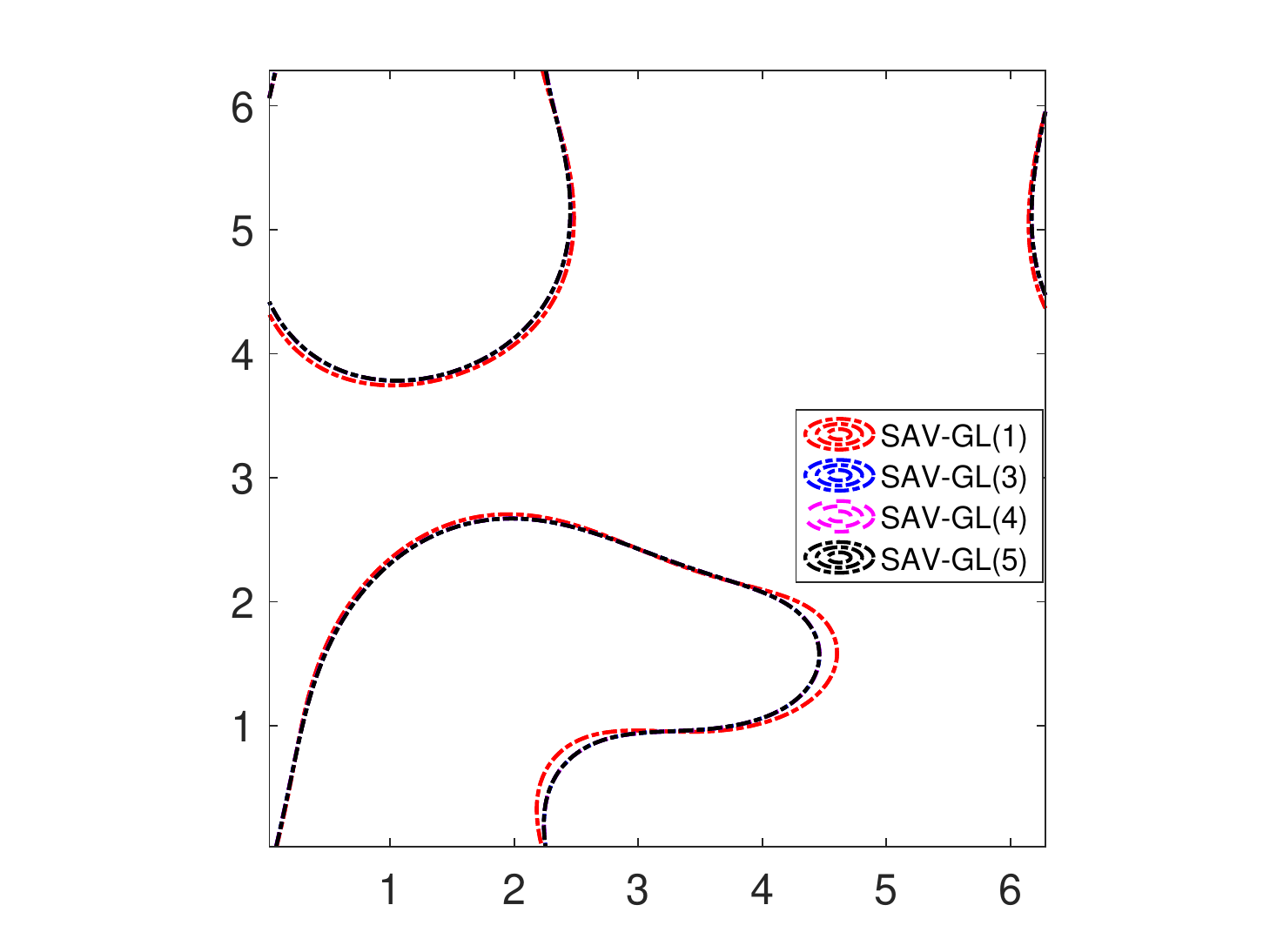}}
\caption{Example \ref{Exp5.1.2}.
Cut lines of $u(x, y,t)$ along $y=\frac{\pi}{2}$ (Left)
and  
contour lines of $u(x,y,t)=-0.1$ (Right) at $t=200$.}
\label{fig5.1.1}
\end{figure}

\begin{figure}
\includegraphics[width=16cm,height=9cm]{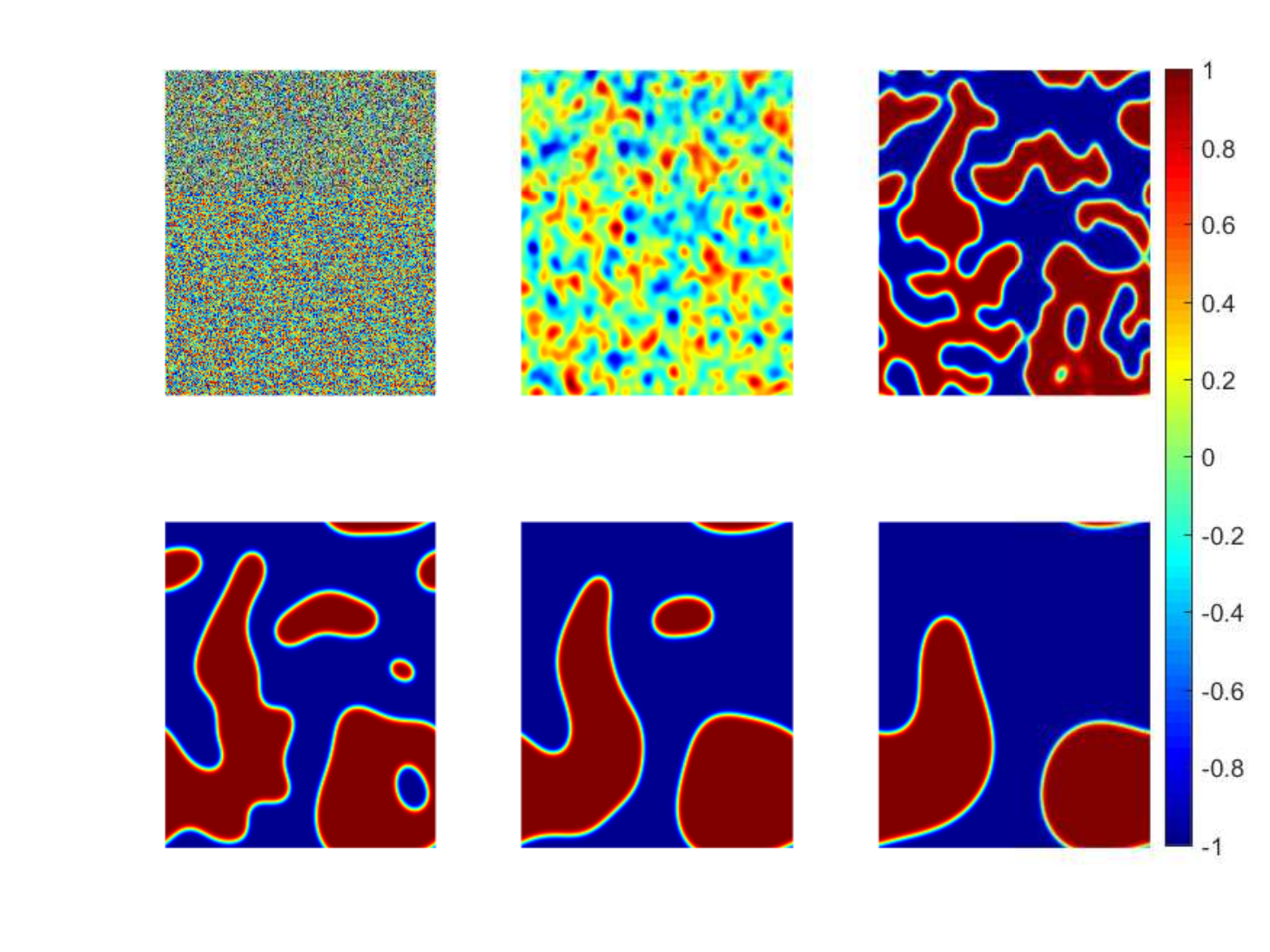}
\caption{Example \ref{Exp5.1.2}.  Snapshots of the numerical solutions at $t = 0, 2, 10, 50, 100$ and  $200$ obtained by using {\tt SAV-GL(5)} with $\tau = 0.1$.}
\label{fig5.1.2}
\end{figure}

\begin{figure}
\centering
{\includegraphics[width=7.5cm]{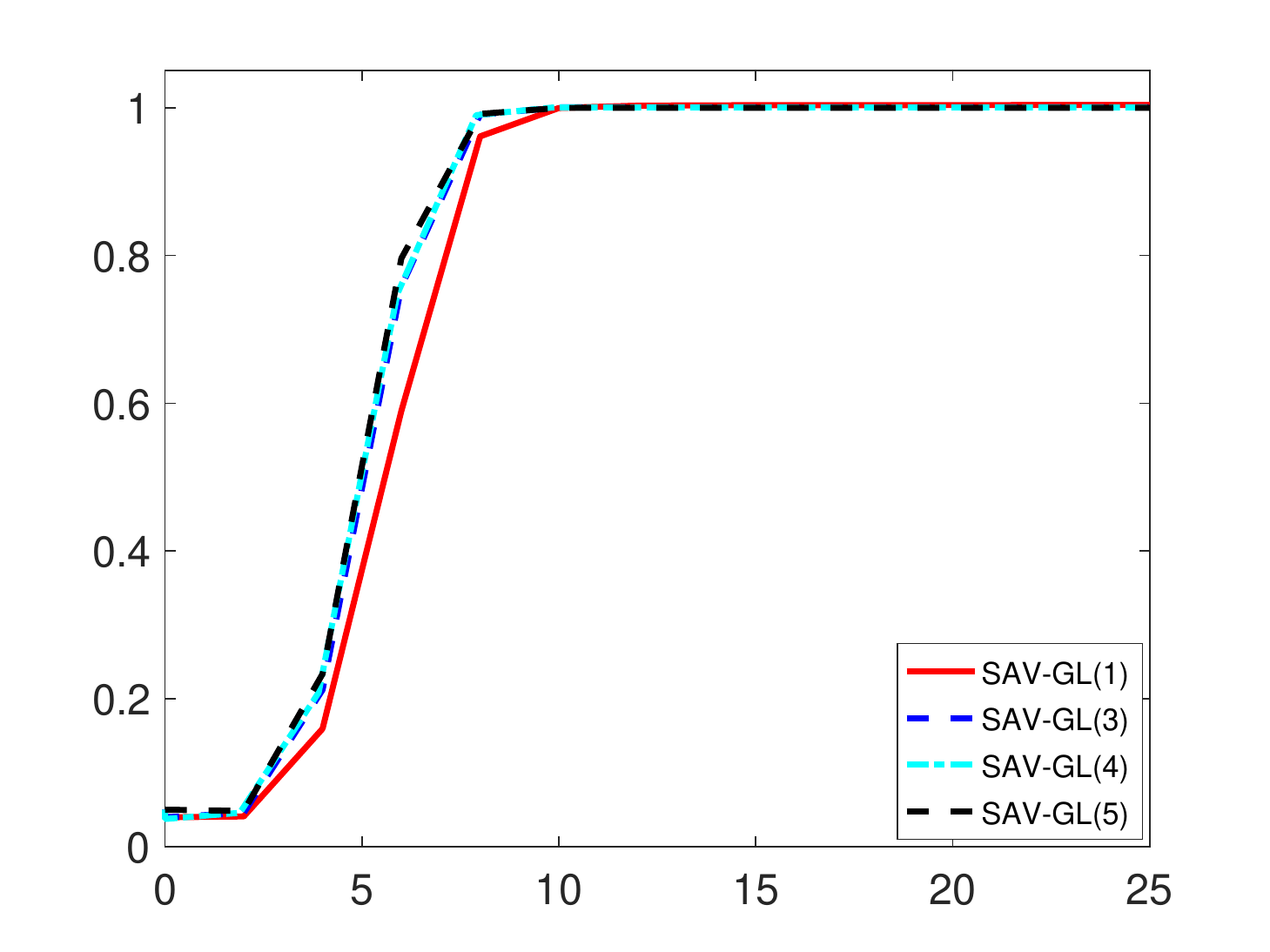}}
{\includegraphics[width=7.5cm]{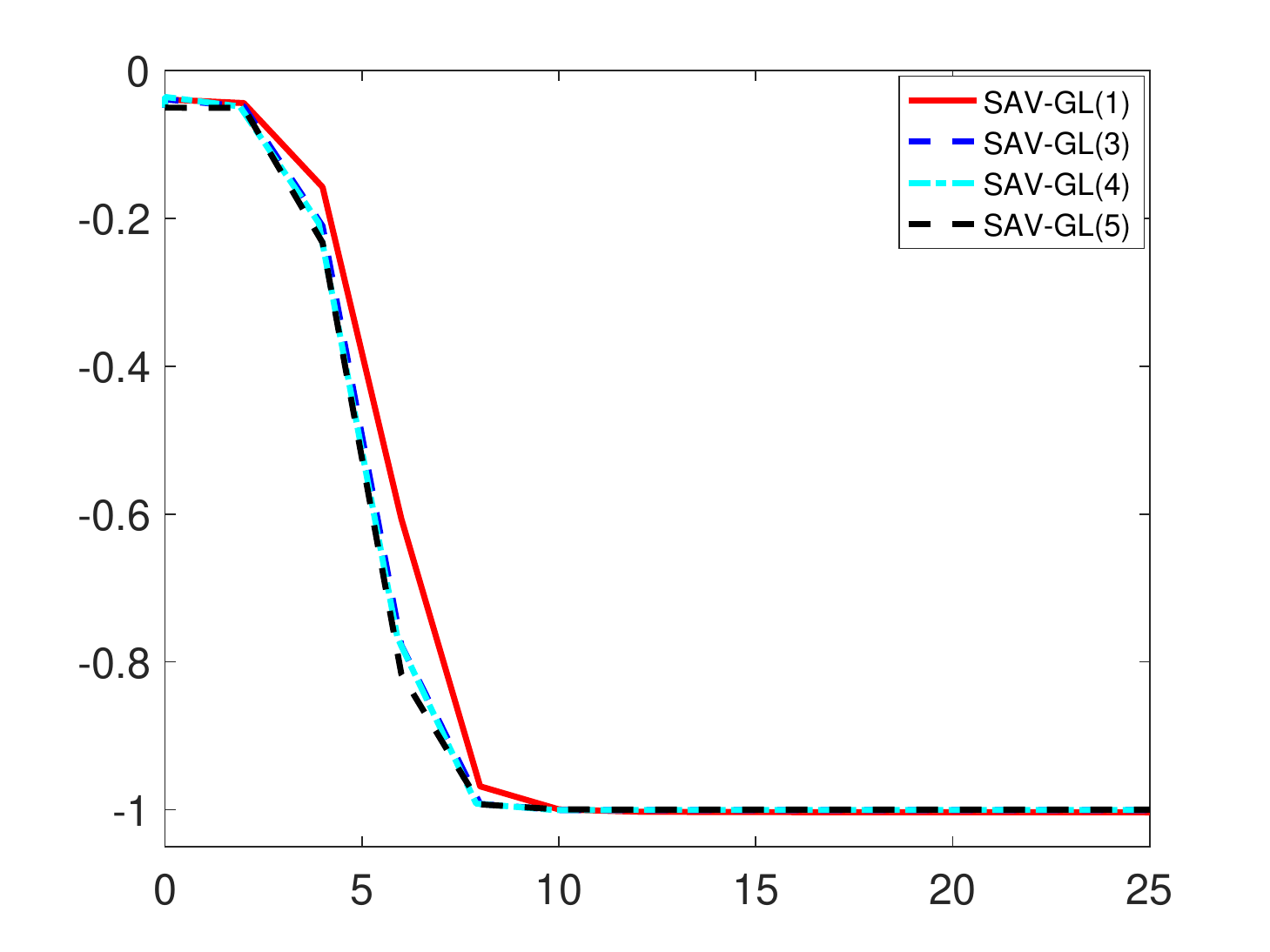}}
\caption{Example \ref{Exp5.1.2}.
Maximal and minimal values (from left to right) of the numerical solutions derived by {\tt SAV-GL(1)}, {\tt SAV-GL(3)}, {\tt SAV-GL(4)} and {\tt SAV-GL(5)} with $\tau = 0.1$.}
\label{fig5.1.3}
\end{figure}

\begin{figure}
{\includegraphics[width=7.5cm]{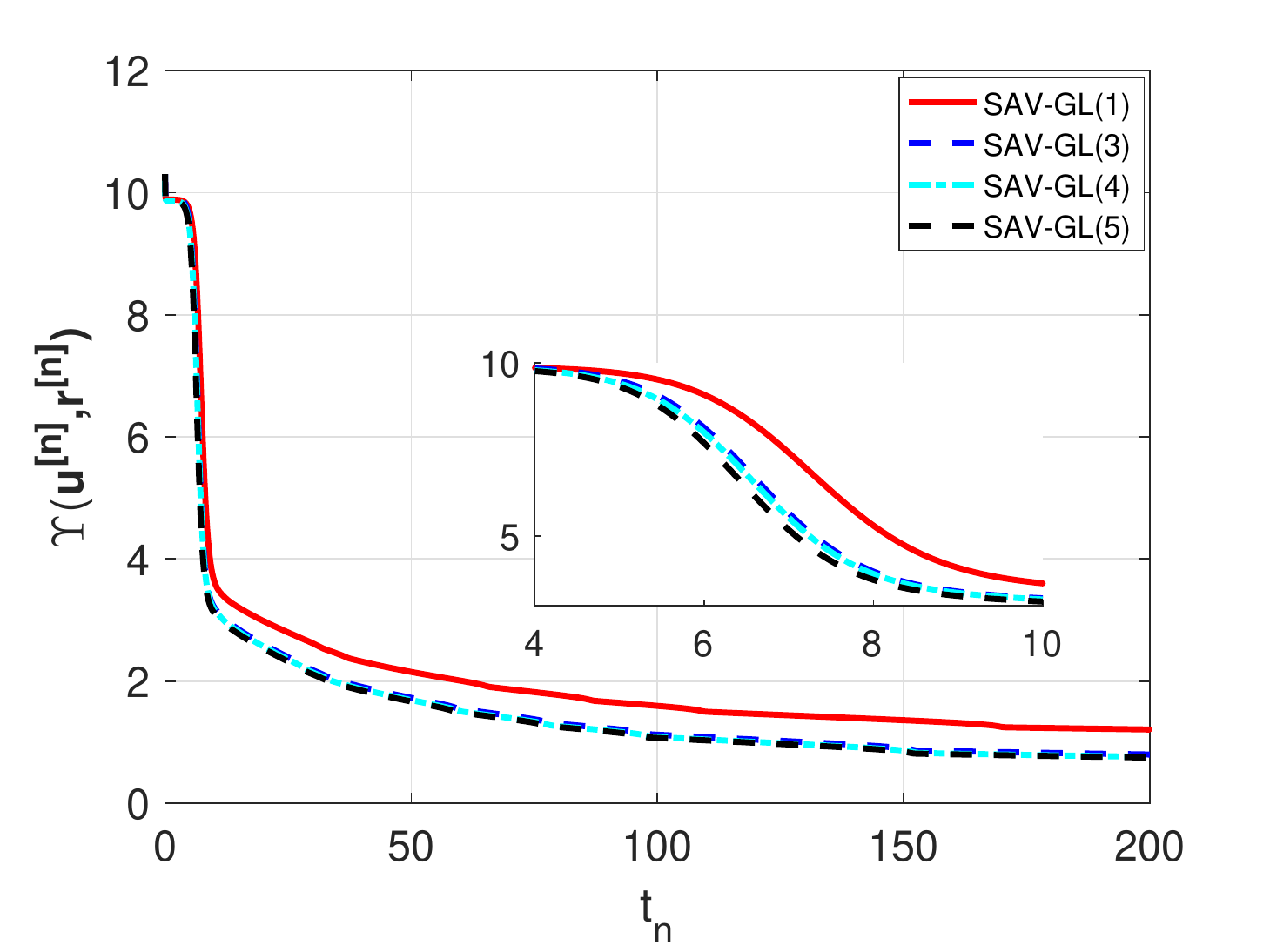}}
{\includegraphics[width=7.5cm]{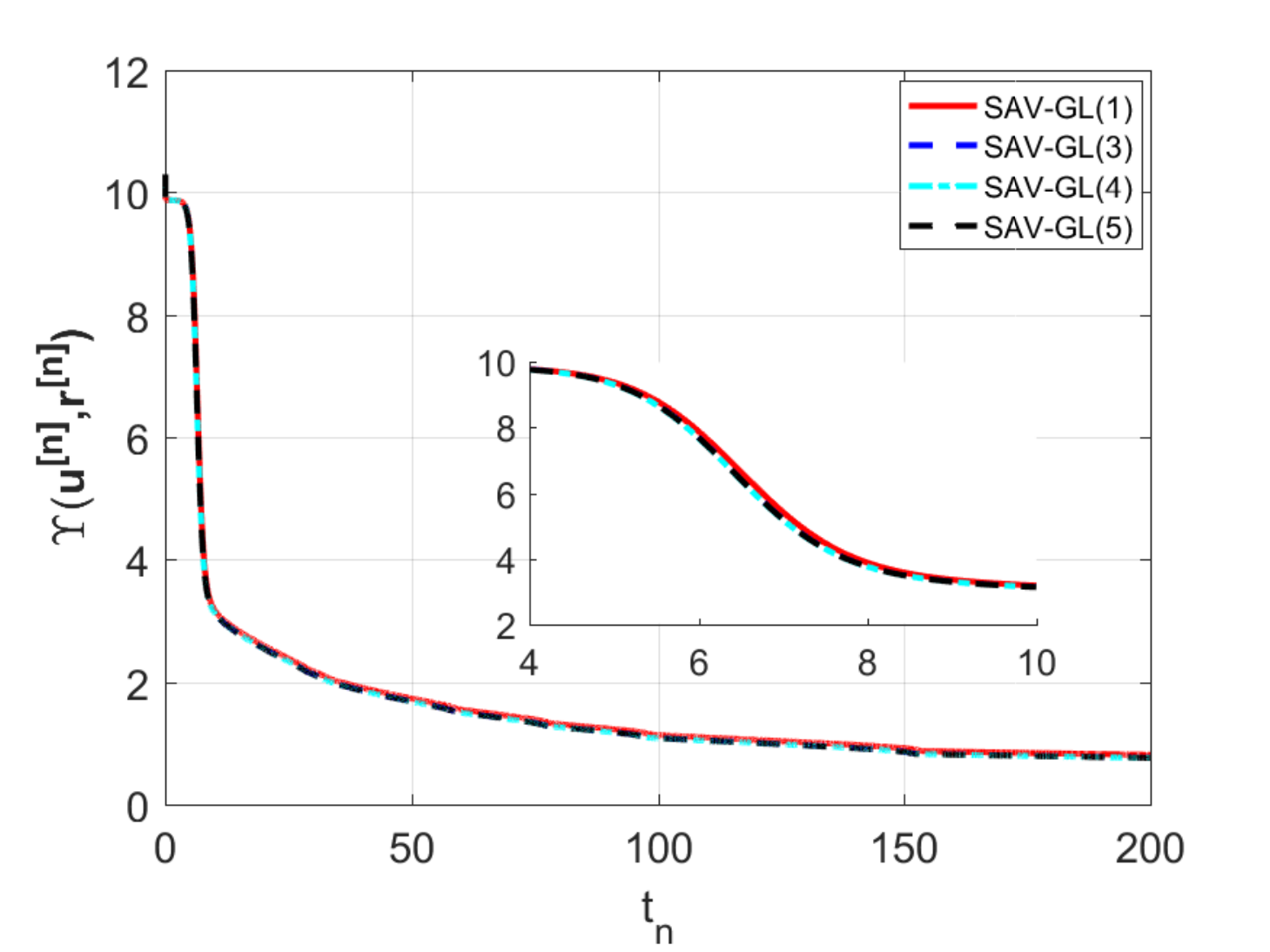}}
\caption{Example \ref{Exp5.1.2}.
Discrete energies of {\tt SAV-GL(1)}, {\tt SAV-GL(3)}, {\tt SAV-GL(4)} and {\tt SAV-GL(5)} with $\tau = 0.1$ (Left) and $0.01$ (Right).}
\label{fig5.1.4}
\end{figure}

\subsection{Cahn-Hilliard model}

The Cahn-Hilliard (CH) model was introduced by Cahn and Hilliard
in \cite{Cahn58} to describe the complicated phase separation and coarsening phenomena. Different from the Allen-Cahn \eqref{5.2},  the CH model is   derived from the $H^{-1}$ gradient flow of the free energy \eqref{5.1} and  is a fourth-order nonlinear PDE as follows
\begin{align}   \label{5.3}
\frac{\partial u}{\partial t} = \Delta \left( -\epsilon^2\Delta u + u^3 - u \right),
\end{align}

In order to validate the energy
stability and accuracy of {\tt SAV-GL(1)}$\sim${\tt SAV-GL(6)} for the CH model \eqref{5.3}, the operators $\mathcal{L}$, $\mathcal{G}$ and  the energy $\mathcal{F}_1(u)$ are taken as
\[  \mathcal{L} = - \epsilon^2 \Delta + \beta ,~~~ \mathcal{G} = \Delta, ~~~ \mathcal{F}_1(u) = \int_{\Omega} \frac{1}{4}\left(u^2 - 1 \right)^2 - \frac{\beta}{2} u^2dx, \]
where the parameter $\beta$ is chosen as $2$ in subsequent simulations, and it is obvious that the energy $\mathcal{F}_1(u)$ is bounded from below.

\begin{example} \label{Exp5.2.1}
This example is  used to test the accuracy of  {\tt SAV-GL(1)}$\sim${\tt SAV-GL(6)} for  the CH model \eqref{5.3}. The parameter $\epsilon$ is chosen as $ 1$, the initial data are  $u(x,y,0) = 0.4\sin(x)\sin(y)$,  and {\tt SAV-GL(6)} with $\tau = 10^{-4}$ is used to generate the reference solution for computing the $L^2$ errors.
Table \ref{tab:5.2} lists the $L^2$ errors of {\tt SAV-GL(1)}$\sim${\tt SAV-GL(6)} at $t=0.3$ and corresponding convergence rates with different time stepsizes.
One can find that the numerical accuracies of {\tt SAV-GL(1)}$\sim${\tt SAV-GL(4)} are consistent with the theoretical, while  {\tt SAV-GL(5)} (resp. {\tt SAV-GL(6)}) with the number of extrapolation points $\nu = 3$ (resp. $\nu=4$) can arrive at the third-order (resp. fourth-order) accuracy, which validates the statement in Remark \ref{remark4.6}.
\end{example}

\begin{table}[!htbp]
\begin{center}
\caption{Example \ref{Exp5.2.1}. $L^2$ errors of {\tt SAV-GL(1)}$\sim${\tt SAV-GL(6)} at {$t = 0.3$} and corresponding convergence rates.}
\begin{tabular*}{\textwidth}{@{\extracolsep{\fill}}lccccccccccc}
\toprule[1pt]
{}&\multicolumn{2}{c}{{\tt SAV-GL(1)}}&{} & \multicolumn{2}{c}{{\tt SAV-GL(2)}} & {} & \multicolumn{2}{c}{{\tt SAV-GL(3)}} & \\
\cmidrule[0.5pt]{2-3}\cmidrule[0.5pt]{5-6} \cmidrule[0.5pt]{8-9}
{$K$} & Errors & Orders &  {} & Errors & Orders  &  {} & Errors & Orders   \\
\midrule[1pt]
$120$        &3.6925e-03     &--          &{}      &3.1250e-05    &--       &{}      &4.1560e-05     &--       \\
$160$       &2.7695e-03     &0.9998     &{}      &1.7598e-05     &1.9961   &{}      &2.3339e-05     &2.0057     \\
$200$       &2.2157e-03     &0.9998      &{}      &1.1270e-05     &1.9969   &{}      &1.4923e-05     &2.0042     \\
$240$       &1.8465e-03     &0.9999      &{}      &7.8304e-06     &1.9974   &{}      &1.0357e-05     &2.0034     \\
$280$       &1.5827e-03     &0.9999      &{}      &5.7548e-06     &1.9978   &{}      &7.6059e-06     &2.0028     \\
\midrule[1pt]
{}&\multicolumn{2}{c}{{\tt SAV-GL(4)}}&{} & \multicolumn{2}{c}{{\tt SAV-GL(5)}} & {} & \multicolumn{2}{c}{{\tt SAV-GL(6)}} & \\
\cmidrule[0.5pt]{2-3}\cmidrule[0.5pt]{5-6} \cmidrule[0.5pt]{8-9}
{$K$} & Errors & Orders &  {} & Errors & Orders  &  {} & Errors & Orders   \\
\midrule[1pt]
$120$        &2.2516e-05     &--          &{}      &1.8203e-09     &--       &{}      &2.4250e-09     &--       \\
$160$       &1.2680e-05     &1.9959      &{}      &7.7214e-10     &2.9811   &{}      &7.2844e-10     &4.1806     \\
$200$       &8.1207e-06     &1.9969      &{}      &3.9652e-10     &2.9866   &{}      &2.8832e-10     &4.1535    \\
$240$       &5.6420e-06     &1.9975      &{}      &2.2981e-10     &2.9919   &{}      &1.3615e-10    &4.1153     \\
$280$       &4.1465e-06     &1.9979      &{}      &1.4486e-10     &2.9935   &{}      &7.2146e-11     &4.1198     \\
\bottomrule[1pt]
\end{tabular*}        \label{tab:5.2}
\end{center}
\end{table}

\begin{example} \label{Exp5.2.2}
This example applies {\tt SAV-GL(1)}$\sim${\tt SAV-GL(6)} to a benchmark problem of studying the coarsening effect. The parameter
$\epsilon$ in \eqref{5.3} is taken as  $0.1$,  the time stepsize $\tau$ is chosen as $0.1$ or $0.01$, and the initial data are specified by the following expression  \cite{YangX20}
\begin{align*}
u_0(x,y,0) = \sum_{i=1}^2 -\tanh\left(\frac{\sqrt{(x-x_i)^2 + (y-y_i)^2 - \nu_i}}{1.2\epsilon}\right) + 1,
\end{align*}
where $(x_1,y_1,\nu_1) = (\pi-0.7,\pi-0.6,1.5)$ and $(x_2,y_2,\nu_2) = (\pi+1.65,\pi+1.6,0.7)$.

Figure \ref{fig5.2.2}  gives the snapshos of the numerical solutions at $t = 0, 2, 5, 7, 9$ and  $20$ obtained by {\tt SAV-GL(5)} with $\tau = 0.01$. One can clearly observe the coarsening effect that the small circle is absorbed into the big circle, and the total absorption happens at around $t = 10$.
Figure \ref{fig5.2.1} presents the cut lines and contour lines of the numerical solutions at $t=5$ and $20$ derived by {\tt SAV-GL(1)}, {\tt SAV-GL(3)}, {\tt SAV-GL(4)} and {\tt SAV-GL(5)} with $\tau=0.1$. It can be seen that those solutions at $t=5$ have some visible differences, see Figure \ref{fig5.2.1} (a), but  they  become quite similar when  $t=20$, see Figure \ref{fig5.2.1} (b).
{Compared to the initial total mass, the total mass differences} at $t_n$  obtained by {\tt SAV-GL(1)}, {\tt SAV-GL(3)}, {\tt SAV-GL(4)} and {\tt SAV-GL(5)} with $\tau = 0.1$ are given in Figure \ref{fig5.2.3}, which checks the mass conservation numerically.
Figure \ref{fig5.2.4} shows the discrete energy curves of {\tt SAV-GL(1)}, {\tt SAV-GL(3)}, {\tt SAV-GL(4)} and {\tt SAV-GL(5)} with $\tau = 0.1$ and $0.01$. One can see that the discrete energy curves are monotonically decreasing so that those schemes are energy stable in solving the CH model \eqref{5.3}; the discrete energy curves of those schemes have big differences when $\tau = 0.1$, but the differences become small for $\tau=0.01$; and similarly, SAV-GL(5)  can obtain the steady state faster than t SAV-GL(1), SAV-GL(3) and {\tt SAV-GL(4)}. The results shown in Figure \ref{fig5.2.4} are also consistent with the differences between the numerical solutions at $t=5$ in Figure \ref{fig5.2.1}.
\end{example}

\begin{figure}
\centering
{\includegraphics[width=7.5cm]{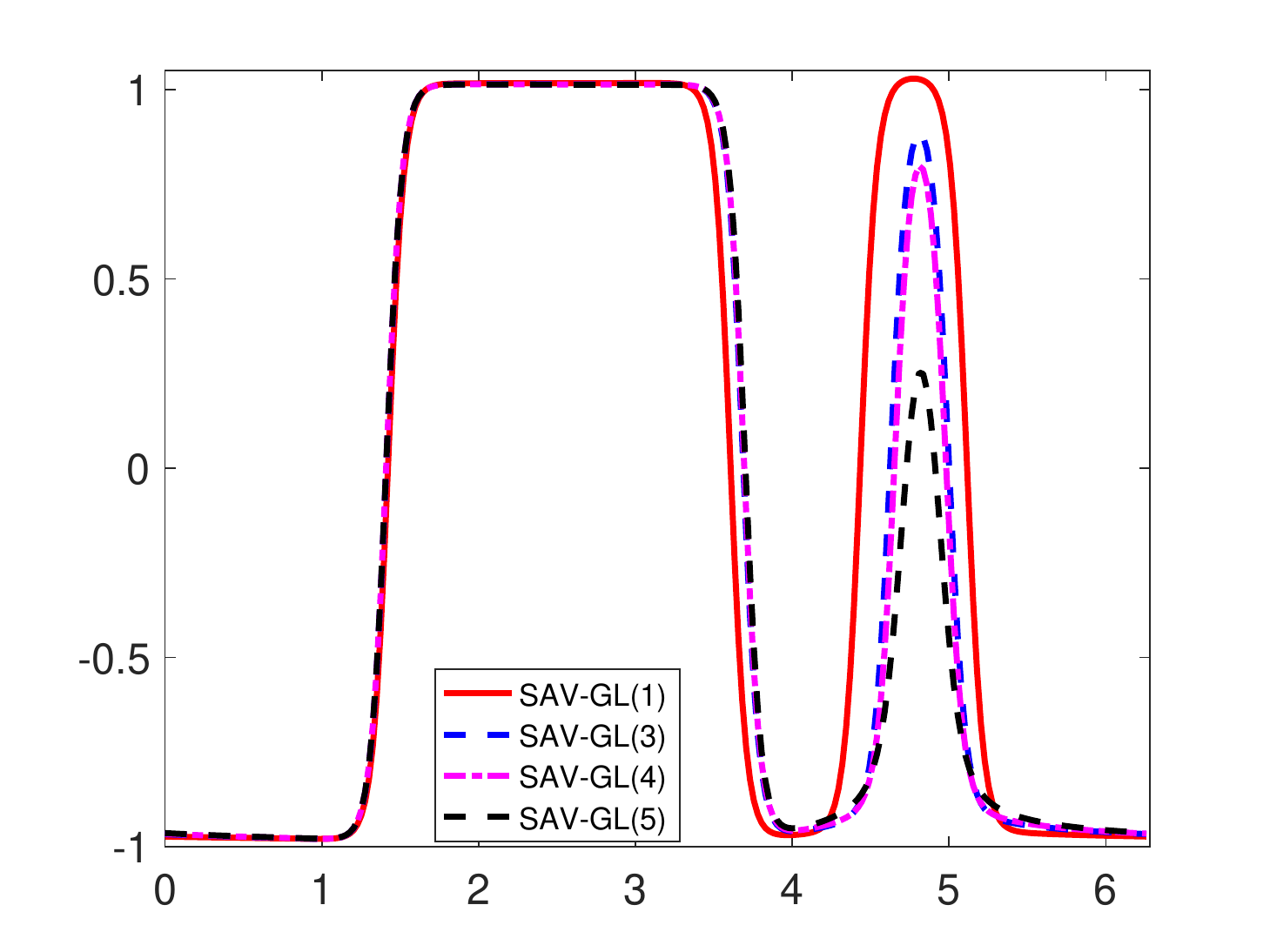}}
{\includegraphics[width=7.5cm]{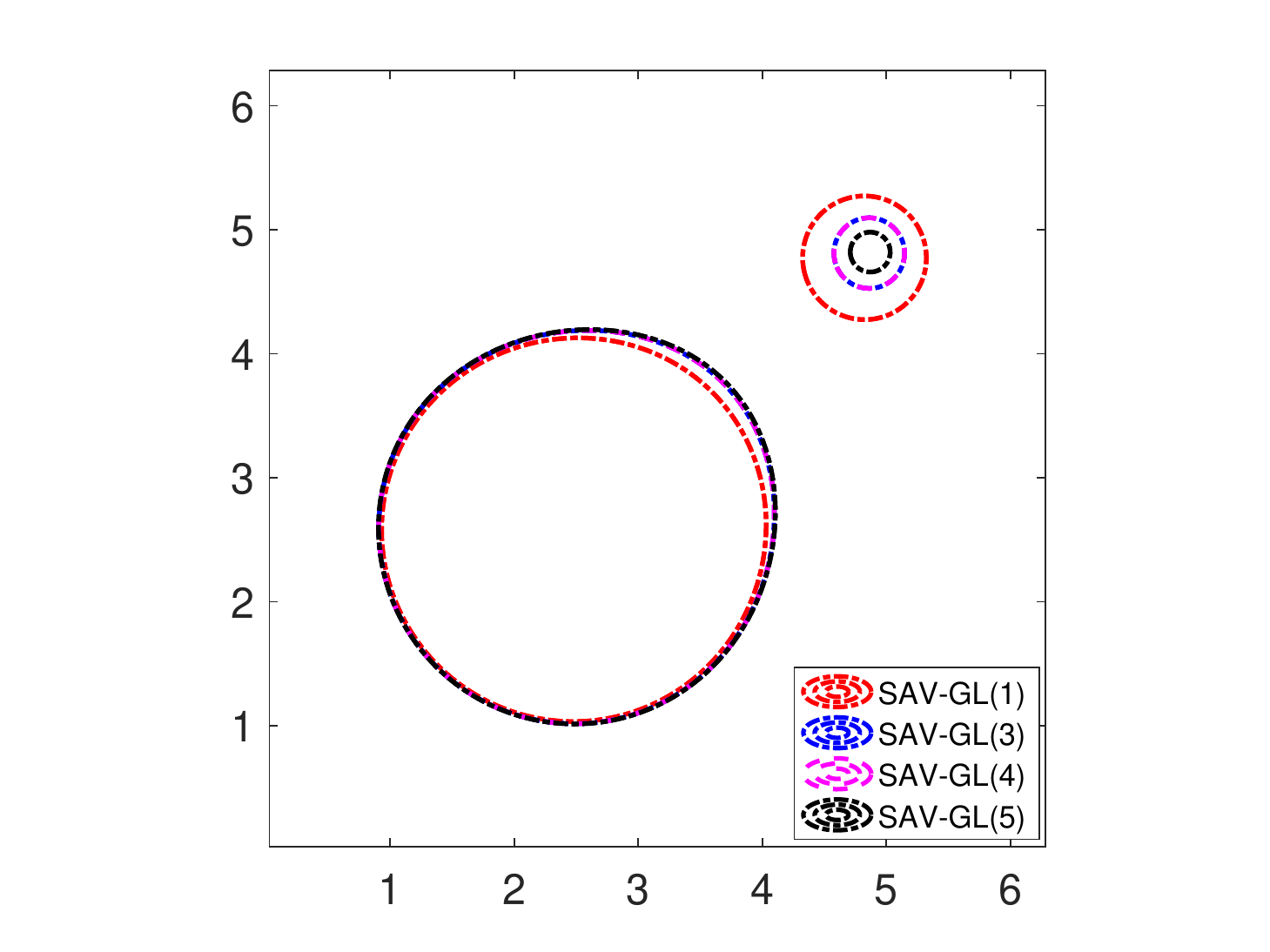}}
  \centerline{ (a) $t=5$}

{\includegraphics[width=7.5cm]{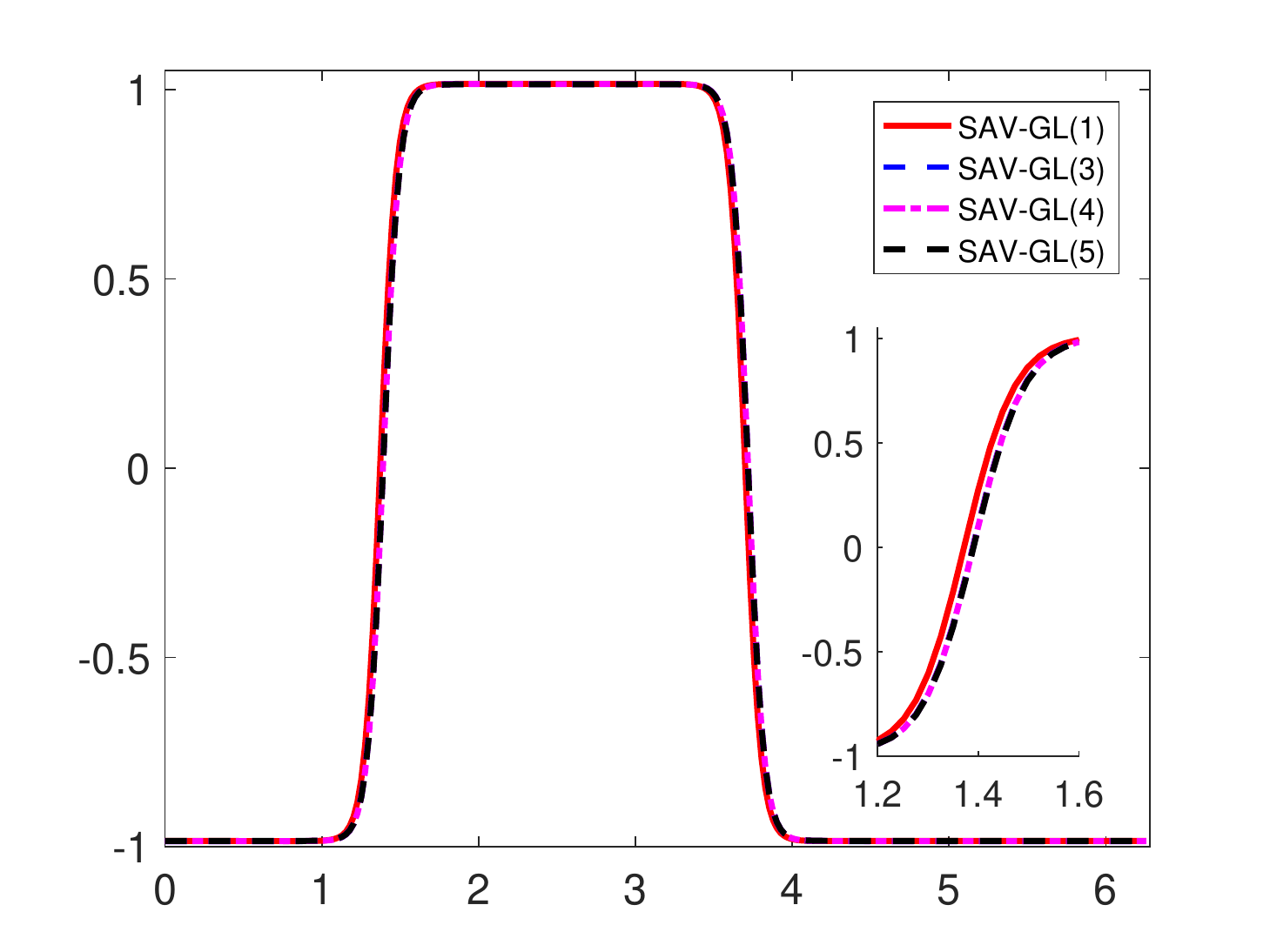}}
{\includegraphics[width=7.5cm]{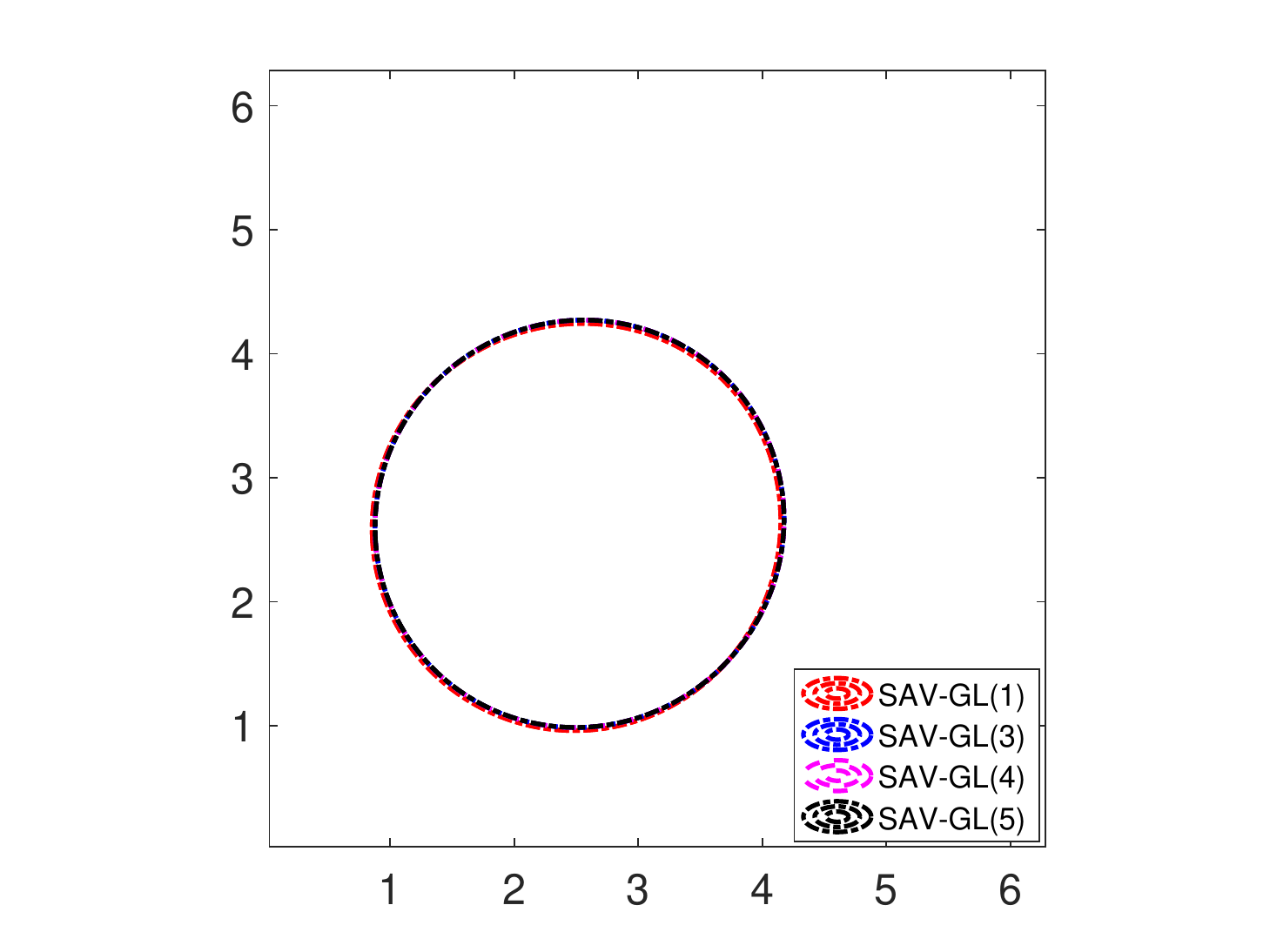}}
  \centerline{ (b) $t=20$}
\caption{Example \ref{Exp5.2.2}.
Left: cut lines of the numerical solutions along $y=x$; 
right: contour lines of  $u(x,y,t)=-0.1$.}
\label{fig5.2.1}
\end{figure}

\begin{figure}
\includegraphics[width=16cm,height=9cm]{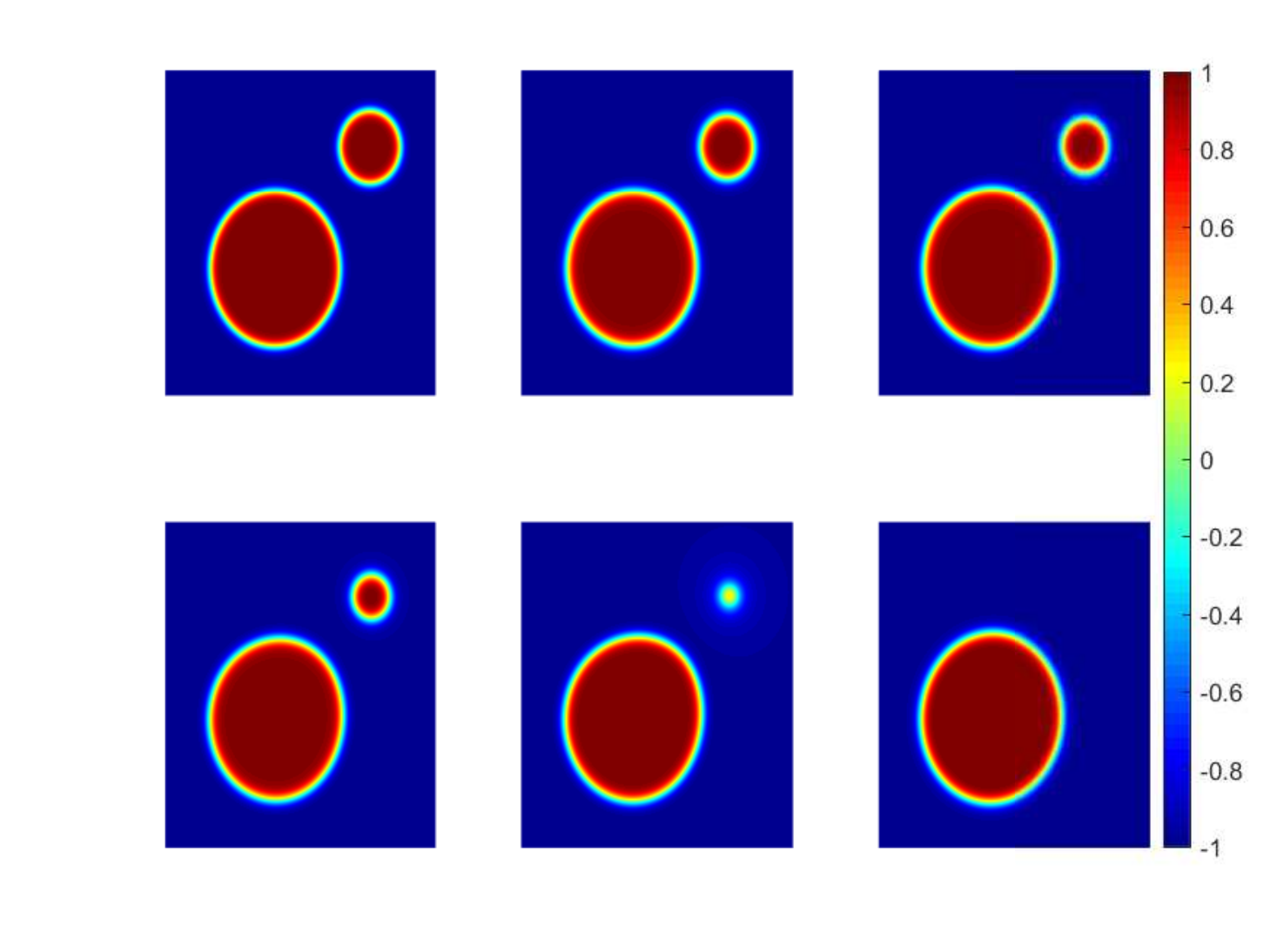}
\caption{Example \ref{Exp5.2.2}.   Snapshots of the numerical solutions at $t = 0, 2, 5, 7, 9$ and  $20$ derived by using {\tt SAV-GL(5)} with $\tau = 0.1$.}
\label{fig5.2.2}
\end{figure}

\begin{figure}
\centerline{\includegraphics[width=12cm,height=7.5cm]{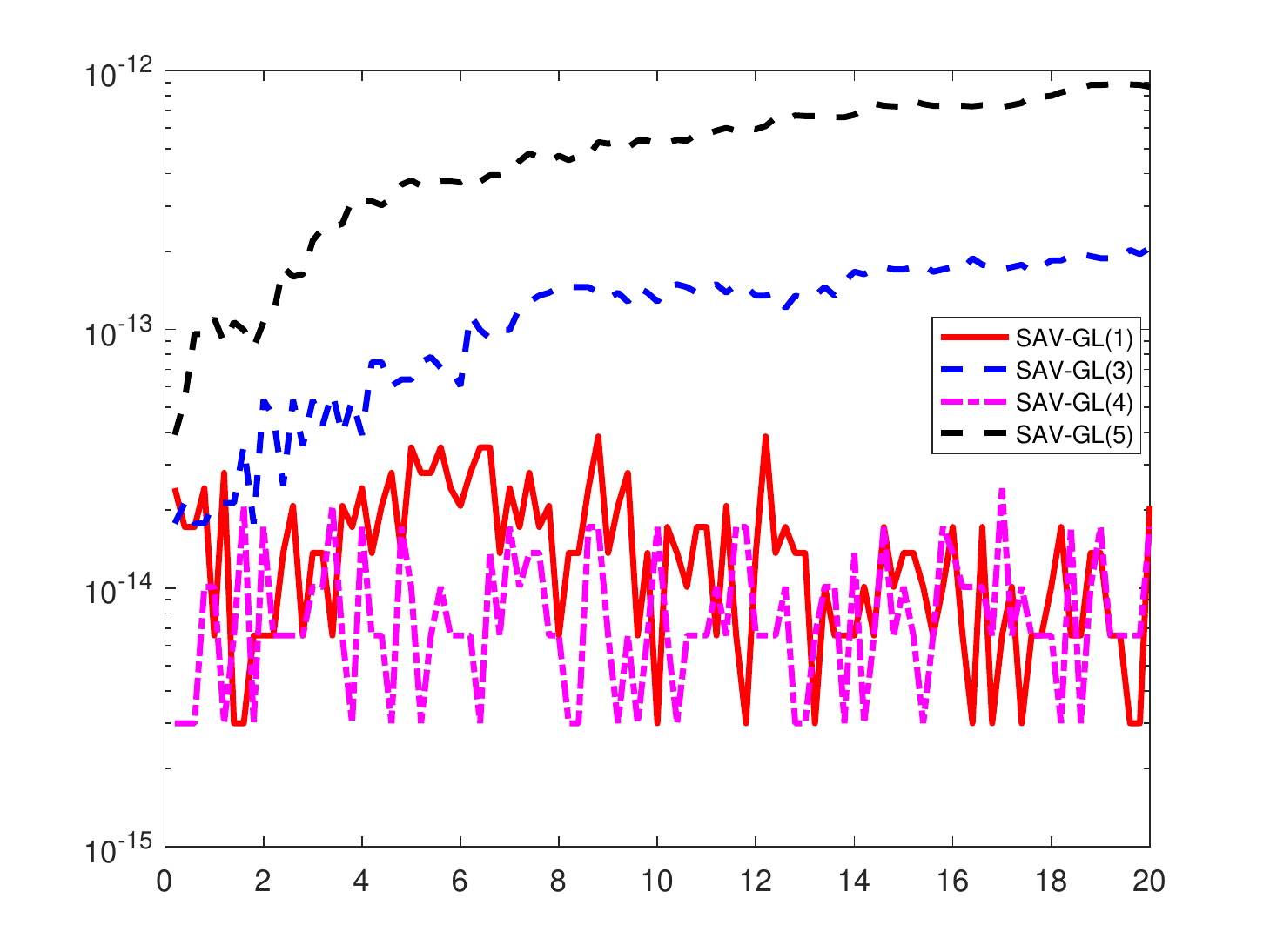}}
\caption{Example \ref{Exp5.2.2}.  {Total mass differences}  of {\tt SAV-GL(1)}, {\tt SAV-GL(3)}, {\tt SAV-GL(4)} and {\tt SAV-GL(5)} with $\tau = 0.1$.}
\label{fig5.2.3}
\end{figure}

\begin{figure}
\begin{minipage}{0.48\linewidth}
  \centerline{\includegraphics[width=7.5cm]{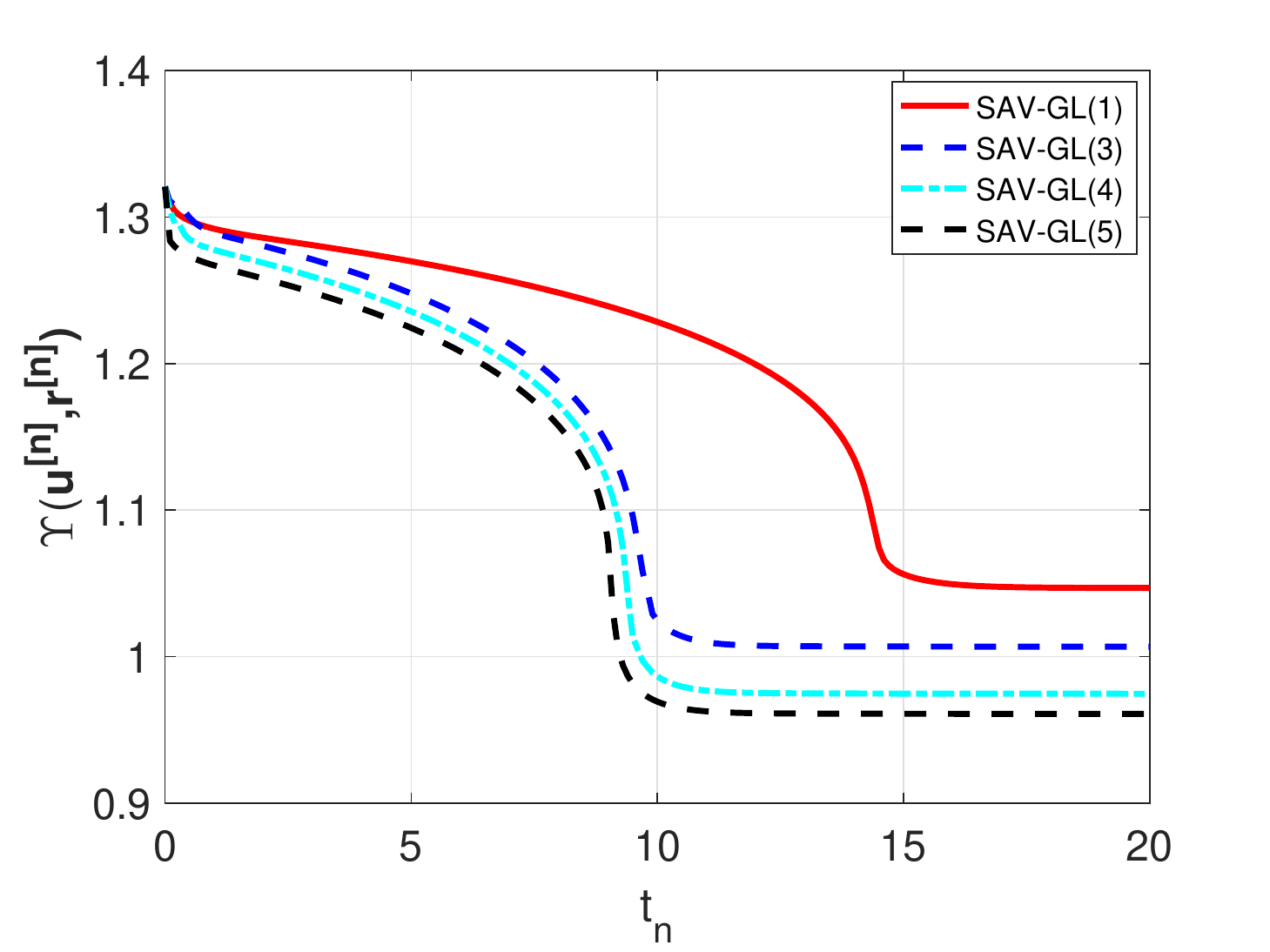}}
\end{minipage}
\hfill
\begin{minipage}{0.48\linewidth}
  \centerline{\includegraphics[width=7.5cm]{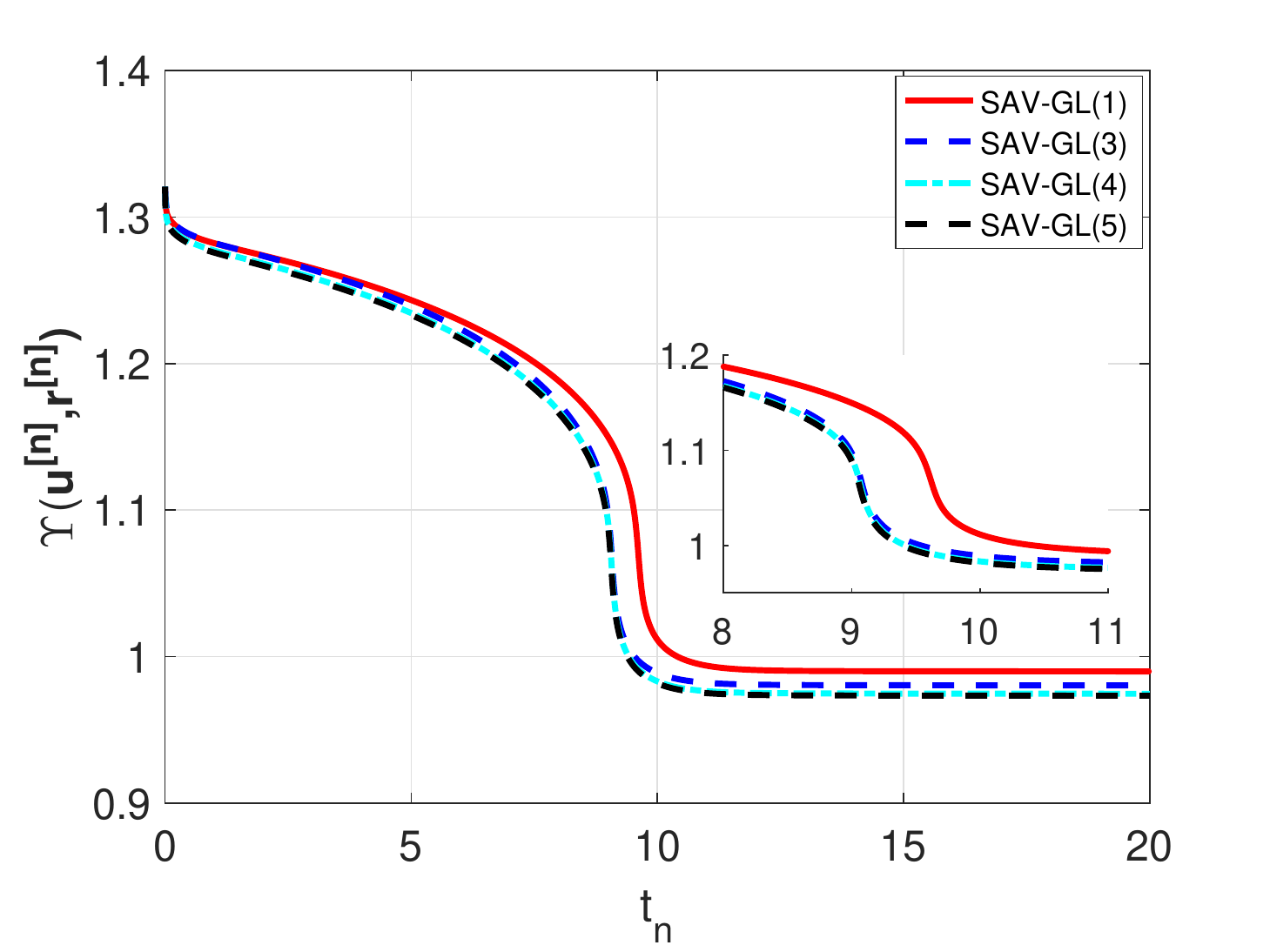}}
\end{minipage}
\caption{Example \ref{Exp5.2.2}.
Discrete energies of {\tt SAV-GL(1)}, {\tt SAV-GL(3)}, {\tt SAV-GL(4)} and {\tt SAV-GL(5)} with $\tau = 0.1$ (Left) and $0.01$ (Right).}
\label{fig5.2.4}
\end{figure}

\subsection{Phase field crystal model}   \label{sec:5.2}

The phase field crystal (PFC) model is a sixth-order nonlinear PDE
\begin{align}   \label{5.5}
\frac{\partial u}{\partial t} = \Delta\mu,~~~ \mu = u^3 - \epsilon_1 u^2 + (1- \epsilon_2) u + 2\Delta u + \Delta^2 u,
\end{align}
which can be derived from the $H^{-1}$ gradient flow of the following free energy
\begin{align*}
\mathcal{F}(u) = \int_{\Omega} \left[ \frac{1}{4}u^4 - \frac{\epsilon_1}{3} u^3 + \frac{1-\epsilon_2}{2} u^2 - |\nabla u|^2 + \frac{1}{2}(\Delta u)^2\right] dx,
\end{align*}
where $\epsilon_1$ and $\epsilon_2$  are two non-negative constants.
This model can be used to describe many crystal phenomena such as edge dislocations \cite{Berry06}, deformation and plasticity in nanocrystalline material \cite{Stefanovic09}, fcc ordering \cite{WuK10}, epitaxial growth and zone refinement \cite{Elder02}.  When $\epsilon_1 = 0$, \eqref{5.5} becomes the classical PFC equation.

In order to apply the SAV-GL schemes for the PFC equation successfully,  the operators $\mathcal{L}$, $\mathcal{G}$ and the energy $\mathcal{F}_1(u)$  are chosen as
\[  \mathcal{L} = \alpha \Delta^2 + \beta, ~~ \mathcal{G} = \Delta,~~ \mathcal{F}_1(u) = \int_{\Omega} \left[ \frac{1}{4}u^4 - \frac{\epsilon_1}{3} u^3+ \frac{1\!-\!\epsilon_2\!-\!\beta}{2} u^2 - |\nabla u|^2 + \frac{1\!-\!\alpha}{2}(\Delta u)^2\right] dx,  \]
where $0<\alpha<1$ and $\beta\ge 0$ are two given parameters. It can be verified  that  $\mathcal{F}_1(u)$ is bounded from below, since
\begin{align*}
\mathcal{F}_1(u) =\;& \int_{\Omega} \left[ \frac{1}{4}u^4 - \frac{\epsilon_1}{3} u^3 + \frac{1\!-\!\epsilon_2\!-\!\beta}{2} u^2 + u\Delta u + \frac{1\!-\!\alpha}{2}(\Delta u)^2\right] dx   \\[2 \jot]
\ge \;& \int_{\Omega} \left[ \frac{1}{4}u^4 - \frac{\epsilon_1}{3} u^3 + \frac{1\!-\!\epsilon_2\!-\!\beta}{2} u^2  - \frac{1}{2(1\!-\!\alpha)} u^2\right] dx,
\end{align*}
where  the inequality $ab\ge -\frac{1}{2\varepsilon}a^2 - \frac{\varepsilon}{2}b^2$, $\varepsilon>0$ is used.
Let us apply {\tt SAV-GL(1)}$\sim${\tt SAV-GL(6)}
 to  the PFC model \eqref{5.5} in order to validate the energy stability and accuracy.

\begin{example} \label{Exp5.3.1}

This example checks the accuracy of {\tt SAV-GL(1)}$\sim${\tt SAV-GL(6)}
for the PFC model \eqref{5.5}. The parameters are taken as $\epsilon_1 = 0$, $\epsilon_2 = 0.5$, $\alpha = 0.99$ and $\beta = 4$, the initial data are $u(x,y,0) = 0.4\sin(x)\cos(y)$, and {\tt SAV-GL(6)} with $\tau = 5\times10^{-5}$ is used to generate the reference solution for computing the $L^2$ errors.
Table \ref{tab:5.3} shows the $L^2$ errors of {\tt SAV-GL(1)}$\sim${\tt SAV-GL(6)} at $t=0.1$ and corresponding convergence rates  with different time stepsizes.
It is seen that  the numerical accuracies of {\tt SAV-GL(1)}$\sim${\tt SAV-GL(4)} are consistent with the theoretical, and {\tt SAV-GL(5)} (resp. {\tt SAV-GL(6)}) with the number of extrapolation points $\nu = 3$ (resp. $\nu=4$) reaches the third-order (resp. fourth-order) accuracy for the PFC model \eqref{5.5}, which validates the statement in Remark \ref{remark4.6}.

\end{example}

\begin{table}[!htbp]
\begin{center}
\caption{Example \ref{Exp5.3.1}. $L^2$ errors of {\tt SAV-GL(1)}$\sim${\tt SAV-GL(6)} at $t = 0.1$ and corresponding convergence rates. }
\begin{tabular*}{\textwidth}{@{\extracolsep{\fill}}lccccccccccc}
\toprule[1pt]
{}&\multicolumn{2}{c}{{\tt SAV-GL(1)}}&{} & \multicolumn{2}{c}{{\tt SAV-GL(2)}} & {} & \multicolumn{2}{c}{{\tt SAV-GL(3)}} & \\
\cmidrule[0.5pt]{2-3}\cmidrule[0.5pt]{5-6} \cmidrule[0.5pt]{8-9}
{$K$} & Errors & Orders &  {} & Errors & Orders  &  {} & Errors & Orders   \\
\midrule[0.5pt]
$240$        &4.3324e-04     &--          &{}      &3.2234-06     &--       &{}      &4.7791e-06     &--       \\
$280$       &3.7131e-04     &1.0007     &{}      &2.3687e-06     &1.9987   &{}      &3.5115e-06     &1.9995     \\
$320$       &3.2487e-04     &1.0006      &{}      &1.8138e-06     &1.9989  &{}      &2.6887e-06     &1.9995      \\
$360$       &2.8875e-04     &1.0005      &{}      &1.4333e-06     &1.9990  &{}      &2.1245e-06     &1.9995      \\
$400$       &2.5987e-04     &1.0005      &{}      &1.1610e-06     &1.9991   &{}      &1.7209e-06     &1.9995     \\
\toprule[1pt]
{}&\multicolumn{2}{c}{{\tt SAV-GL(4)}}&{} & \multicolumn{2}{c}{{\tt SAV-GL(5)}} & {} & \multicolumn{2}{c}{{\tt SAV-GL(6)}} & \\
\cmidrule[0.5pt]{2-3}\cmidrule[0.5pt]{5-6} \cmidrule[0.5pt]{8-9}
{$K$} & Errors & Orders &  {} & Errors & Orders  &  {} & Errors & Orders   \\
\midrule[0.5pt]
$240$        &8.8351e-07    &--          &{}      &1.1427e-09     &--       &{}      &4.4513e-10     &--       \\
$280$       &6.4931e-07    &1.9980      &{}      &7.1376e-10     &3.0528   &{}      &2.4953e-10     &3.7546    \\
$320$       &4.9724e-07     &1.9982      &{}      &4.7372e-10    &3.0700   &{}      &1.4818e-10     &3.9028    \\
$360$       &3.9296e-07     &1.9984     &{}      &3.2929e-10     &3.0877   &{}      &9.3183e-11    &3.9384     \\
$400$       &3.1835e-07     &1.9985     &{}      &2.3740e-10    &3.1055
&{}      &6.1172e-11     &3.9947     \\
\bottomrule[1pt]
\end{tabular*}        \label{tab:5.3}
\end{center}
\end{table}

\begin{example} \label{Exp5.3.2}
This example is used to simulate the polycrystal growth in a supercool liquid by solving \eqref{5.5}. In this case, $ \epsilon_1 = 0$, $\epsilon_2 = 0.25 $, $\alpha = 0.8$ and $\beta = 0$, the spatial stepsize $h$ is $1$ for  the domain $\Omega = [0,400]\times[0,400]$, and the time stepsize $\tau$ is taken as $0.1$ or $0.01$. The initial value $u_0$ is fixed to be a constant value $\phi_0 =0.285$  firstly and then modified by setting three crystallites in three small square patches of the domain, where the centers of three crystallites are located at $(150,150), (200,250)$ and $(250,150)$, respectively, the length of each path is $40$, and the three crystallites are defined by the following expression (see e.g. \cite{YangX17e,LiuZ21})
\[ u(x_l,y_l) = \phi_0  + B\left[\cos\left(\frac{\vartheta}{\sqrt{3}}y_l\right) \cos\left(\vartheta x_l\right) - \frac{1}{2} \cos\left(\frac{2\vartheta}{\sqrt{3}}y_l\right) \right], \]
with $ B = 0.446$, $\vartheta = 0.66$, and local
coordinates $x_l, y_l$ given by
\[ x_l(x,y) = x\sin(\theta) + y\cos(\theta) ,~~~y_l(x,y) = -x\cos(\theta) + y\sin(\theta), ~~\mbox{for}~~\theta = \frac{\pi}{4}, 0, -\frac{\pi}{4}. \]

Figure \ref{fig5.3.1} gives the cut lines and contour lines of the numerical solutions at $t=150$ derived by {\tt SAV-GL(1)}, {\tt SAV-GL(3)}, {\tt SAV-GL(4)} and {\tt SAV-GL(5)} with $\tau=0.1$, from which one can see that the numerical solutions computed by those schemes are quite similar.
 Figure \ref{fig5.3.3} shows the differences between the total masses at $t_n$ and $t_0$ obtained by {\tt SAV-GL(1)}, {\tt SAV-GL(3)}, {\tt SAV-GL(4)} and {\tt SAV-GL(5)} with $\tau = 0.1$, which checks the mass conservation numerically.
 Figure \ref{fig5.3.2} gives the snapshots of the numerical solutions at $t = 0, 70, 150, 300, 450$ and  $1000$ computed by using {\tt SAV-GL(5)} with $\tau = 0.1$. It can be seen that the three different crystal grains grow and become large enough to form grain boundaries finally.
Figure \ref{fig5.3.4} displays the discrete energy curves of {\tt SAV-GL(1)}, {\tt SAV-GL(3)} and {\tt SAV-GL(5)} with $\tau = 0.1$ and $0.01$, which indicates that those schemes are energy stable in solving the PFC model \eqref{5.5}.  The discrete energy curves of those schemes have some visible differences with $\tau=0.1$, but the differences are almost indistinguishable for $\tau=0.01$.  It means that  {\tt SAV-GL(5)} may have some advantages to get the accurate steady solution of the PFC model \eqref{5.5} when  a large time stepsize is taken.

\end{example}

\begin{figure}
\begin{minipage}{0.48\linewidth}
  \centerline{\includegraphics[width=7.5cm]{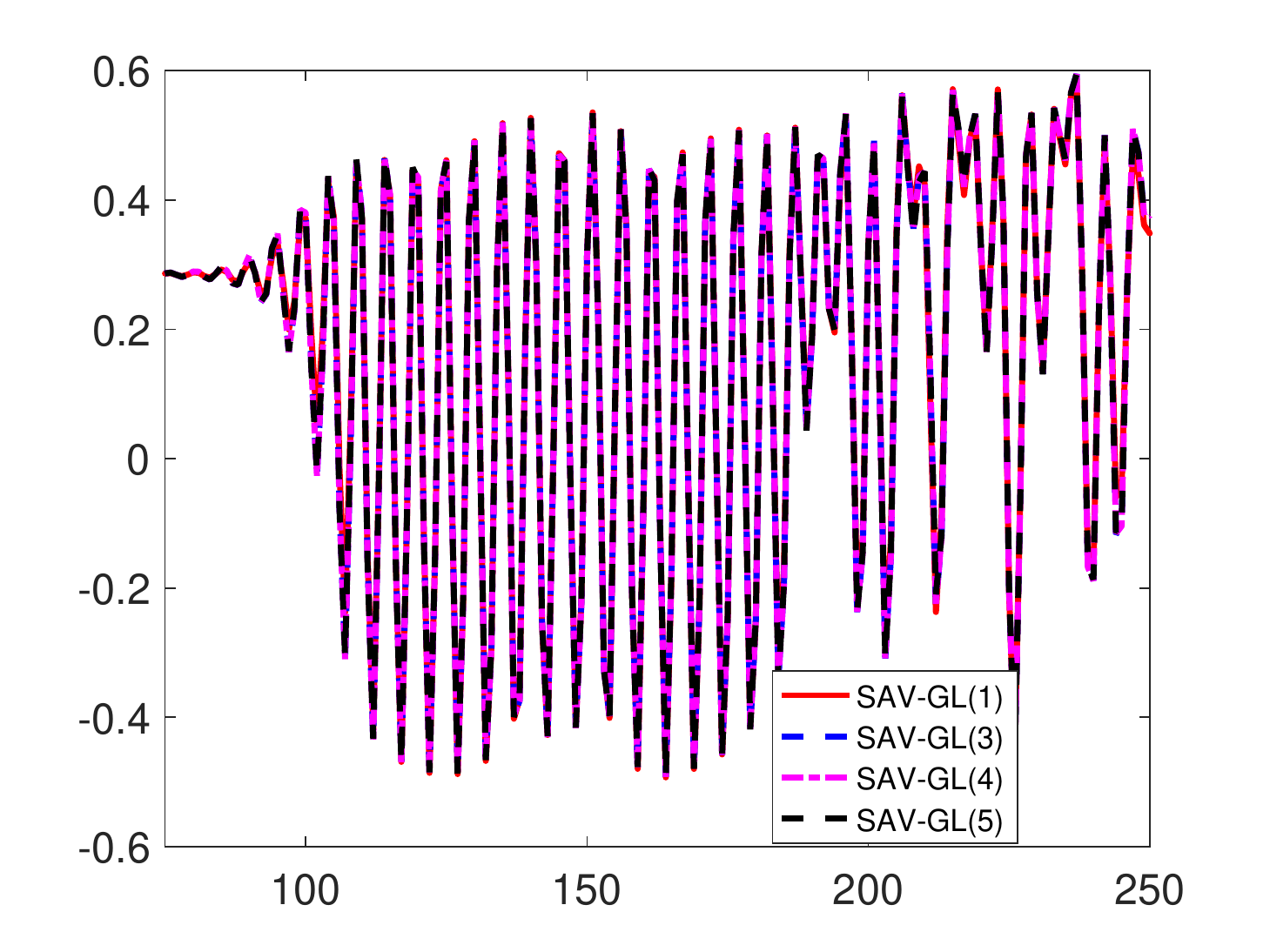}}
  \centerline{ (a)}
\end{minipage}
\hfill
\begin{minipage}{0.48\linewidth}
  \centerline{\includegraphics[width=7.5cm]{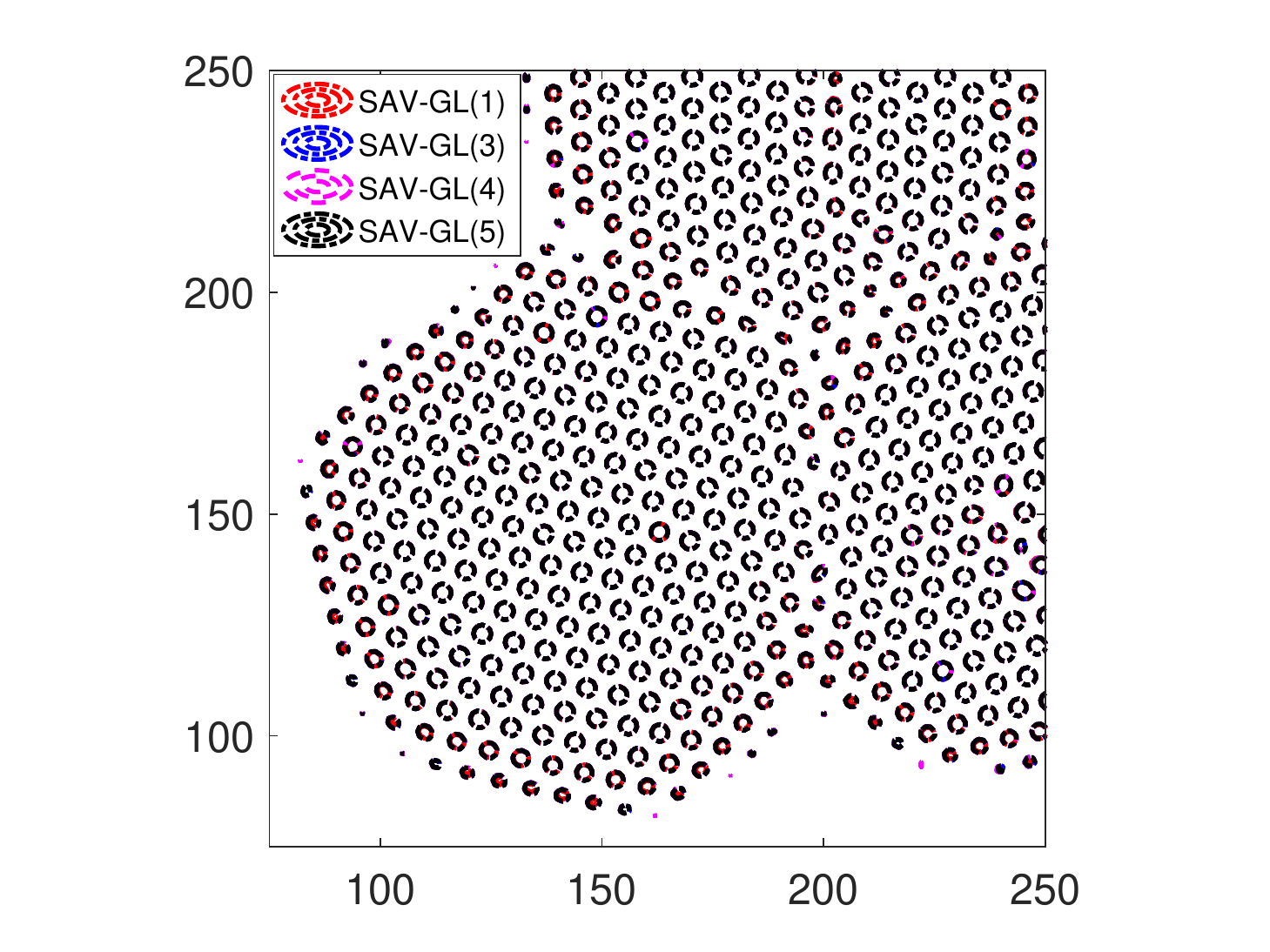}}
  \centerline{ (b)}
\end{minipage}
\caption{Example \ref{Exp5.3.2}. Left:
cut lines of the numerical solutions at $t=150$ along $y=x~(x\in[75,250])$; 
right: contour lines of $u=0.1$ at $t=150$.}
\label{fig5.3.1}
\end{figure}

\begin{figure}
\includegraphics[width=16cm,height=9cm]{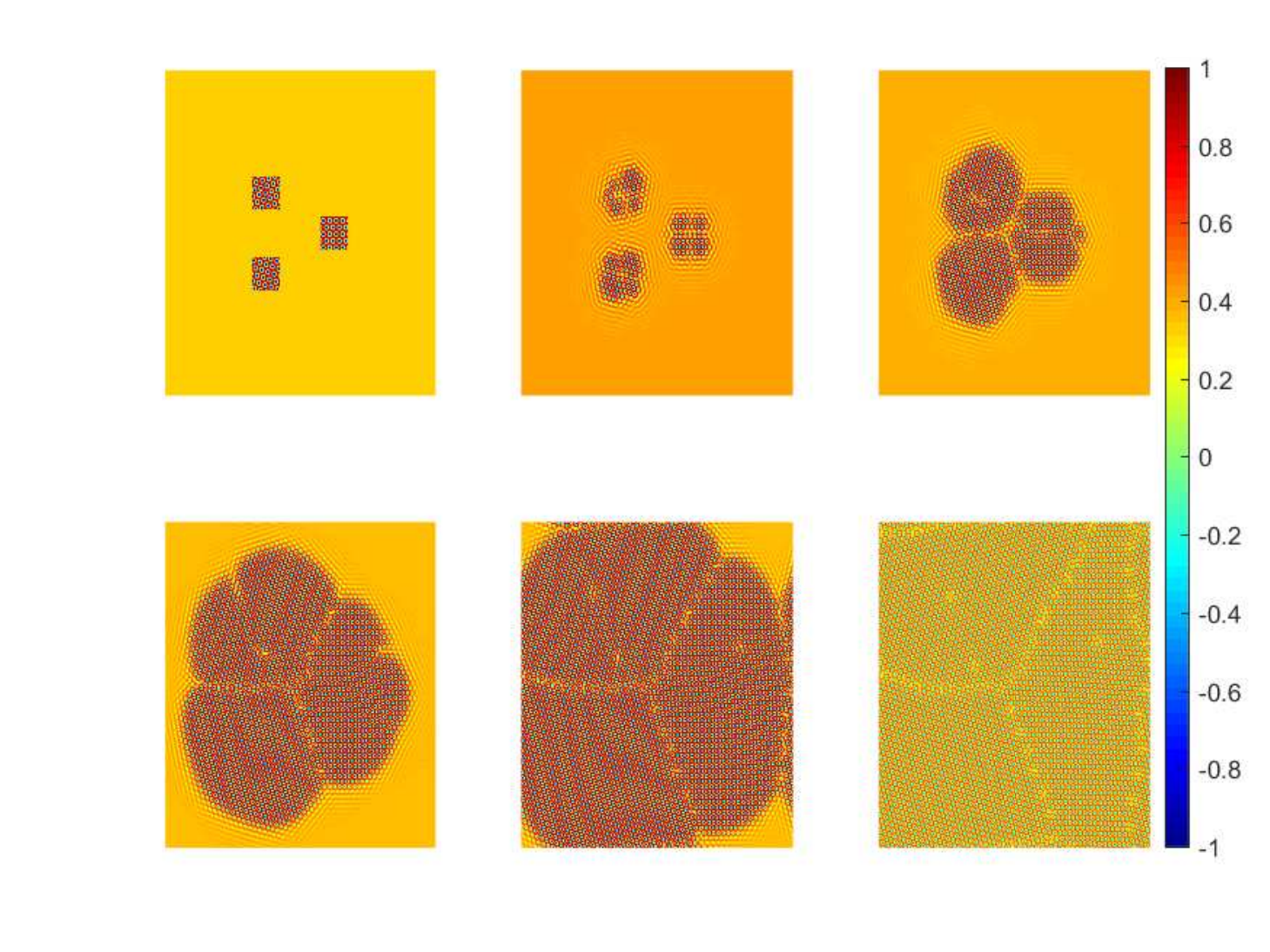}
\caption{Example \ref{Exp5.3.2}.   Snapshots of the numerical solutions at $t = 0, 70, 150, 300, 450$ and  $1000$ computed by using {\tt SAV-GL(5)} with $\tau = 0.1$.}
\label{fig5.3.2}
\end{figure}

\begin{figure}
\centerline{\includegraphics[width=12cm,height=7.5cm]{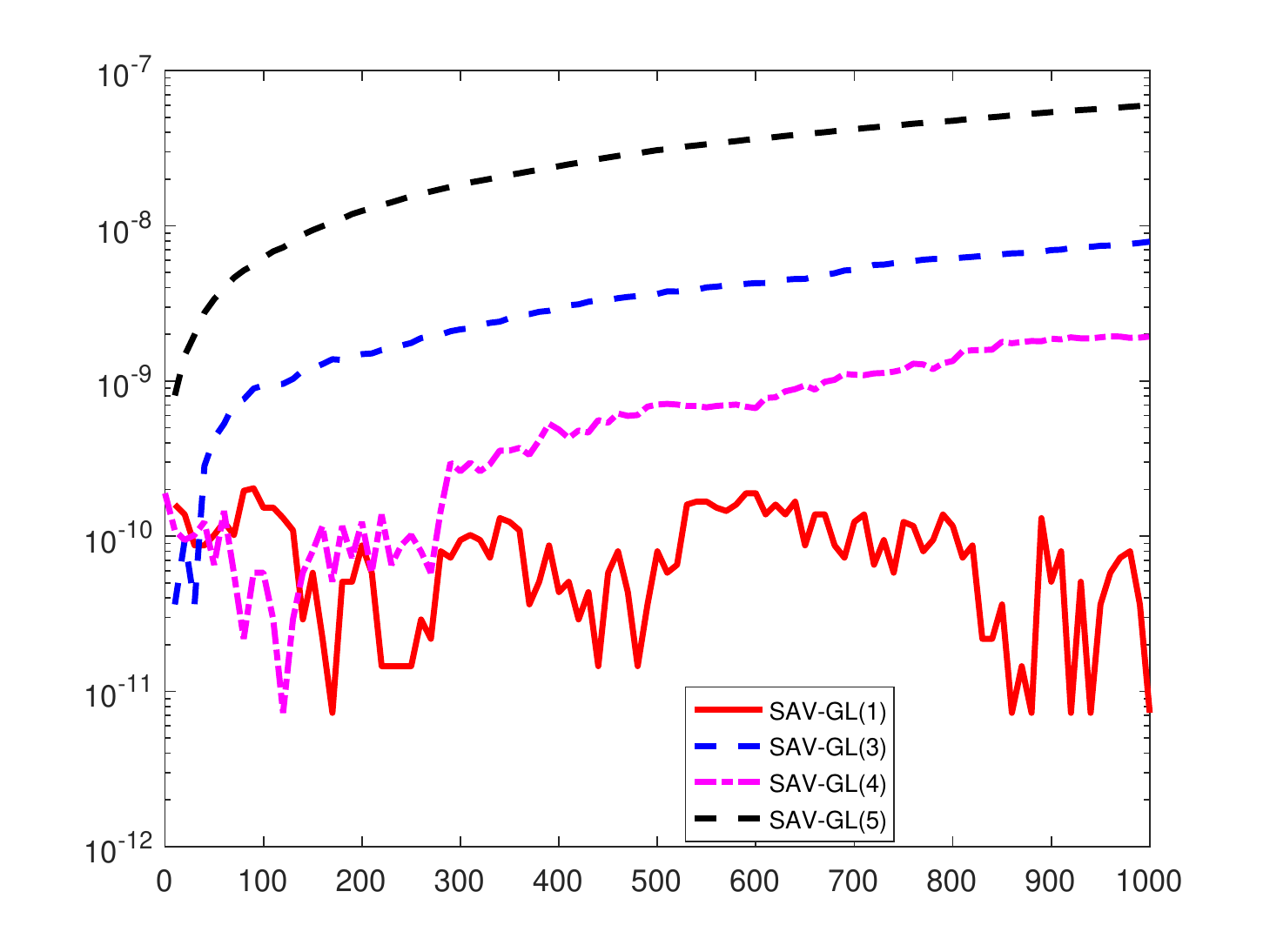}}
\caption{Example \ref{Exp5.3.2}. {Total mass differences}  of {\tt SAV-GL(1)}, {\tt SAV-GL(3)}, {\tt SAV-GL(4)} and {\tt SAV-GL(5)} with $\tau = 0.1$.}
\label{fig5.3.3}
\end{figure}

\begin{figure}
\begin{minipage}{0.48\linewidth}
  \centerline{\includegraphics[width=7.5cm]{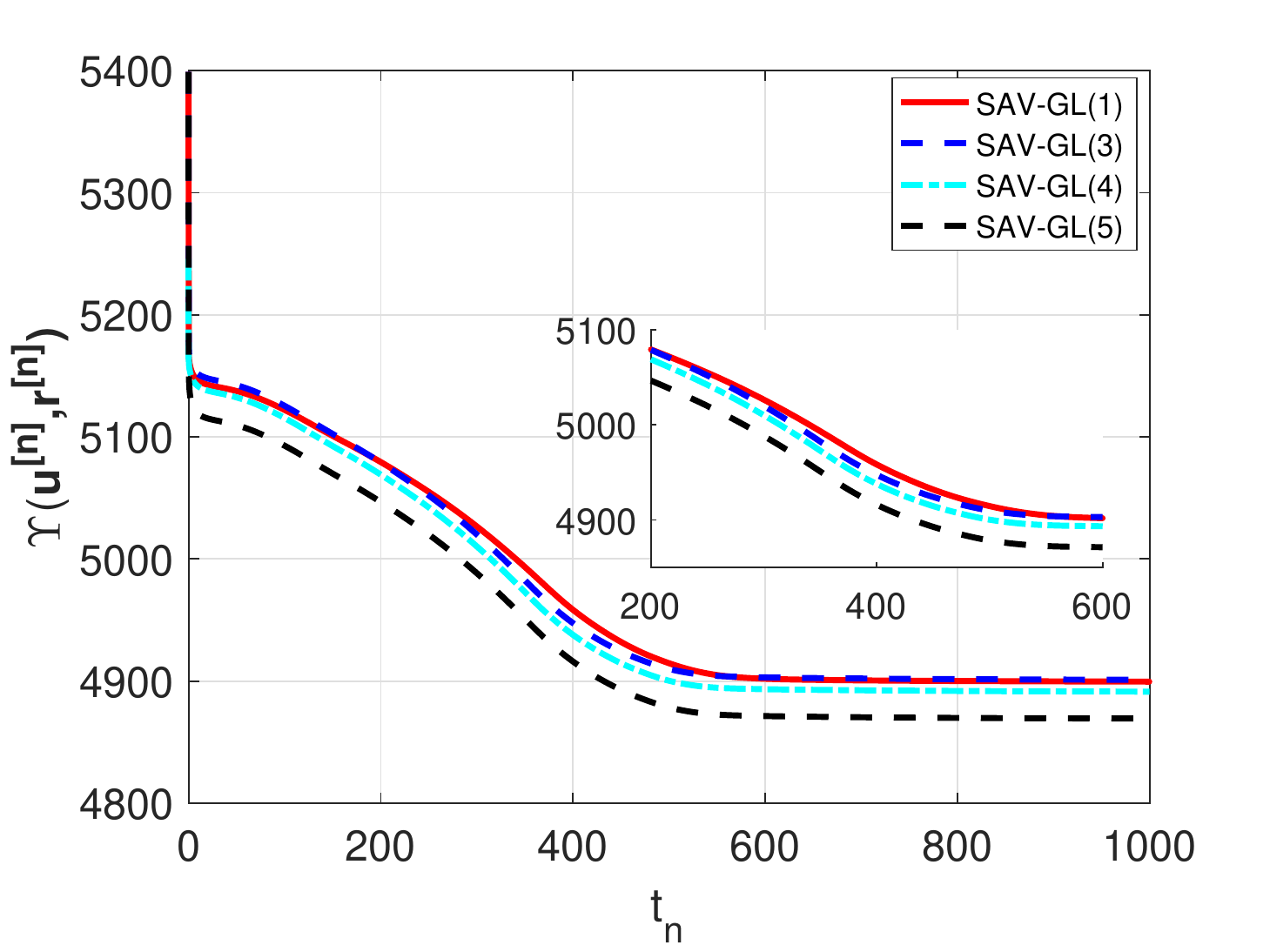}}
\end{minipage}
\hfill
\begin{minipage}{0.48\linewidth}
  \centerline{\includegraphics[width=7.5cm]{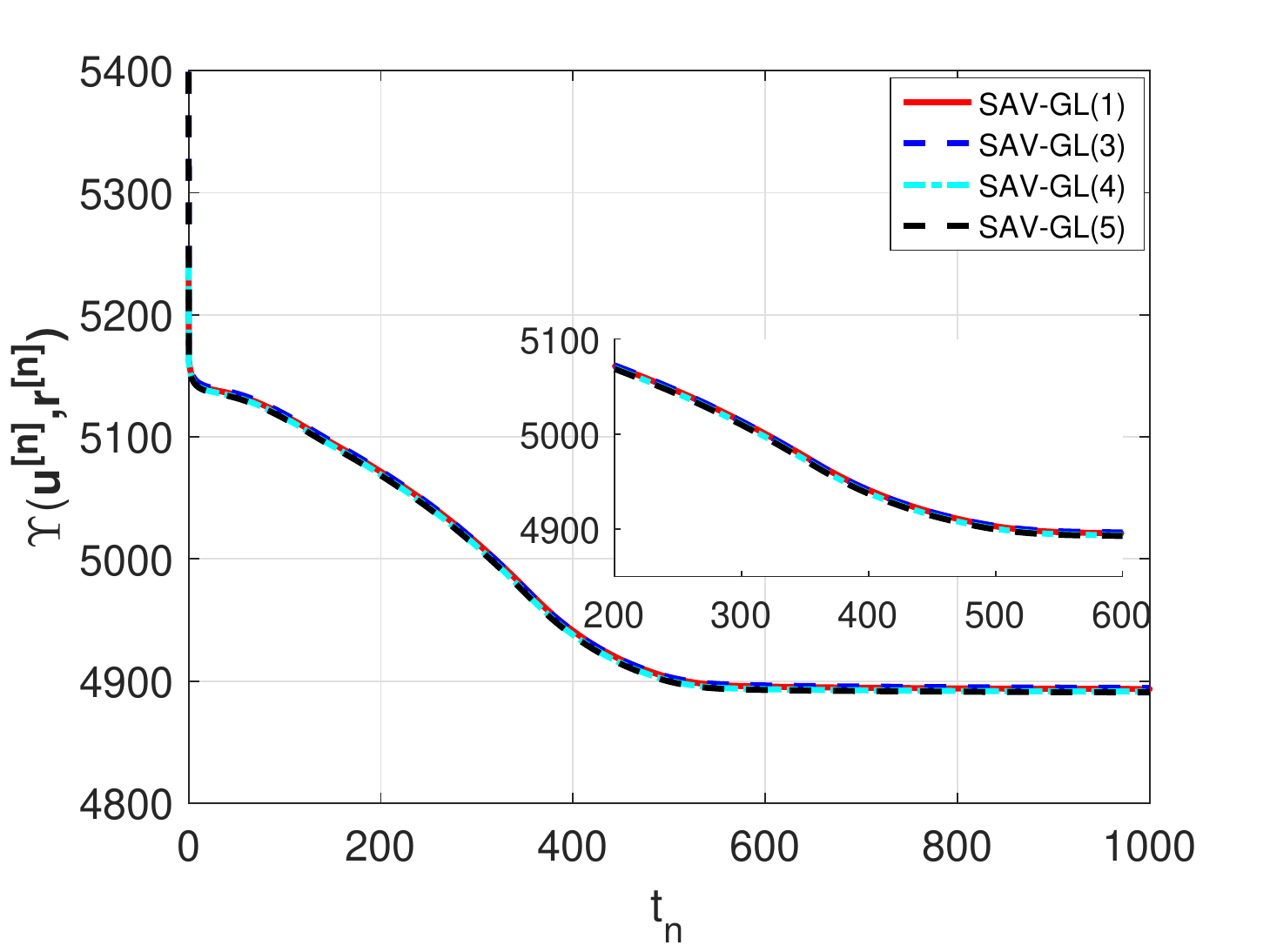}}
\end{minipage}
\caption{Example \ref{Exp5.3.2}.
Discrete energies of {\tt SAV-GL(1)}, {\tt SAV-GL(3)}, {\tt SAV-GL(4)} and {\tt SAV-GL(5)} with $\tau = 0.1$ (Left) and $0.01$ (Right).}
\label{fig5.3.4}
\end{figure}

\section{Conclusions}   \label{sec:7}

This paper proposed a general class of linear and unconditionally energy stable numerical schemes for the gradient flows by using the SAV  and the general linear time discretizations (GLTDs).
Those SAV-GL schemes could reach arbitrarily high-order accuracy in time, and only a coupled system of linear equations was solved at each time step since the nonlinear terms of the reformulated SAV equations were  linearized based on extrapolation.
Importantly, the resulting SAV-GL schemes contained most of the time integration schemes for the gradient flows in literature and many new schemes.
The semi-discrete-in-time  SAV-GL schemes were proved to be unconditionally energy stable
when the GLTD was algebraically stable,
and to be  convergent with the order of $\min\{\hat{q},\nu\}$
under the diagonal stability and some suitable regularity and    accurate  starting values,
where $\hat{q}$ was the generalized stage order of the GLTD and $\nu$ denoted the number of the extrapolation points in time.
As two typical examples, the so-called one-leg and multistep Runge-Kutta (MRK) time integration schemes were considered for the gradient flows and their energy stabilities and error estimates were discussed  separately.
Because the proof of the energy stability   was variational and the energy stability was available for the boundary conditions which made all boundary terms  disappear when the integration by parts was performed,
the above energy stability could be straightforwardly extended to the fully discrete SAV-GL schemes with the Galerkin finite element  or  the spectral methods  or  the finite
difference methods, satisfying the summation by parts for the spatial discretization.

In order to demonstrate numerically the energy stability and accuracy of the SAV-GL schemes,
the fully discrete SAV-GL schemes with the Fourier spectral spatial discretization were presented
for three gradient flows equations (the Allen-Cahn, Cahn-Hilliard and phase field crystal models) with periodic boundary conditions.
 Our numerical experiments well demonstrated the theoretical results of {\tt SAV-GL(1)}$\sim${\tt SAV-GL(6)}
and also checked the discrete maximum principle  for the Allen-Cahn model and the mass conservation for the Cahn-Hilliard and  phase field crystal models. 
They also showed that
 the high-order SAV-GL schemes such as {\tt SAV-GL(5)} might have some obvious advantage
to derive the accurate steady solutions for the three gradient flow equations when taking a large time stepsize,
and the SAV-GL scheme \eqref{3.2.13}-\eqref{3.2.15} with the integer $s\ge 2$  and the number of extrapolation points $\nu = s+1$ was convergent with the order of $s+1$.

Besides the one-leg and MRK time integration schemes considered in this paper, other time discretizations,  such as the diagonally implicit multistage integrations (see e.g. \cite{Izzo14,Butcher93}) and the general class of two-step Runge-Kutta methods (see e.g. \cite{Dambrosio12,Jackiewicz95}) etc.,
can be reformulated as the form of the GLTDs \eqref{2.1.2}. Those methods may also possess good energy stability and accuracy for the gradient flows.
In future, we will further address several interesting topics on numerical schemes for
 the gradient flows:
exploring higher-order one-leg schemes with the aid of the novel SAV approach  \cite{HuangF20},
estimating the errors of the fully discrete SAV-GL schemes,
and combining the present SAV-GL schemes with the adaptive moving mesh method \cite{Zhang-Tang2007} for the mixture of two incompressible fluids etc.

\section*{Acknowledgments}
The authors were partially supported by 
the National Key R\&D Program of China (Project Number 2020YFA0712000) and
the National Natural Science Foundation of China (No. 12126302 \& 12171227).

\end{document}